\documentclass{article}
\usepackage[utf8]{inputenc}

\title{Crushing Surfaces of Positive Genus}
\date{\today}

\author{
\bgroup
\setlength{\tabcolsep}{12pt}
\begin{tabular}{cc}
    \shortstack{
    Benjamin A. Burton\\
    The University of Queensland\\
    \email{bab}{maths.uq.edu.au}
    }
    &
    \shortstack{
    Thiago de Paiva\\
    Monash University\\
    \email{thiago.depaivasouza}{monash.edu}
    }
    \\[12pt]
    \shortstack{
    Alexander He\\
    The University of Queensland\\
    \email{a.he}{uqconnect.edu.au}
    }
    &
    \shortstack{Connie On Yu Hui\\
    Monash University\\
    \email{onyu.hui}{monash.edu}
    }
\end{tabular}
\egroup
}

\usepackage{amsmath, amssymb, amsthm, caption, subcaption, xcolor}
\usepackage{thmtools}
\usepackage{thm-restate}

\usepackage[shortlabels]{enumitem}
\setlist[itemize]{leftmargin=*, noitemsep}
\setlist[enumerate]{leftmargin=*, noitemsep}
\setlist[description]{leftmargin=*, labelwidth=*, noitemsep}

\usepackage[a4paper, margin=2cm]{geometry}

\usepackage{graphicx}
\graphicspath{{CrushFigures/}}

\usepackage{tikz}
\usetikzlibrary{cd} 
\usetikzlibrary{decorations.pathmorphing} 
\usetikzlibrary{arrows.meta}
\usetikzlibrary{calc}

\tikzcdset{every label/.append style = {font = \normalsize}}

\usepackage[backref=page]{hyperref}

\theoremstyle{plain}
\newtheorem{theorem}{Theorem}
\newtheorem{lemma}[theorem]{Lemma}

\newtheorem{corollary}[theorem]{Corollary}

\newtheorem{observation}[theorem]{Observation}

\newtheorem{atomic}{Claim}

\newtheorem{subatomic}{Claim}
\numberwithin{subatomic}{atomic}

\declaretheoremstyle[
	headfont=\normalfont\bfseries,
	numbered=yes,
	bodyfont=\normalfont,
	qed={$\blacksquare$},
	spaceabove=1em,
	spacebelow=1em,
]{qeddef}

\declaretheorem[
	style=qeddef,
	title=Definitions,
	sibling=theorem,
]{definitions}
\declaretheorem[
	style=qeddef,
	title=Notation,
	sibling=atomic,
]{atomicnotn}

\theoremstyle{definition}
\newtheorem*{remark*}{Remark}

\newcommand{\Regina}{\texttt{\textup{Regina}}}
\newcommand{\ConnSum}{\mathbin{\#}}
\DeclareMathOperator{\Wt}{wt}

\makeatletter
\renewenvironment{proof}[1][\proofname] {\par\pushQED{\qed}\normalfont\topsep6\p@\@plus6\p@\relax\trivlist\item[\hskip\labelsep\itshape\bfseries#1\@addpunct{.}]\ignorespaces}{\popQED\endtrivlist\@endpefalse}
\makeatother

\newcommand{\email}[2]{\href{mailto:#1@#2}{\textsf{#1\hspace{1pt}$@$\hspace{1pt}#2}}}

\DeclareMathOperator{\Int}{int}

\begin{document}

\maketitle

\begin{abstract}
The operation of crushing a normal surface has proven to be a powerful tool in computational $3$-manifold topology,
with applications both to triangulation complexity and to algorithms.
The main difficulty with crushing is that it can drastically change the topology of a triangulation,
so applications to date have been limited to relatively simple surfaces:
$2$-spheres, discs, annuli, and closed boundary-parallel surfaces.
We give the first detailed analysis of the topological effects of crushing closed essential surfaces of positive genus.
To showcase the utility of this new analysis, we use it to prove some results
about how triangulation complexity interacts with JSJ decompositions and satellite knots;
although similar applications can also be obtained using techniques of Matveev,
our approach has the advantage that it avoids the machinery of almost simple spines and handle decompositions.
\end{abstract}


\section{Introduction}\label{sec:intro}

The idea of crushing a normal surface was first developed by Jaco and Rubinstein~\cite{JacoRubinstein2003}
as part of a broader program of giving a theory of ``efficient'' $3$-manifold triangulations.
This led to new insights on minimal triangulations~\cite{JacoRubinstein2003}, and has also been the key to developing ``efficient''
(in various senses of the word, depending on the particular application) algorithms to solve a number of fundamental problems in
low-dimensional topology \cite{Burton2013Regina,Burton2014,BCT2013,BurtonHe2023SoCG,BurtonOzlen12,BurtonTillmann2018,Fowler2003,INT2022}.

The key obstacle in developing new applications of crushing is that this operation can drastically alter the topology of a triangulation.
This difficulty was initially compounded by the complicated formulation of crushing that was originally given by Jaco and Rubinstein;
although Jaco and Rubinstein were able to give a number of applications,
they required intricate arguments about the topological effects of crushing $2$-spheres, discs and closed boundary-parallel surfaces~\cite{JacoRubinstein2003}.
More recent applications rely on simpler formulations of crushing that are easier to understand and use:
\begin{itemize}
\item Following unpublished ideas of Casson, Fowler~\cite{Fowler2003} used
the language of special spines to understand the effect of crushing $2$-spheres.
\item The first author introduced a way to break crushing down into a sequence of simple atomic moves,
and used this atomic approach to describe the topological effects of crushing $2$-spheres and discs~\cite{Burton2014};
this has proven to be extremely useful for turning crushing into an accessible algorithmic tool
for working with $3$-manifolds \cite{Burton2013Regina,Burton2014,BCT2013,BurtonHe2023SoCG,BurtonOzlen12,BurtonTillmann2018}.
This atomic approach has also recently been applied to crushing certain types of properly embedded annuli~\cite{INT2022}.
\end{itemize}

We emphasise that although it is, in principle, possible to crush any normal surface, the applications to date have
only involved $2$-spheres, discs, annuli and closed boundary-parallel surfaces.
Probably the main reason for this is that as the surfaces get more complicated, the topological effects of crushing also appear to get more complicated.
Nevertheless, we demonstrate in this paper that it is possible to push through this challenge by building upon the atomic approach to crushing from~\cite{Burton2014}.

To be precise, we use the atomic approach to understand the topological effects of crushing closed normal surfaces of positive genus;
in particular, we are able to crush essential surfaces, not just boundary-parallel ones.
This work is distributed across two sections of this paper.
First, in Section~\ref{sec:atomicIdeal}, we carefully work through the necessary details to extend the atomic approach to crushing.
Then, in Section~\ref{sec:genus}, we apply the work from Section~\ref{sec:atomicIdeal} to
actually understand the effect of crushing a surface of positive genus.

To state the main theorem from Section~\ref{sec:genus}, we require some notation and terminology
which we now outline (see Section~\ref{sec:genus} for the precise definitions).
Given a normal surface $S$ in a $3$-manifold triangulation $\mathcal{T}$, our goal is to triangulate a submanifold $X$ of $\mathcal{T}$ that is ``cut out'' by the surface $S$.
More precisely, we fix a component of $\mathcal{T}-S$, which we call the \emph{chosen region}, and then take $X$ to be the closure of the chosen region.
After crushing $S$ to obtain a new triangulation $\mathcal{T}'$, each component of $\mathcal{T}-S$ ``falls apart'' to yield some subset of the components of $\mathcal{T}'$.
In particular, the chosen region yields some subset $\mathcal{T}^\ast$ of $\mathcal{T}'$, and our hope is that (a component of) $\mathcal{T}^\ast$ actually gives a triangulation of $X$.

For our purposes, it turns out to be important to ``push'' or ``expand'' $S$ as far into the chosen region as possible, to obtain what we call a \emph{maximal} surface.
We show in Lemma~\ref{lem:maximal} that we can always, without loss of generality, assume that $S$ is maximal.
With this groundwork, together with our analysis from Section~\ref{sec:atomicIdeal}, we are able to prove the following theorem in Section~\ref{sec:genus}:

\begin{restatable}{theorem}{crushPositiveGenus}\label{thm:crushPositiveGenus}
Suppose that $X$ is irreducible, $\partial$-irreducible and anannular, and that it contains no two-sided properly embedded M\"{o}bius bands.
Also, suppose $S$ is maximal.
Then $\mathcal{T}^\ast$ is a valid triangulation such that:
\begin{itemize}[nosep]
\item one of its components is an ideal triangulation of $X$; and
\item every other component is a triangulation of the $3$-sphere.
\end{itemize}
\end{restatable}

In Section~\ref{sec:submanifold}, we apply Theorem~\ref{thm:crushPositiveGenus} to study the \emph{triangulation complexity} $\Delta(\mathcal{M})$:
the minimum number of tetrahedra required to triangulate some particular $3$-manifold $\mathcal{M}$.
In particular, we obtain the following general result as a relatively straightforward consequence of Theorem~\ref{thm:crushPositiveGenus}:

\begin{restatable}{theorem}{submanifold}\label{thm:submanifold}
Let $\mathcal{M}$ be a compact $3$-manifold with no $2$-sphere boundary components.
Suppose $\mathcal{M}$ contains a (possibly disconnected) closed incompressible surface $S$ with
no $2$-sphere components, no projective plane components, and no boundary-parallel components.
Let $\mathcal{R}$ be a component obtained after cutting $\mathcal{M}$ along $S$.
If $\mathcal{R}$ is irreducible, $\partial$-irreducible, anannular, and does not contain any
two-sided properly embedded M\"{o}bius bands, then $\Delta(\mathcal{R}) < \Delta(\mathcal{M})$.
\end{restatable}

We continue in Section~\ref{sec:submanifold} by specialising Theorem~\ref{thm:submanifold}
to the particularly interesting setting where $S$ is a collection of essential tori.
This gives various nice results about how triangulation complexity interacts with JSJ decompositions and satellite knots.

The applications that we obtain in Section~\ref{sec:submanifold} are not entirely new, since they can also be obtained by
combining various pieces of machinery from Matveev's book~\cite{Matveev2007}
(we discuss this in a little more detail in Section~\ref{sec:submanifold}).
Nevertheless, our applications demonstrate that crushing normal surfaces of positive genus has
non-trivial consequences for objects that are of independent interest.
This provides hope that future refinements of our techniques could lead to further applications,
such as new algorithms involving decompositions along surfaces of positive genus.

It is also worth noting that whilst Matveev's techniques use almost simple spines and handle decompositions, our work does not require such machinery;
instead, our analysis of crushing only uses triangulations and cell decompositions.
Some readers might therefore find our approach more accessible than that of Matveev.
Moreover, in contrast to handle decompositions, crushing has the advantage that it is
well-established in software such as \Regina~\cite{Burton2013Regina,Regina};
thus, our approach is probably more amenable for practical algorithmic applications.

\paragraph{Acknowledgements.}
This project is the result of a conversation that began at the \emph{PhD Student Symposium: Graduate Talks in Geometry and Topology Get-Together, or (GT)\textsuperscript{3}}, hosted by MATRIX in July 2022;
we would therefore like to thank the MATRIX Institute and the organisers of the symposium.

AH was supported by an Australian Federal Government Research Training Program Scholarship.

We would like to thank the anonymous reviewer for their many suggestions, which greatly helped to improve the readability of this paper.

\section{Preliminaries}\label{sec:prelims}

The main purpose of this section is to review all the definitions that we will require for our analysis of crushing.

As a convention that we will use throughout this paper, except where we explicitly state otherwise,
all $3$-manifolds will be compact.
We will call a (compact) $3$-manifold \textbf{closed} if it has empty boundary, and \textbf{bounded} if it has non-empty boundary.

Also, whenever we are working with an object $X$ (such as a knot or a surface) embedded in a $3$-manifold $\mathcal{M}$,
we will often refer to ambient isotopies of $X$ in $\mathcal{M}$ simply as isotopies of $X$.
For example, when we speak of isotoping a knot $K$ (embedded in the $3$-sphere $S^3$),
we really mean that we are applying an \emph{ambient} isotopy to the embedding of $K$ in $S^3$.

\subsection{Triangulations and cell decompositions}\label{subsec:tri}

A \textbf{(generalised) triangulation} $\mathcal{T}$ consists of finitely many (abstract) tetrahedra with some or all of their triangular faces
\textbf{glued} together in pairs via affine identifications
(Figure~\ref{fig:gluedTets} illustrates a single such gluing);
denote the number of tetrahedra in $\mathcal{T}$ by $\lvert\mathcal{T}\rvert$.
We allow faces from the same tetrahedron to be glued together, which means that $\mathcal{T}$ need not be a simplicial complex;
indeed, generalised triangulations can usually be made much smaller than topologically equivalent
simplicial complexes, which is often important for computational purposes.

\begin{figure}[htbp]
\centering
	\includegraphics[scale=0.6]{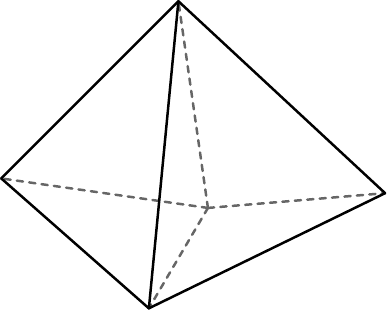}
\caption{Two tetrahedra glued together along a single pair of triangular faces.}
\label{fig:gluedTets}
\end{figure}

In this paper, we also work with cell decompositions, which generalise the triangulations that we just defined.
We build gradually towards a definition of cell decompositions, starting with
an explanation of how we generalise tetrahedra to obtain a larger class of ``building blocks''.

Topologically, we can think of a tetrahedron as a $3$-ball whose boundary $2$-sphere
is divided into triangles by an embedding of the complete graph on four vertices.
To generalise this, consider a topological $3$-ball $\Delta$ with a multigraph $\Gamma$ embedded in $\partial\Delta$.
We call $\Delta$ an \textbf{(abstract) 3-cell} if:
\begin{itemize}
\item $\Gamma$ has no degree one vertices; and
\item the closure of each component of $(\partial\Delta)-\Gamma$ forms
an embedded disc, which we call a \textbf{face} of $\Delta$, whose boundary circle contains two or more vertices of $\Gamma$.
\end{itemize}
Assuming that these conditions are indeed satisfied, we refer to the vertices and edges of $\Gamma$ as
\textbf{vertices} and \textbf{edges}, respectively, of the $3$-cell $\Delta$.
Intuitively, each face of an abstract $3$-cell forms a curvilinear polygon with two or more edges;
indeed, depending on the number of edges, we will often describe $3$-cell faces as bigons, triangles, quadrilaterals, and so on.

There are infinitely many types of $3$-cells.
However, for our purposes, we will only need to deal with a finite number of these;
some examples are shown in Figure~\ref{fig:examples3Cells}.
For details on precisely which types of $3$-cells we need, see Definitions~\ref{defs:inducedCells} and Section~\ref{subsec:atomic}.

\begin{figure}[htbp]
\centering
	\begin{tikzpicture}

	\newcommand{\Size}{0.6}

	\node at (0,0) {
		\includegraphics[scale=\Size]{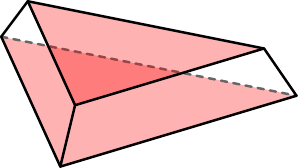}
	};
	\node at (3.2,0) {
		\includegraphics[scale=\Size]{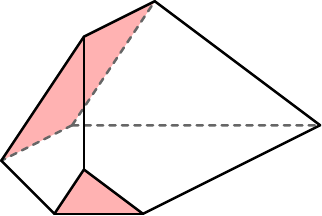}
	};
	\node at (6.5,0) {
		\includegraphics[scale=\Size]{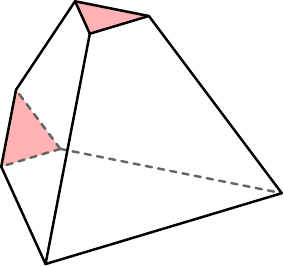}
	};
	\node at (8.45,0) {
		\includegraphics[scale=\Size]{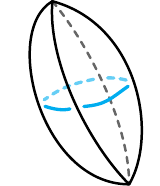}
	};
	\node at (10.55,0) {
		\includegraphics[scale=\Size]{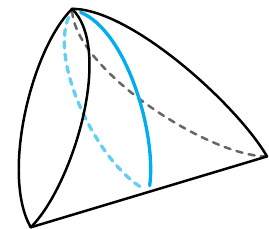}
	};
	\node at (13.25,0) {
		\includegraphics[scale=\Size]{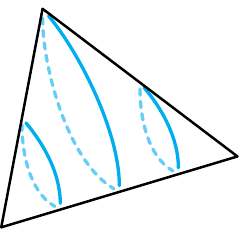}
	};

	\end{tikzpicture}
\caption{Some examples of the non-tetrahedron cells that we will encounter.}
\label{fig:examples3Cells}
\end{figure}

We now explain how we glue $3$-cells together to obtain a cell decomposition.
Endow every edge $e$ of a $3$-cell with an affine structure---a homeomorphism from $e$ to the interval $[0,1]$.
We \textbf{glue} two distinct faces of two (not necessarily distinct) $3$-cells via a homeomorphism that:
\begin{itemize}
    \item maps vertices to vertices;
    \item maps edges to edges; and
    \item restricts to an affine map on each edge.
\end{itemize}
A \textbf{cell decomposition} is a collection of finitely many $3$-cells
with some or all of their faces glued together in pairs;
we emphasise again that we allow faces from the same $3$-cell to be glued together.
Since triangulations are a special case of cell decompositions,
all of the subsequent definitions for cell decompositions apply to triangulations too.

Let $\mathcal{D}$ denote a cell decomposition.
The gluings that define $\mathcal{D}$ give an equivalence relation on the faces of the $3$-cells of $\mathcal{D}$;
call each equivalence class a \textbf{face} or \textbf{2-cell} of $\mathcal{D}$.
More explicitly, a face of $\mathcal{D}$ is either:
\begin{itemize}
\item a pair of $3$-cell faces that have been glued together, in which case we say that the face is \textbf{internal}; or
\item a single $3$-cell face that has been left unglued, in which case we say that the face is \textbf{boundary}.
\end{itemize}
The \textbf{boundary} of $\mathcal{D}$ is the (possibly empty) union of all its boundary faces.

The gluings that define $\mathcal{D}$ also merge vertices and edges of the $3$-cells into equivalence classes;
call each such vertex class a \textbf{vertex} or \textbf{0-cell} of $\mathcal{D}$, and call each such edge class an \textbf{edge} or \textbf{1-cell} of $\mathcal{D}$.
For each $k\in\{0,1,2\}$, define the \textbf{\emph{k}-skeleton} of $\mathcal{D}$, denoted $\mathcal{D}^{(k)}$,
to be the union of all $n$-cells of $\mathcal{D}$, where $n$ runs over all dimensions up to and including $k$.

In general, if we consider the quotient topology arising from the face gluings that define a cell decomposition $\mathcal{D}$,
the resulting topological space might fail to be a $3$-manifold.
Specifically, although nothing goes wrong in the interiors of $3$-cells and the interiors of faces,
we need to be careful with vertices and with midpoints of edges.

We begin by considering the midpoint $p$ of an edge $e$.
If $e$ lies entirely in the boundary of $\mathcal{D}$, then the frontier of a small regular neighbourhood of $p$ is a disc;
in this case, nothing goes wrong, and we say that $e$ is \textbf{boundary}.
However, if $e$ does not lie in the boundary, then we have two possibilities:
\begin{itemize}
    \item If $e$ is identified with itself in reverse, then the frontier of a small regular neighbourhood
    of $p$ is a projective plane;
    this cannot occur in a $3$-manifold.
    In this case, we say that $e$ is \textbf{invalid}.
    \item Otherwise, the frontier of a small regular neighbourhood of $p$ is a $2$-sphere.
    In this case, nothing goes wrong and we say that $e$ is \textbf{internal}.
\end{itemize}
We also say that $e$ is \textbf{valid} if it is either boundary or internal.

For a vertex $v$, consider the surface given by the frontier of a small regular neighbourhood of $v$;
we call this surface the \textbf{link} of $v$.
When $v$ lies in the boundary of $\mathcal{D}$, its link is a surface with boundary.
If the link is a disc, then nothing goes wrong and we say that $v$ is \textbf{boundary};
otherwise, if the link is any other surface with boundary, then we say that the vertex is \textbf{invalid}.

On the other hand, when $v$ does not lie in the boundary, its link is a closed surface.
If the link is a $2$-sphere, then nothing goes wrong and we say that $v$ is \textbf{internal};
otherwise, if the link is any other closed surface, then we say that $v$ is \textbf{ideal}.

A cell decomposition is \textbf{valid} if it has no invalid edges or vertices, and \textbf{invalid} otherwise.
Given a (possibly invalid) cell decomposition $\mathcal{D}$, we often find it useful to
\textbf{truncate} a vertex $v$ by deleting a small open regular neighbourhood of $v$.
In particular, by truncating each ideal or invalid vertex in $\mathcal{D}$,
we obtain a pseudomanifold $\mathcal{P}$ that we call the \textbf{truncated pseudomanifold} of $\mathcal{D}$;
the reason $\mathcal{P}$ is a pseudomanifold (and not necessarily a manifold)
is that midpoints of invalid edges in $\mathcal{D}$ would give non-manifold points in $\mathcal{P}$.

Observe that if $\mathcal{D}$ has no invalid edges, then the truncated pseudomanifold $\mathcal{P}$ is actually a (compact) $3$-manifold.
In this case, we will often refer to $\mathcal{P}$ as the \textbf{truncated $3$-manifold} of $\mathcal{D}$,
and we will say that $\mathcal{D}$ \textbf{represents} the $3$-manifold $\mathcal{P}$;
when $\mathcal{D}$ happens to be a triangulation, we will also often say that $\mathcal{D}$ \textbf{triangulates} $\mathcal{P}$.
Moreover, in the case where $\mathcal{D}$ is valid and has no ideal vertices, since we do not need to truncate any vertices
to obtain the truncated $3$-manifold $\mathcal{P}$, we will sometimes find it more natural to
refer to $\mathcal{P}$ as the \textbf{underlying $3$-manifold} of $\mathcal{D}$.

If we assume that $\mathcal{D}$ is actually valid (so it has neither invalid edges nor invalid vertices),
then the boundary components of the truncated $3$-manifold $\mathcal{P}$ come in two possible types:
\begin{itemize}
\item \textbf{ideal} boundary components, which are the boundary components that arise from truncating the ideal vertices; and
\item \textbf{real} boundary components, which are built from boundary faces of $\mathcal{D}$.
\end{itemize}
In this case, it will be convenient to distinguish the following special types of cell decompositions:
\begin{itemize}
\item A valid cell decomposition is \textbf{closed} if every vertex is internal.
For a closed cell decomposition, the truncated $3$-manifold is a closed $3$-manifold.
\item A valid cell decomposition is \textbf{bounded} if it has at least one boundary vertex, and has no ideal vertices.
For a bounded cell decomposition, the truncated $3$-manifold is a bounded $3$-manifold whose boundary components are all real.
\item A valid cell decomposition is \textbf{ideal} if it has at least one ideal vertex, and has no boundary vertices.
For an ideal cell decomposition, the truncated $3$-manifold is again a bounded $3$-manifold, but this time the boundary components are all ideal.
\end{itemize}

\begin{remark*}
When we have an ideal cell decomposition $\mathcal{D}$, we use the notion of
the truncated $3$-manifold to turn $\mathcal{D}$ into a \emph{compact} $3$-manifold $\mathcal{M}$.
A very common alternative (which we do not use in this paper) is to turn $\mathcal{D}$ into a \emph{noncompact} $3$-manifold $\mathcal{M}'$ by
simply deleting (rather than truncating) each ideal vertex.
Observe that $\mathcal{M}'$ is homeomorphic to the interior of $\mathcal{M}$, so this distinction is not too important.
\end{remark*}

\begin{remark*}
Suppose $\mathcal{T}$ is either a closed or ideal triangulation, and let $\mathcal{M}$ denote the truncated $3$-manifold of $\mathcal{T}$.
Since we do not truncate the internal vertices of $\mathcal{T}$, observe that $\mathcal{M}$ is a $3$-manifold with no $2$-sphere boundary components.
For this reason, we will often find it convenient to make the mild assumption that a $3$-manifold has no $2$-sphere boundary components.
\end{remark*}

\subsection{Decomposing along curves and surfaces}\label{subsec:decomposing}

The goal in this section is to introduce some terminology that will streamline our descriptions of the topological effects of crushing.
The idea is that crushing often changes the truncated $3$-manifold or pseudomanifold by ``decomposing along'' a properly embedded surface;
we will build gradually towards defining precisely what we mean by this.
We start by going one dimension down, and defining what we mean by decomposing a surface along an embedded curve;
this is useful in its own right, since it will help us describe how crushing changes the links of vertices.

Consider an embedded closed curve $\gamma$ in a compact surface $S$.
Let $S^\dagger$ denote the surface obtained from $S$ by \textbf{cutting along}
$\gamma$---that is, removing a small open regular neighbourhood of $\gamma$ from $S$.
If $\gamma$ is a two-sided curve in $S$, then we have two new copies of $\gamma$ in $\partial S^\dagger$;
on the other hand, if $\gamma$ is one-sided, then we have a single new curve in $\partial S^\dagger$.
Call each of these new curves in $\partial S^\dagger$ a \textbf{remnant} of $\gamma$;
see Figures~\ref{subfig:cuttingTwoSidedCurve} and~\ref{subfig:cuttingOneSidedCurve}.
Consider the surface $S'$ given by \textbf{filling} each remnant of $\gamma$ with a disc;
we say that $S'$ is obtained from $S$ by \textbf{decomposing along} $\gamma$.

\begin{figure}[htbp]
\raggedleft
	\begin{subfigure}[t]{0.78\textwidth}
	\centering
		\begin{tikzpicture}

		\node (before) at (0,0) {
			\includegraphics[scale=1.4]{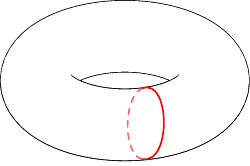}
		};
		\node (after) at (7,0) {
			\includegraphics[scale=1.4]{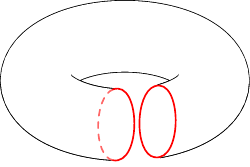}
		};
		\draw[thick, line cap=round, -Stealth] ($(before.east)+(-0.03,0)$) -- ($(after.west)+(0.05,0)$);

		\end{tikzpicture}
	\caption{Cutting along a two-sided curve yields a pair of remnants.}
	\label{subfig:cuttingTwoSidedCurve}
	\end{subfigure}
	\hfill
	\bigskip
	\hfill
	\begin{subfigure}[t]{0.78\textwidth}
	\centering
		\begin{tikzpicture}

		\node (before) at (0,0) {
			\includegraphics[scale=0.6]{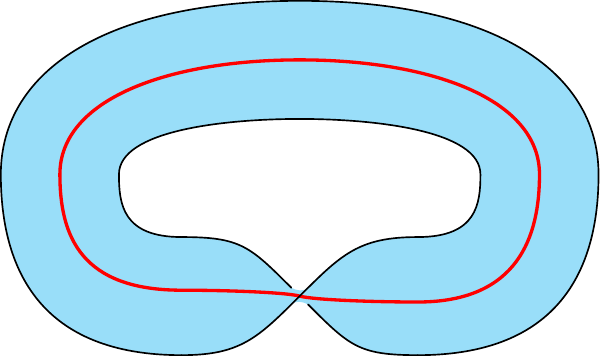}
		};
		\node (after) at (7,0) {
			\includegraphics[scale=0.6]{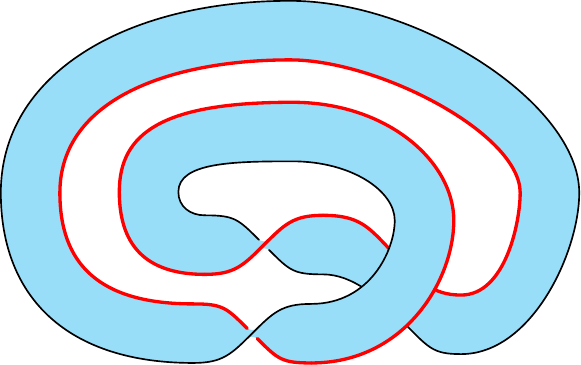}
		};
		\draw[thick, line cap=round, -Stealth] ($(before.east)+(-0.03,0)$) -- ($(after.west)+(0.05,0)$);

		\end{tikzpicture}
	\caption{Cutting along a one-sided curve yields a single remnant.}
	\label{subfig:cuttingOneSidedCurve}
	\end{subfigure}
\caption{Cutting along an embedded closed curve in a surface.}
\label{fig:cuttingCurve}
\end{figure}

We now aim to define similar terminology for truncated pseudomanifolds.
Consider a (possibly disconnected) properly embedded surface $S$ in a truncated pseudomanifold $\mathcal{P}$.
Let $\mathcal{P}^\dagger$ denote the pseudomanifold obtained from $\mathcal{P}$ by \textbf{cutting along} $S$---similar to before,
this means that we obtain $\mathcal{P}^\dagger$ by removing a small open regular neighbourhood of $S$ from $\mathcal{P}$.
For each two-sided component $E$ of $S$, we have two new copies of $E$ in $\partial\mathcal{P}^\dagger$;
on the other hand, for each one-sided component $E$ of $S$, we have a single new double cover of $E$ in $\partial\mathcal{P}^\dagger$.
Call each of these new pieces in $\partial\mathcal{P}^\dagger$ a \textbf{remnant} of $S$.

For our purposes, it will be useful to have a notion of ``decomposing along'' $S$ when
$S$ is one of the following seven types of (properly embedded) surface:
\begin{itemize}
\item a $2$-sphere---which means that cutting along $S$ yields a pair of $2$-sphere remnants;
\item a two-sided annulus---which means that cutting along $S$ yields a pair of annulus remnants;
\item a one-sided annulus---which means that cutting along $S$ yields a single annulus remnant;
\item a two-sided projective plane---which means that cutting along $S$ yields a pair of projective plane remnants;
\item a one-sided projective plane---which means that cutting along $S$ yields a single $2$-sphere remnant;
\item a two-sided M\"{o}bius band---which means that cutting along $S$ yields a pair of M\"{o}bius band remnants; or
\item a one-sided M\"{o}bius band---which means that cutting along $S$ yields a single annulus remnant.
\end{itemize}
Notice that for these types of surface, the remnants are always either $2$-spheres, annuli, projective planes or M\"{o}bius bands.

Similar to what we did with curves on surfaces, we construct the result of ``decomposing along'' $S$ by ``filling'' the remnants of $S$.
To do this for projective plane and M\"{o}bius band remnants, we use the following terminology:
define an \textbf{invalid cone} to be a pseudomanifold given by taking a cone over a projective plane.
With this in mind, let $S^\dagger$ denote a remnant of $S$ in $\mathcal{P}^\dagger$,
and suppose $S^\dagger$ is either a $2$-sphere, annulus, projective plane or M\"{o}bius band.
We define the operation of \textbf{filling} $S^\dagger$ as follows:
\begin{itemize}
\item If $S^\dagger$ is a $2$-sphere, then filling means attaching a $3$-ball $B$ by identifying $S^\dagger$ with the $2$-sphere boundary of $B$.
\item If $S^\dagger$ is an annulus, then filling means attaching
a thickened disc $D\times[0,1]$ by identifying $S^\dagger$ with the annulus $(\partial D)\times[0,1]$.
\item If $S^\dagger$ is a projective plane, then filling means attaching
an invalid cone $\mathcal{C}$ by identifying $S^\dagger$ with the projective plane boundary of $\mathcal{C}$.
\item If $S^\dagger$ is a M\"{o}bius band, then filling means attaching an invalid cone $\mathcal{C}$ by
choosing a small open disc $D$ in $\partial\mathcal{C}$, and identifying $S^\dagger$ with the M\"{o}bius band given by $(\partial\mathcal{C})-D$.
\end{itemize}

Putting everything together, suppose $S$ is one of the seven types of surface listed above,
and let $\mathcal{P}'$ denote the pseudomanifold obtained from $\mathcal{P}^\dagger$ by filling each remnant of $S$.
We say that $\mathcal{P}'$ is obtained from $\mathcal{P}$ by \textbf{decomposing along} $S$.

\subsection{Normal surfaces}\label{subsec:normal}

A \textbf{normal surface} in a triangulation $\mathcal{T}$ is a (possibly disconnected) properly embedded surface that:
\begin{itemize}
\item is disjoint from the vertices of $\mathcal{T}$;
\item meets the edges and faces of $\mathcal{T}$ transversely; and
\item intersects each tetrahedron $\Delta$ of $\mathcal{T}$ in a (possibly empty) disjoint union of finitely many discs, called \textbf{elementary discs},
where each such disc forms a curvilinear triangle or quadrilateral whose vertices lie on different edges of $\Delta$.
\end{itemize}
Two normal surfaces are \textbf{normally isotopic} if they are related by a \textbf{normal isotopy}---that is,
an ambient isotopy that preserves each vertex, edge, face and tetrahedron of the triangulation.
Up to normal isotopy, the elementary discs in each tetrahedron $\Delta$ come in seven possible types:
\begin{itemize}
\item four \textbf{triangle types}, each of which separates one vertex of $\Delta$ from the other three,
as shown in Figure~\ref{subfig:elemTriangles}; and
\item three \textbf{quadrilateral types}, each of which separates a pair of opposite edges of $\Delta$,
as shown in Figure~\ref{subfig:elemQuads}.
\end{itemize}
Observe that if a tetrahedron contains two elementary quadrilaterals of different types, then these two quadrilaterals will always intersect each other;
since normal surfaces are embedded, this means that if a tetrahedron contains quadrilaterals, then these quadrilaterals must all be of the same type.

\begin{figure}[htbp]
\centering
	\begin{subfigure}[t]{0.23\textwidth}
	\centering
		\includegraphics[scale=0.75]{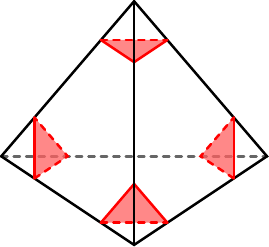}
	\caption{The four triangle types.}
	\label{subfig:elemTriangles}
	\end{subfigure}
	\hfill
	\begin{subfigure}[t]{0.67\textwidth}
	\centering
		\begin{tikzpicture}

		\node[inner sep=0pt] at (0,0) {
			\includegraphics[scale=0.75]{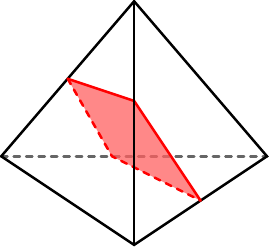}
		};
		\node[inner sep=0pt] at (3.8,0) {
			\includegraphics[scale=0.75]{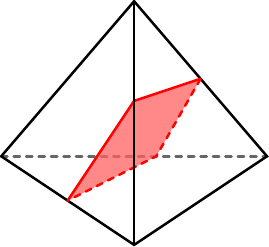}
		};
		\node[inner sep=0pt] at (7.6,0) {
			\includegraphics[scale=0.75]{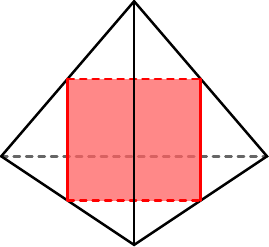}
		};

		\end{tikzpicture}
	\caption{The three quadrilateral types.}
	\label{subfig:elemQuads}
	\end{subfigure}
\caption{The seven types of elementary disc.}
\label{fig:elemDiscs}
\end{figure}

We call a normal surface \textbf{non-trivial} if it includes at least one elementary quadrilateral, and \textbf{trivial} otherwise.
It is easy to see that trivial normal surfaces always exist, and that every component of such a surface is just a vertex link.
The existence of non-trivial normal surfaces is less obvious.
In fact, it is possible to prove that many ``interesting'' embedded surfaces appear as (non-trivial) normal surfaces;
we will get a glimpse of why this is the case when we discuss the theory of barriers and normalisation in Section~\ref{subsec:barriers}.

\begin{figure}[htbp]
\centering
	\includegraphics[scale=1]{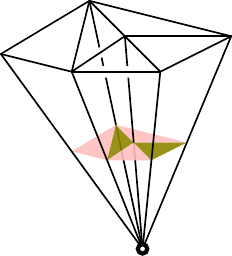}
\caption{A portion of a normal surface built entirely out of triangles.}
\label{fig:trivialNormal}
\end{figure}

A normal surface naturally splits a triangulation into a finer cell decomposition.
To describe this idea more precisely, we introduce the following definitions,
which are partly based on some terminology used by Jaco and Rubinstein~\cite[p.~91]{JacoRubinstein2003}:

\begin{definitions}\label{defs:inducedCells}
Let $S$ be a normal surface in a triangulation $\mathcal{T}$.
The surface $S$ divides each tetrahedron $\Delta$ of $\mathcal{T}$ into a collection of \textbf{induced cells} of the following types:
\begin{itemize}[nosep]
\item \textbf{Parallel cells} of two types (see Figure~\ref{fig:parallelCells}):
	\begin{itemize}[nosep]
	\item \textbf{Parallel triangular cells}:
	These lie between two parallel triangles of $S$.
	\item \textbf{Parallel quadrilateral cells}:
	These lie between two parallel quadrilaterals of $S$.
	\end{itemize}
\item \textbf{Non-parallel cells} of nine types:
	\begin{itemize}[nosep]
	\item \textbf{Corner cells}: These are tetrahedra that lie between a single triangle of $S$ and a single vertex of $\Delta$
	\item \textbf{Wedge cells} of three types (see Figure~\ref{fig:wedgeCells}):
	These only occur when $S$ meets $\Delta$ in one or more quadrilaterals.
	In this case, if we ignore any parallel and corner cells in $\Delta$, then the two cells left over are the wedge cells.
	\item \textbf{Central cells} of five types (see Figure~\ref{fig:centralCells}):
	These only occur when $S$ does not meet $\Delta$ in any quadrilaterals.
	In this case, if we ignore any parallel and corner cells in $\Delta$, then the single cell left over is the central cell.
	\end{itemize}
\end{itemize}
Amongst the faces of these induced cells, we will find it useful to distinguish the \textbf{bridge faces},
which are the quadrilateral faces that intersect $S$ precisely in a pair of opposite edges.
Note that bridge faces only appear in parallel and wedge cells (see Figures~\ref{fig:parallelCells} and~\ref{fig:wedgeCells}).

Let $\mathcal{P}$ denote the truncated pseudomanifold of $\mathcal{T}$,
and let $\mathcal{P}^\dagger$ denote the pseudomanifold obtained from $\mathcal{P}$ by cutting along $S$.
The induced cells naturally yield a cell decomposition $\mathcal{D}$ of $\mathcal{P}$,
such that the surface $S$ is given by a union of faces of $\mathcal{D}$.
Moreover, ungluing the faces of $\mathcal{D}$ that lie inside $S$ yields a cell decomposition $\mathcal{D}^\dagger$ of $\mathcal{P}^\dagger$.
We say that the cell decompositions $\mathcal{D}$ and $\mathcal{D}^\dagger$, and any cell decompositions given by
components of $\mathcal{D}$ and $\mathcal{D}^\dagger$, are \textbf{induced} by the normal surface $S$.
\end{definitions}

\begin{figure}[htbp]
\centering
	\begin{subfigure}[t]{0.49\textwidth}
	\centering
		\includegraphics[scale=0.7]{CrushFigures/CrushCellParallelTri.pdf}
	\caption{A parallel triangular cell.}
	\label{subfig:parallelTri}
	\end{subfigure}
	\hfill
	\begin{subfigure}[t]{0.49\textwidth}
	\centering
		\includegraphics[scale=0.7]{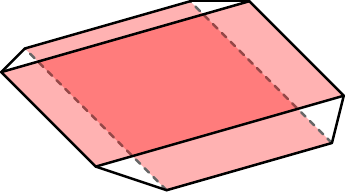}
	\caption{A parallel quadrilateral cell.}
	\label{subfig:parallelQuad}
	\end{subfigure}
	\hfill
\caption{A normal surface $S$ can induce two types of parallel cell.
The faces that lie inside $S$ are shaded red;
all the other faces are bridge faces.}
\label{fig:parallelCells}
\end{figure}

\begin{figure}[htbp]
\centering
	\begin{subfigure}[t]{0.27\textwidth}
	\centering
		\includegraphics[scale=0.7]{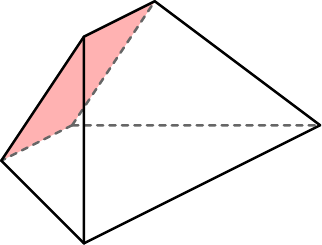}
	\caption{A wedge cell with no bridge faces.}
	\label{subfig:wedge0}
	\end{subfigure}
	\hfill
	\begin{subfigure}[t]{0.27\textwidth}
	\centering
		\includegraphics[scale=0.7]{CrushFigures/CrushCellWedge1.pdf}
	\caption{A wedge cell with one bridge face.}
	\label{subfig:wedge1}
	\end{subfigure}
	\hfill
	\begin{subfigure}[t]{0.27\textwidth}
	\centering
		\includegraphics[scale=0.7]{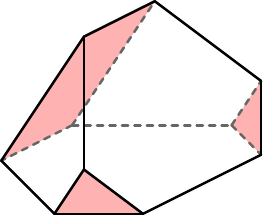}
	\caption{A wedge cell with two bridge faces.}
	\label{subfig:wedge2}
	\end{subfigure}
\caption{A normal surface $S$ can induce three types of wedge cell.
The faces that lie inside $S$ are shaded red.}
\label{fig:wedgeCells}
\end{figure}

\begin{figure}[htbp]
\centering
	\begin{tikzpicture}

	\node at (0,0) {
		\includegraphics[scale=0.7]{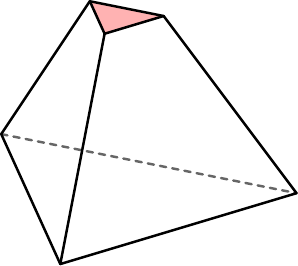}
	};
	\node at (4.3,0) {
		\includegraphics[scale=0.7]{CrushFigures/CrushCellCentral2.pdf}
	};
	\node at (8.6,0) {
		\includegraphics[scale=0.7]{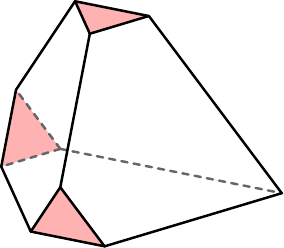}
	};
	\node at (12.5,0) {
		\includegraphics[scale=0.7]{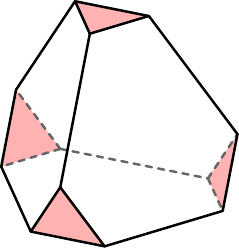}
	};

	\end{tikzpicture}
\caption{A normal surface $S$ can induce five types of central cell:
either a tetrahedron, or one of the four non-tetrahedron cells shown here.
The faces that lie inside $S$ are shaded red.}
\label{fig:centralCells}
\end{figure}

Since a tetrahedron can contain many parallel elementary discs, we could have arbitrarily many parallel cells.
However, there are always at most six non-parallel cells per tetrahedron $\Delta$:
\begin{itemize}
\item If $\Delta$ meets the normal surface in one or more quadrilaterals,
then we have no central cells, exactly two wedge cells, and up to four corner cells.
\item If $\Delta$ does not meet the normal surface in any quadrilaterals,
then we have no wedge cells, exactly one central cell, and again up to four corner cells.
\end{itemize}
We will find this simple observation useful in Section~\ref{subsec:badBigonPaths}.

\subsection{Barriers and normalisation}\label{subsec:barriers}

We now review the theory of normalisation, which gives a procedure for
transforming any properly embedded surface $S$ into a normal surface (not necessarily isotopic to $S$).
We also review the notion of a barrier surface, which gives a tool for ``controlling'' the result of the normalisation procedure.
The material here is essentially an abridged and informal version of Section~3 of~\cite{JacoRubinstein2003},
focusing only on the details that are necessary for our purposes in this paper.

Throughout Section~\ref{subsec:barriers}, let $S$, $S'$ and $B$ denote (possibly disconnected) surfaces that are properly embedded in a triangulation $\mathcal{T}$.
Assume that these surfaces are disjoint from the vertices of $\mathcal{T}$, and transverse to the $2$-skeleton of $\mathcal{T}$.

The idea of the normalisation procedure is to reduce the number of ``anomalies'' in a surface $S$ until it becomes a normal surface.
For instance, for $S$ to be a normal surface, it cannot intersect any tetrahedron $\Delta$ in anomalous pieces such as:
\begin{itemize}
\item a $2$-sphere component that is \textbf{trivial} in the sense that it lies entirely inside $\Delta$; or
\item a disc component that is \textbf{trivial} in the sense that its boundary curve lies
entirely inside a single boundary face, and its interior lies entirely inside $\Delta$.
\end{itemize}
To keep track of these and other anomalous features of $S$, we use the following measures of ``complexity'':
\begin{itemize}
\item Define the \textbf{weight} $\Wt(S)$ to be the number of times $S$ meets the $1$-skeleton $\mathcal{T}^{(1)}$;
that is, $\Wt(S)=\left\lvert S\cap\mathcal{T}^{(1)} \right\rvert$.
In general, $S$ could meet a tetrahedron $\Delta$ of $\mathcal{T}$ in a non-normal piece that ``doubles back'' on itself to meet a single edge twice
(for example, see Figures~\ref{fig:normalMoveIsotopy} and~\ref{fig:normalMoveBdryCompress});
the weight of $S$ gives a proxy for counting the number of such anomalies.
\item For each tetrahedron $\Delta$ of $\mathcal{T}$, let
\[
x_\Delta = \sum_{c\neq S^2}\left( 1-\chi(c) \right),
\]
where $c$ runs over all components of $S\cap\Delta$ other than $2$-spheres;
define the \textbf{local Euler number} $\lambda(S)$ to be
\[
\sum_{\Delta}x_\Delta.
\]
Recall that a normal surface must, in particular, meet each tetrahedron of $\mathcal{T}$ in a disjoint union of discs;
apart from trivial $2$-spheres (which we handle separately), the local Euler number detects any anomalies that violate this requirement.
\item Let $\sigma(S)$ denote the number of closed curves in which $S$ intersects the internal faces of $\mathcal{T}$.
A normal surface cannot have any such anomalous curves.
\item Let $\tau(S)$ denote the number of components of $S$ that form trivial $2$-spheres or trivial discs.
\end{itemize}
Define the \textbf{complexity} of $S$, denoted $C(S)$, to be the tuple
\[
\left(\; \Wt(S),\; \lambda(S),\; \sigma(S),\; \tau(S) \;\right).
\]
We will consider $S$ to have smaller complexity than some other surface $S'$ if $C(S)$ occurs before $C(S')$ in the lexicographical ordering.
As suggested earlier, normalisation consists of a series of steps, each of which reduces the complexity.

Before we define the steps involved in normalisation, we introduce some useful terminology.
Call a disc $D$ an \textbf{edge-compression disc} for $S$ if it is embedded so that:
\begin{itemize}
\item the interior of $D$ lies entirely in the interior of a tetrahedron $\Delta$ of $\mathcal{T}$; and
\item the boundary of $D$ consists of two arcs $\alpha$ and $\gamma$ that intersect each other only at their endpoints,
such that $\alpha=D\cap S$ and $\gamma$ is a sub-arc of an edge $e$ of $\Delta$.
\end{itemize}
Examples of edge-compression discs are shown in Figures~\ref{fig:normalMoveIsotopy}
and~\ref{fig:normalMoveBdryCompress}.
Call an edge-compression disc \textbf{internal} if it meets an internal edge of $\mathcal{T}$,
and \textbf{boundary} if it meets a boundary edge of $\mathcal{T}$;
notice that a boundary edge-compression disc is, in particular, a $\partial$-compression disc for $S$.

With all the preceding setup in mind, the normalisation procedure proceeds by performing the following \textbf{normal moves} on a surface $S$:
\begin{enumerate}[label={(\arabic*)}]
\item\label{normalMove:tetCompress}
\textbf{Compressions along discs that lie entirely in the interior of a tetrahedron.} (See Figure~\ref{fig:normalMoveTetCompress}.)\\
Each such compression reduces the complexity $C(S)$ because
it leaves the weight $\Wt(S)$ unchanged and reduces the local Euler number $\lambda(S)$.
These compressions can be performed until $S$ meets each tetrahedron of $\mathcal{T}$ in a union of $2$-spheres and discs, at which point $\lambda(S) = 0$.
We assume for the rest of the normal moves that we have already reduced $\lambda(S)$ to $0$ in this way.
\begin{figure}[htbp]
\centering
	\begin{tikzpicture}

	\node (before) at (0,0) {
		\includegraphics[scale=1]{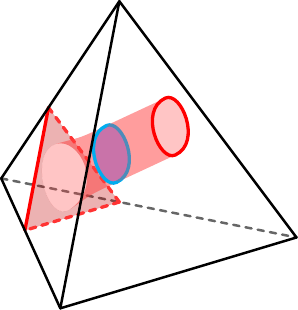}
	};
	\node (after) at (6.9,0) {
		\includegraphics[scale=1]{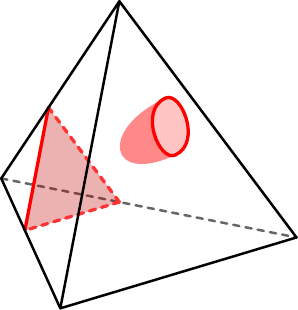}
	};
	\draw[thick, line cap=round, -Stealth] ($(before.east)+(-0.8,0)$) --++ (2.5,0);

	\end{tikzpicture}
\caption{An example of a normal move of type~\ref{normalMove:tetCompress}:
a compression of a surface (shaded red) along a compression disc (shaded blue) lying entirely inside a tetrahedron.}
\label{fig:normalMoveTetCompress}
\end{figure}
\item\label{normalMove:isotopy}
\textbf{Isotopies along internal edge-compression discs.} (See Figure~\ref{fig:normalMoveIsotopy}.)\\
Each such isotopy reduces the complexity $C(S)$ because it reduces the weight $\Wt(S)$.
\begin{figure}[htbp]
\centering
	\begin{tikzpicture}

	\node (before) at (0,0) {
		\includegraphics[scale=1]{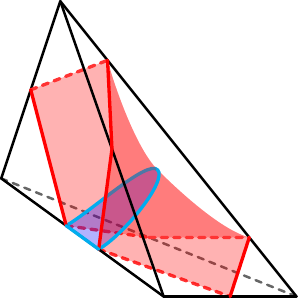}
	};
	\node (after) at (6,-0.5) {
		\includegraphics[scale=1]{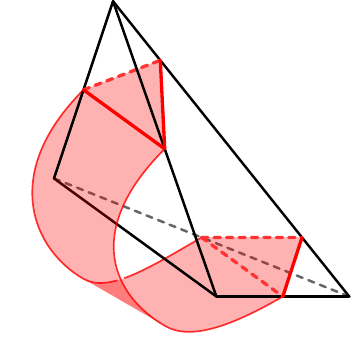}
	};
	\draw[thick, line cap=round, -Stealth] ($(before.east)+(-1.7,0)$) --++ (2.5,0);

	\end{tikzpicture}
\caption{An example of a normal move of type~\ref{normalMove:isotopy}:
an isotopy of a surface (shaded red) along an internal edge-compression disc (shaded blue).}
\label{fig:normalMoveIsotopy}
\end{figure}
\item\label{normalMove:bdryCompress}
\textbf{$\partial$-compressions along boundary edge-compression discs.} (See Figure~\ref{fig:normalMoveBdryCompress}.)\\
Like the isotopies in the previous step, each such $\partial$-compression reduces the complexity $C(S)$ because
it reduces the weight $\Wt(S)$.
For the remaining two normal moves, we assume that we have performed all possible isotopies and $\partial$-compressions along edge-compression discs,
which ensures that $S$ meets each tetrahedron $\Delta$ of $\mathcal{T}$ in a union of:
	\begin{itemize}
	\item elementary discs;
	\item trivial $2$-spheres; and
	\item discs whose boundary curves lie entirely in the interior of some face of $\Delta$.
	\end{itemize}
\begin{figure}[htbp]
\centering
	\begin{tikzpicture}

	\node (before) at (0,0) {
		\includegraphics[scale=1]{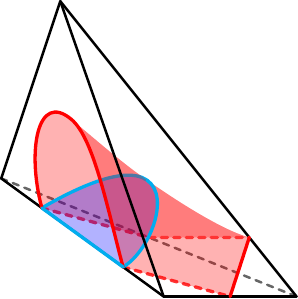}
	};
	\node (after) at (6.1,0) {
		\includegraphics[scale=1]{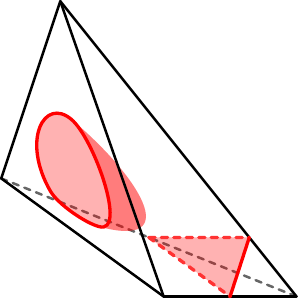}
	};
	\draw[thick, line cap=round, -Stealth] ($(before.east)+(-1.7,0)$) --++ (2.5,0);

	\end{tikzpicture}
\caption{An example of a normal move of type~\ref{normalMove:bdryCompress}:
a $\partial$-compression of a surface (shaded red) along a boundary edge-compression disc (shaded blue).}
\label{fig:normalMoveBdryCompress}
\end{figure}
\item\label{normalMove:faceCompress}
\textbf{Compressions along discs that lie entirely in the interior of an internal face.} (See Figure~\ref{fig:normalMoveFaceCompress}.)\\
Each such compression reduces the complexity $C(S)$ because
it leaves $\Wt(S)$ and $\lambda(S)$ unchanged, and reduces $\sigma(S)$.
After performing these compressions until no more such moves are possible,
$S$ meets each tetrahedron of $\mathcal{T}$ in a union of elementary discs, trivial $2$-spheres, and trivial discs;
we assume for the final normal move that this has already been done.
\begin{figure}[htbp]
\centering
	\begin{tikzpicture}

	\node (before) at (0,0) {
		\includegraphics[scale=1]{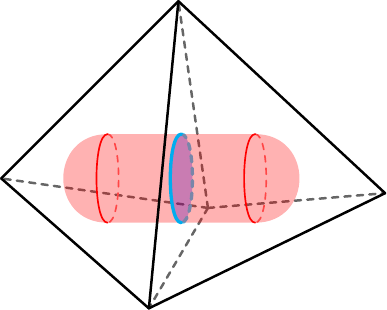}
	};
	\node (after) at (8.65,0) {
		\includegraphics[scale=1]{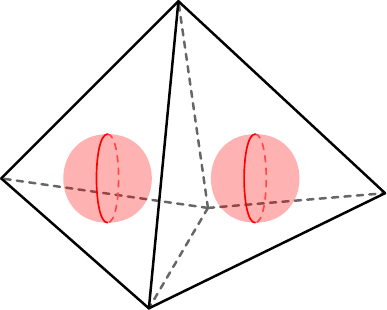}
	};
	\draw[thick, line cap=round, -Stealth] ($(before.east)+(-0.45,0)$) --++ (2.5,0);

	\end{tikzpicture}
\caption{An example of a normal move of type~\ref{normalMove:faceCompress}:
a compression of a surface (shaded red) along a compression disc (shaded blue) lying entirely inside an internal face.}
\label{fig:normalMoveFaceCompress}
\end{figure}
\item\label{normalMove:deleteTrivial}
\textbf{Deletion of trivial $2$-sphere and disc components.}\\
This final ``clean-up'' step reduces the complexity $C(S)$ because
it leaves $\Wt(S)$, $\lambda(S)$ and $\sigma(S)$ unchanged, and reduces $\tau(S)$ to zero.
At the end of this step, $S$ is a normal surface.
\end{enumerate}
For a complete explanation of why normalisation works as we have claimed, see Section~3.2 of~\cite{JacoRubinstein2003}.
Note that, in general, the normal surface that we obtain might not be isotopic to the original surface,
because of the steps where we perform compressions and $\partial$-compressions.
However, if we assume that the original surface was incompressible and $\partial$-incompressible, and also that the ambient $3$-manifold is irreducible and $\partial$-irreducible,
then normalising must produce a normal surface with one component isotopic to the original surface.

We can get even more control over the result of normalisation using the notion of a barrier surface;
we now review the aspects of barrier surfaces that we require for our purposes.
Given a properly embedded surface $B$ in $\mathcal{T}$, let $\mathcal{N}$ denote a fixed but arbitrary component of $\mathcal{T}-B$.
Call $B$ a \textbf{barrier} for $\mathcal{N}$ if any surface $S$ that is properly embedded in $\mathcal{N}$ can actually be normalised inside $\mathcal{N}$;
that is, the discs along which we compress, isotope or $\partial$-compress always lie entirely inside $\mathcal{N}$,
and at every stage the surface $S$ remains properly embedded in $\mathcal{N}$.

In Theorem~3.2 from~\cite{JacoRubinstein2003}, Jaco and Rubinstein list a number of examples of barrier surfaces.
For our purposes, we will need part~(5) of this theorem, which we restate here:

\begin{theorem}\label{thm:subcomplexBarrier}
Consider a (compact) $3$-manifold $\mathcal{M}$ with no $2$-sphere boundary components.
If $\mathcal{M}$ is closed, let $\mathcal{T}$ be a closed triangulation of $\mathcal{M}$;
otherwise, if $\mathcal{M}$ is bounded, let $\mathcal{T}$ be an ideal triangulation of $\mathcal{M}$.
Let $S$ be a normal surface in $\mathcal{T}$, and let $A$ be a subcomplex of the cell decomposition of $\mathcal{M}$ induced by $S$.
The boundary $B$ of a small regular neighbourhood of $S\cup A$ is a barrier surface for any component of $\mathcal{M}-B$ that does not meet $S\cup A$.
\end{theorem}

\subsection{Crushing via atomic moves}\label{subsec:atomic}

The main purpose of this section is to review the atomic formulation of crushing that was introduced by the first author~\cite{Burton2014}.
We augment this with some new terminology, as this will be useful for our purposes in Section~\ref{sec:genus}.
To begin, we state a version of Definition~1 from~\cite{Burton2014}:

\begin{definitions}[The crushing procedure]\label{defs:stages}
Let $S$ be a normal surface in a triangulation $\mathcal{T}$.
Each of the following operations builds on the previous one:
\begin{enumerate}[nosep, label={(\arabic*)}]
    \item Cut along $S$, and let $\mathcal{D}$ denote the resulting induced cell decomposition.
    \item Using the quotient topology, collapse each remnant of $S$ to a point.
    This turns $\mathcal{D}$ into a new cell decomposition $\mathcal{D}'$
    with $3$-cells of the following four possible types
    (see Figure~\ref{fig:destructibleCellFlatten}):
    \begin{itemize}[nosep]
        \item \textbf{3-sided footballs}, which are obtained from
        corner cells and parallel triangular cells;
        \item \textbf{4-sided footballs}, which are obtained from parallel quadrilateral cells;
        \item \textbf{triangular purses}, which are obtained from wedge cells; and
        \item \textbf{tetrahedra}, which are obtained from central cells.
    \end{itemize}
    We say that $\mathcal{D}'$ is obtained by \textbf{non-destructively crushing} $S$.
    Also, if a cell decomposition $\mathcal{D}^\ast$ is built entirely from $3$-cells of the four types listed above
    (even if it was not directly obtained by non-destructive crushing),
    then we call $\mathcal{D}^\ast$ a \textbf{destructible} cell decomposition.
    \item\label{step:flatten} To recover a triangulation from a destructible cell decomposition $\mathcal{D}^\ast$,
    we first build an intermediate cell complex $\mathcal{C}^\ast$ by using the quotient topology to flatten:
    \begin{itemize}[nosep]
        \item all 3-sided and 4-sided footballs to edges; and
        \item all triangular purses to triangular faces.
    \end{itemize}
    This is illustrated in Figure~\ref{fig:destructibleCellFlatten}.
    Since triangulations are defined only by face gluings between tetrahedra, there are two ways in which $\mathcal{C}^\ast$ might fail to form a triangulation:
    \begin{enumerate}[nosep, label={(\alph*)}]
    \item
    $\mathcal{C}^\ast$ could contain vertices, edges and/or triangles that are \textbf{isolated}, meaning that they do not belong to any tetrahedra.
    \item
    $\mathcal{C}^\ast$ could contain vertices or edges that are \textbf{pinched}, meaning that they include identifications that are independent of any face gluings.
    In contrast, recall that every vertex of a triangulation is an equivalence class of vertices of tetrahedra, where all vertex identifications arise as consequences of face gluings.
    Similarly, every edge of a triangulation is given by edge identifications arising solely as consequences of face gluings.
    \end{enumerate}
    Thus, as illustrated in Figure~\ref{fig:extractTri}, we need to perform the following
    two operations to \textbf{extract} a triangulation $\mathcal{T}^\ast$ from $\mathcal{C}^\ast$:
    \begin{enumerate}[nosep, label={(\alph*)}]
        \item Delete all isolated vertices, edges and
        triangles.
        \item Separate pieces of the cell complex that are only joined together along pinched vertices or edges
        (thereby ensuring that all vertex and edge identifications arise solely as consequences of face gluings).
    \end{enumerate}
    We say that $\mathcal{T}^\ast$ is obtained by \textbf{flattening} $\mathcal{D}^\ast$.
    Consider the triangulation $\mathcal{T}'$ obtained by flattening
    the cell decomposition $\mathcal{D}'$ that results from non-destructively crushing $S$;
    we say that $\mathcal{T}'$ is obtained by \textbf{(destructively) crushing} $S$.
\qedhere
\end{enumerate}
\end{definitions}

\begin{figure}[htbp]
\centering
	\begin{subfigure}[t]{0.25\textwidth}
	\centering
		\begin{tikzpicture}

		\node[inner sep=0pt] (football) at (0,0) {
			\includegraphics[scale=0.7]{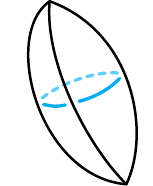}
		};
		\node[inner sep=0pt] (edge) at (2,0) {
			\includegraphics[scale=0.7]{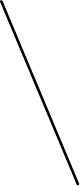}
		};
		\draw[thick, line cap=round, -Stealth] ($(football.east)+(-0.07,0)$) -- ($(edge.west)+(0.25,0)$);

		\end{tikzpicture}
	\caption{Flattening a $3$-sided football to an edge.}
	\label{subfig:football3}
	\end{subfigure}
	\hfill
	\begin{subfigure}[t]{0.25\textwidth}
	\centering
		\begin{tikzpicture}

		\node[inner sep=0pt] (football) at (0,0) {
			\includegraphics[scale=0.7]{CrushFigures/CrushCellFootball4.pdf}
		};
		\node[inner sep=0pt] (edge) at (2,0) {
			\includegraphics[scale=0.7]{CrushFigures/FlattenFootball.pdf}
		};
		\draw[thick, line cap=round, -Stealth] ($(football.east)+(-0.07,0)$) -- ($(edge.west)+(0.25,0)$);

		\end{tikzpicture}
	\caption{Flattening a $4$-sided football to an edge.}
	\label{subfig:football4}
	\end{subfigure}
	\hfill
	\begin{subfigure}[t]{0.45\textwidth}
	\centering
		\begin{tikzpicture}

		\node[inner sep=0pt] (purse) at (0,0) {
			\includegraphics[scale=0.7]{CrushFigures/CrushCellTriPurse.pdf}
		};
		\node[inner sep=0pt] (face) at (3.75,0) {
			\includegraphics[scale=0.7]{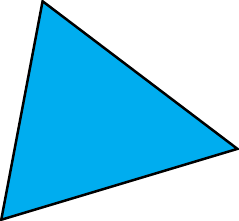}
		};
		\draw[thick, line cap=round, -Stealth] ($(purse.east)+(-0.12,0)$) -- ($(face.west)+(0.12,0)$);

		\end{tikzpicture}
	\caption{Flattening a triangular purse to a triangular face.}
	\label{subfig:triPurse}
	\end{subfigure}
\caption{In addition to tetrahedra, a destructible cell decomposition $\mathcal{D}^\ast$ can contain three other types of $3$-cells.
To recover a triangulation from $\mathcal{D}^\ast$, we need to flatten the non-tetrahedron cells.}
\label{fig:destructibleCellFlatten}
\end{figure}

\begin{figure}[htbp]
\centering
	\begin{tikzpicture}

	\node (before) at (0,0) {
		\includegraphics[scale=0.9]{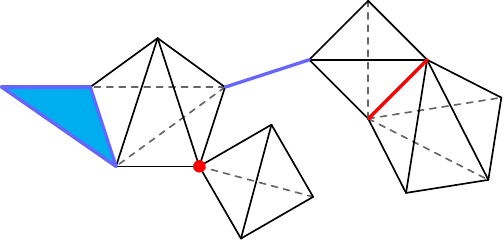}
	};
	\node (after) at (9,0) {
		\includegraphics[scale=0.9]{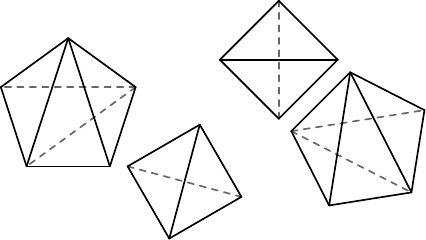}
	};
	\draw[thick, line cap=round, -Stealth] ($(before.east)+(0.05,0)$) -- ($(after.west)+(0.1,0)$);

	\end{tikzpicture}
\caption{Extracting a triangulation by deleting isolated edges and triangles (highlighted in blue),
and separating pinched vertices and edges (highlighted in red).}
\label{fig:extractTri}
\end{figure}

It is not too difficult to see what happens if we crush a \emph{trivial} normal surface $S$ in a triangulation $\mathcal{T}$.
Cutting along $S$ yields one central cell per tetrahedron, together with some number of corner and parallel triangular cells.
All the corner and parallel cells together form components that do not contain any central cells,
so after the non-destructive crushing and flattening steps, these components become isolated edges that do not appear in the final triangulation.
For the central cells, observe that non-destructive crushing turns these into tetrahedra that are glued together in the same way as the original triangulation.
The upshot is that crushing a trivial normal surface always leaves the triangulation unchanged.

Suppose now that $S$ is a \emph{non-trivial} normal surface in a triangulation $\mathcal{T}$,
and let $\mathcal{T}'$ denote the triangulation obtained by crushing $S$.
As before, each tetrahedron of $\mathcal{T}'$ comes from a central cell in the cell decomposition $\mathcal{D}$ given by cutting along $S$.
However, this time, at least one tetrahedron of $\mathcal{T}$ contains an elementary quadrilateral,
which means that not every tetrahedron of $\mathcal{T}$ gives rise to a central cell in $\mathcal{D}$.
Thus, we see that crushing has the following useful feature:

\begin{observation}\label{obs:crushReduce}
Let $\mathcal{T}$ be a triangulation, and let $\mathcal{T}'$ denote
the triangulation obtained by crushing a non-trivial normal surface in $\mathcal{T}$.
Then $\lvert\mathcal{T}'\rvert < \lvert\mathcal{T}\rvert$.
\end{observation}

The difficulty with crushing a non-trivial normal surface is that this operation could drastically change the topology of our triangulations.
In particular, the triangulations before and after crushing could represent different $3$-manifolds, assuming they even represent $3$-manifolds at all.
In~\cite{JacoRubinstein2003}, Jaco and Rubinstein work through this difficulty using
a complicated global analysis of their version of the crushing procedure.

In contrast, the formulation of crushing given in Definitions~\ref{defs:stages} is simpler to work with.
This is because the process of flattening a destructible cell decomposition
can always be realised by a sequence consisting of atomic moves of three types.
The following lemma~\cite[Lemma~3]{Burton2014} gives a precise statement of this idea:

\begin{lemma}[Crushing lemma]\label{lem:benCrushing}
Let $\mathcal{T}^\ast$ be the triangulation given by flattening some destructible cell decomposition $\mathcal{D}^\ast$.
Then $\mathcal{T}^\ast$ can be obtained from $\mathcal{D}^\ast$ by performing a sequence of
zero or more of the following \textbf{atomic moves} (see Figure~\ref{fig:atomicMoves}), one at a time, in some order:
\begin{itemize}[nosep]
    \item flattening a \textbf{triangular pillow} to a triangular face;
    \item flattening a \textbf{bigon pillow} to a bigon face; and
    \item flattening a bigon face to an edge.
\end{itemize}
Since our cell decompositions are defined only by face gluings between $3$-cells,
after each atomic move we implicitly \textbf{extract} a cell decomposition by:
\begin{itemize}[nosep]
    \item deleting all \textbf{isolated} vertices, edges, bigons and triangles that do not belong to any $3$-cells; and
    \item separating pieces of the cell complex that are only joined together along
    \textbf{pinched} vertices or edges.
\end{itemize}
\end{lemma}

\begin{figure}[htbp]
\centering
	\begin{subfigure}[t]{0.4\textwidth}
	\centering
		\begin{tikzpicture}

		\node[inner sep=0pt] (pillow) at (0,0) {
			\includegraphics[scale=0.7]{CrushFigures/CrushCellTriPillow.pdf}
		};
		\node[inner sep=0pt] (face) at (3.4,0) {
			\includegraphics[scale=0.7]{CrushFigures/FlattenTriPurse.pdf}
		};
		\draw[thick, line cap=round, -Stealth] ($(pillow.east)+(-0.32,0)$) -- ($(face.west)+(0.12,0)$);

		\end{tikzpicture}
	\caption{Flattening a triangular pillow.}
	\label{subfig:flattenTriPillow}
	\end{subfigure}
	\hfill
	\begin{subfigure}[t]{0.3\textwidth}
	\centering
		\begin{tikzpicture}

		\node[inner sep=0pt] (pillow) at (0,0) {
			\includegraphics[scale=0.7]{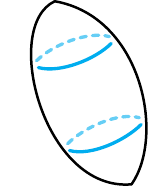}
		};
		\node[inner sep=0pt] (face) at (2.6,0) {
			\includegraphics[scale=0.7]{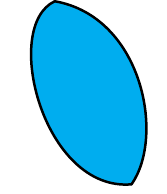}
		};
		\draw[thick, line cap=round, -Stealth] ($(pillow.east)+(-0.07,0)$) -- ($(face.west)+(0.25,0)$);

		\end{tikzpicture}
	\caption{Flattening a bigon pillow.}
	\label{subfig:flattenBigonPillow}
	\end{subfigure}
	\hfill
	\begin{subfigure}[t]{0.25\textwidth}
	\centering
		\begin{tikzpicture}

		\node[inner sep=0pt] (face) at (0,0) {
			\includegraphics[scale=0.7]{CrushFigures/FlattenBigon.pdf}
		};
		\node[inner sep=0pt] (edge) at (2,0) {
			\includegraphics[scale=0.7]{CrushFigures/FlattenFootball.pdf}
		};
		\draw[thick, line cap=round, -Stealth] ($(face.east)+(-0.07,0)$) -- ($(edge.west)+(0.25,0)$);

		\end{tikzpicture}
	\caption{Flattening a bigon face.}
	\label{subfig:flattenBigonFace}
	\end{subfigure}
\caption{The three atomic moves for flattening a destructible cell decomposition.}
\label{fig:atomicMoves}
\end{figure}

As part of the proof of the crushing lemma, the first author showed~\cite{Burton2014} that
if we are careful about the order in which we perform the atomic moves,
then we only ever encounter cell decompositions with $3$-cells of the following seven types:
\begin{itemize}
    \item 3-sided footballs;
    \item 4-sided footballs;
    \item triangular purses;
    \item tetrahedra;
    \item triangular pillows;
    \item bigon pillows; and
    \item \textbf{bigon pyramids} (see Figure~\ref{fig:bigonPyramid}).
\end{itemize}

\begin{figure}[htbp]
\centering
	\includegraphics[scale=0.7]{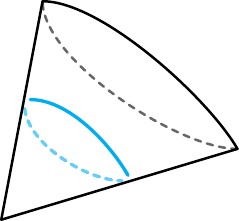}
\caption{A bigon pyramid.}
\label{fig:bigonPyramid}
\end{figure}

The crushing lemma allows us to understand the topological effects of crushing by examining atomic moves one at a time.
In particular, the first author proved the following result~\cite[Lemma~4]{Burton2014}
(which, among other things, paved the way for a practical algorithm for
non-orientable prime decomposition~\cite{Burton2014,Burton2013Regina}):

\begin{lemma}\label{lem:benCompact}
Let $\mathcal{D}_0$ be a valid cell decomposition with no ideal vertices.
If the underlying $3$-manifold $\mathcal{M}_0$ contains no two-sided projective planes,
then performing one of the atomic moves of Lemma~\ref{lem:benCrushing} will yield a (valid) cell decomposition
of a $3$-manifold $\mathcal{M}_1$ such that one of the following holds:
\begin{itemize}[nosep]
    \item $\mathcal{M}_0 = \mathcal{M}_1$.
    \item We flattened a triangular pillow, and $\mathcal{M}_1$ is obtained from $\mathcal{M}_0$ by
    deleting a single component $\mathcal{C}$, where $\mathcal{C}$ is either
    a $3$-ball, a $3$-sphere or a copy of the lens space $L_{3,1}$.
    \item We flattened a bigon pillow, and $\mathcal{M}_1$ is obtained from $\mathcal{M}_0$ by
    deleting a single component $\mathcal{C}$, where $\mathcal{C}$ is either
    a $3$-ball, a $3$-sphere, or a copy of real projective space $\mathbb{R}P^3$.
    \item We flattened a bigon face, and $\mathcal{M}_1$ is
    related to $\mathcal{M}_0$ in one of the following ways:
    \begin{enumerate}[nosep, label={(\roman*)}]
        \item $\mathcal{M}_1$ is obtained by cutting
        along a properly embedded disc in $\mathcal{M}_0$;
        \item $\mathcal{M}_1$ is obtained by filling
        a boundary $2$-sphere of $\mathcal{M}_0$ with a $3$-ball;
        \item $\mathcal{M}_1$ is obtained by decomposing along an embedded $2$-sphere in $\mathcal{M}_0$; or
        \item $\mathcal{M}_0 = \mathcal{M}_1 \ConnSum \mathbb{R}P^3$---that is,
        $\mathcal{M}_1$ removes a single $\mathbb{R}P^3$ summand from the connected sum decomposition of $\mathcal{M}_0$.
    \end{enumerate}
\end{itemize}
\end{lemma}

One of our main goals in this paper is to extend Lemma~\ref{lem:benCompact} to cell decompositions that
may be invalid and may have ideal vertices;
this, in particular, allows us to study the topological effects of crushing a closed surface of positive genus,
since non-destructively crushing such a surface produces a cell decomposition with ideal vertices.
To do this, we will find it helpful to have ``flattening maps'' that keep track of how the points in a cell decomposition are affected by an atomic move.
Although an atomic move ``looks like'' a quotient operation, the corresponding quotient map does not account for the
implicit operation of extracting a cell decomposition, so a little care is required to define ``flattening maps'' appropriately:

\begin{definitions}\label{defs:flatteningMap}
Let $\mathcal{D}_1$ be a cell decomposition obtained by performing a single atomic move on some cell decomposition $\mathcal{D}_0$.
In Lemma~\ref{lem:benCrushing}, each atomic move implicitly finishes with the operation of extracting a cell decomposition;
consider the intermediate cell complex $\mathcal{C}$ that we obtain by performing the atomic move \emph{without} subsequently extracting a cell decomposition.
Note that $\mathcal{C}$ is obtained as a quotient of $\mathcal{D}_0$, so we have a quotient map $q:\mathcal{D}_0\to\mathcal{C}$.

We use $q$ to construct a map $\widehat{\varphi}_0:\mathcal{D}_0 \to 2^{\mathcal{D}_1}$
(here, $2^X$ denotes the \emph{power set} of a set $X$) that acts on points $p$ in $\mathcal{D}_0$ as follows:
\begin{itemize}[nosep]
\item If $q(p)$ is part of an isolated vertex, edge, bigon or triangle---which
means that $q(p)$ is deleted when we extract a cell decomposition---then
take $\widehat{\varphi}_0(p)$ to be the empty set.
\item If $q(p)$ is part of a pinched edge or vertex---which
means that $q(p)$ gets separated into multiple points when we extract a cell decomposition---then
take $\widehat{\varphi}_0(p)$ to be the set of points in $\mathcal{D}_1$ that originate from $q(p)$.
\item Otherwise, $q(p)$ remains untouched when we extract a cell decomposition,
in which case we take $\widehat{\varphi}_0(p)=\{q(p)\}$
(here, by an abuse of notation, we are viewing $q(p)$ as a point in $\mathcal{D}_1$).
\end{itemize}
Intuitively, $\widehat{\varphi}_0$ keeps track of how points in $\mathcal{D}_0$ are affected when we perform an atomic move.

Observe that the non-empty sets in the image of $\widehat{\varphi}_0$ give a partition of the points in $\mathcal{D}_1$.
Thus, we can construct a map $\widehat{\varphi}_1:\mathcal{D}_1 \to 2^{\mathcal{D}_0}$ as follows:
for each point $p$ in $\mathcal{D}_1$, let $U$ be the (unique) set in the image of $\widehat{\varphi}_0$ that contains $p$,
and define $\widehat{\varphi}_1(p)$ to be the set $\widehat{\varphi}_0^{-1}(U)$.
Intuitively, $\widehat{\varphi}_1$ keeps track of how points in $\mathcal{D}_1$ would be affected if we perform an atomic move in reverse.

For each $i\in\{0,1\}$, define a map $\varphi_i:2^{\mathcal{D}_i} \to 2^{\mathcal{D}_{1-i}}$ that sends any subset $S$ of $\mathcal{D}_i$
(when we actually use the ideas defined here, $S$ will usually be a vertex, edge, face or $3$-cell of $\mathcal{D}_i$) to the set
\[
\bigcup_{p\in S}\widehat{\varphi}_i(p) \subseteq \mathcal{D}_{1-i}.
\]
We call $\varphi_0$ the \textbf{flattening map} associated to the atomic move,
and $\varphi_1$ the \textbf{inverse flattening map}
(although, strictly speaking, these maps are not actually inverses of each other).
\end{definitions}

\section{Atomic moves on cell decompositions with ideal vertices}\label{sec:atomicIdeal}

Let $S$ be a normal surface in a triangulation $\mathcal{T}$.
When $S$ is either a $2$-sphere or a disc, non-destructively crushing $S$ creates new vertices whose links are either $2$-spheres or discs.
Thus, if the vertices of $\mathcal{T}$ are all either internal or boundary,
then the topological effect of destructively crushing $S$ only depends on
how atomic moves affect cell decompositions whose vertices are all either internal or boundary;
this was the motivation for Lemma~\ref{lem:benCompact} in~\cite{Burton2014}.

Our main goal in this section is to extend this atomic approach to crushing beyond the case where $S$ is a $2$-sphere or disc.
This requires us to study atomic moves on cell decompositions that are allowed to have ideal or invalid vertices.
A similarly general understanding of atomic moves is necessary to understand crushing
if we allow the initial triangulation $\mathcal{T}$ to have ideal or invalid vertices.
Moreover, when $\mathcal{T}$ triangulates a non-orientable $3$-manifold,
it turns out to be possible for an atomic move to create an invalid \emph{edge}.
The upshot is that, for a completely general analysis of atomic moves,
we should not restrict the links of the vertices involved, and we should not exclude the possibility of invalid edges.

Of the three atomic moves, flattening a triangular pillow and flattening a bigon pillow are relatively straightforward to understand.
We study these two atomic moves in full generality in Section~\ref{subsec:pillows}.

We then devote Section~\ref{subsec:bigonFaces} to understanding the topological effects of flattening a bigon face.
In contrast to the other two atomic moves, there would be a tediously large number of cases to consider if we wanted to give a complete analysis.
Thus, for the sake of brevity and clarity, we will focus mainly on flattening bigon faces
in valid cell decompositions whose vertices are all either internal or ideal.
This is sufficient to understand the effects of crushing $S$ if the following conditions are satisfied:
\begin{itemize}
\item $S$ is a closed surface;
\item $\mathcal{T}$ is valid and has no boundary vertices; and
\item the truncated $3$-manifold of $\mathcal{T}$ contains no two-sided properly embedded projective planes or M\"{o}bius bands
(which is true, in particular, for all orientable $3$-manifolds).
\end{itemize}
For the cases not covered by our analysis, we leave the details to whomever may require them in future work.

\subsection{Flattening triangular and bigon pillows}\label{subsec:pillows}

\begin{lemma}[Flattening triangular pillows]\label{lem:flattenTriangularPillows}
Let $\mathcal{D}_0$ be a (possibly invalid) cell decomposition, and let $\mathcal{D}_1$ be
the cell decomposition obtained by flattening a triangular pillow $F$ in $\mathcal{D}_0$.
One of the following holds:
\begin{enumerate}[nosep, label={(\alph*)}]
\item\label{case:triangularPillowHomeo}
The two triangular faces of $F$ are not identified and not both boundary,
in which case the truncated pseudomanifolds of $\mathcal{D}_0$ and $\mathcal{D}_1$ are homeomorphic.
\item\label{case:triangularPillowBall}
$F$ forms a (bounded) cell decomposition of a $3$-ball,
in which case $\mathcal{D}_1$ is obtained from $\mathcal{D}_0$ by deleting this $3$-ball component.
\item\label{case:triangularPillowClosed}
$F$ forms a (closed) cell decomposition of either $S^3$ (the $3$-sphere) or $L_{3,1}$ (a lens space),
in which case $\mathcal{D}_1$ is obtained from $\mathcal{D}_0$ by deleting this closed component.
\item\label{case:triangularPillowInvalid}
$F$ forms a two-vertex component $\mathcal{C}$ of $\mathcal{D}_0$ with exactly one invalid edge $e$;
one of the vertices is incident to $e$ and has $2$-sphere link,
while the other vertex is not incident to $e$ and has projective plane link.
In this case, $\mathcal{D}_1$ is obtained from $\mathcal{D}_0$ by deleting this invalid component $\mathcal{C}$.
\end{enumerate}
\end{lemma}

\begin{proof}
Throughout this proof, let $t$ and $t'$ denote the triangular faces that bound the triangular pillow $F$.
We have several cases to consider, depending on
how $t$ and $t'$ are glued to other faces of $\mathcal{D}_0$ (if at all).

First, suppose $t$ and $t'$ are not glued to each other.
In this case, the triangular pillow $F$ forms a $3$-ball.
If $t$ and $t'$ are not both boundary, then this ball lives inside some larger component of $\mathcal{D}_0$,
and flattening $F$ does not change the truncated pseudomanifold;
this corresponds to case~\ref{case:triangularPillowHomeo}.
On the other hand, if $t$ and $t'$ are both boundary, then $F$ forms the entirety of a $3$-ball component of $\mathcal{D}_0$,
and flattening $F$ deletes this $3$-ball component;
this corresponds to case~\ref{case:triangularPillowBall}.

With that out of the way, suppose $t$ and $t'$ \emph{are} glued to each other.
Up to symmetry, there are two possibilities for an orientation-reversing gluing:
\begin{itemize}
\item If $t$ and $t'$ are glued without a twist, then $F$ forms a cell decomposition of $S^3$
(see Figure~\ref{subfig:triPillowSphere}).
\item If $t$ and $t'$ are glued with a twist, then $F$ forms a cell decomposition of $L_{3,1}$
(see Figure~\ref{subfig:triPillowLens}).
\end{itemize}
In either case, we see that $F$ forms a closed component of $\mathcal{D}_0$.
Moreover, flattening $F$ has the effect of deleting this closed component.
This corresponds to case~\ref{case:triangularPillowClosed}.

For an orientation-preserving gluing of $t$ and $t'$, there is only one possibility up to symmetry.
With this gluing, $F$ forms a two-vertex component $\mathcal{C}$ of $\mathcal{D}_0$ with exactly one invalid edge $e$
(see Figure~\ref{subfig:triPillowInvalid}).
One of the vertices of $\mathcal{C}$ is given by identifying the two endpoints of $e$, and has $2$-sphere link.
The other vertex of $\mathcal{C}$ is given by the vertex of $F$ disjoint from $e$, and has projective plane link.
This corresponds to case~\ref{case:triangularPillowInvalid}.
\end{proof}

\begin{figure}[htbp]
\centering
	\begin{subfigure}[t]{0.3\textwidth}
	\centering
		\includegraphics[scale=1]{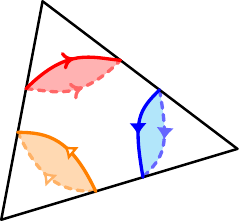}
	\caption{A triangular pillow that forms a (closed) cell decomposition of $S^3$.}
	\label{subfig:triPillowSphere}
	\end{subfigure}
	\hfill
	\begin{subfigure}[t]{0.3\textwidth}
	\centering
		\includegraphics[scale=1]{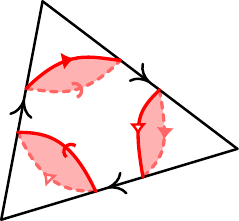}
	\caption{A triangular pillow that forms a (closed) cell decomposition of $L_{3,1}$.}
	\label{subfig:triPillowLens}
	\end{subfigure}
	\hfill
	\begin{subfigure}[t]{0.3\textwidth}
	\centering
		\includegraphics[scale=1]{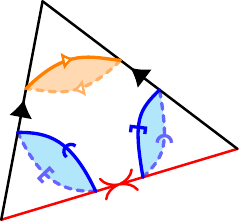}
	\caption{A triangular pillow that forms a component with an invalid edge (highlighted red).}
	\label{subfig:triPillowInvalid}
	\end{subfigure}
\caption{Cases where the two faces of a triangular pillow are glued to each other.}
\label{fig:triPillowGlued}
\end{figure}

\begin{lemma}[Flattening bigon pillows]\label{lem:flattenBigonPillows}
Let $\mathcal{D}_0$ be a (possibly invalid) cell decomposition, and let $\mathcal{D}_1$ be
the cell decomposition obtained by flattening a bigon pillow $F$ in $\mathcal{D}_0$.
One of the following holds:
\begin{enumerate}[nosep, label={(\alph*)}]
\item\label{case:bigonPillowHomeo}
The two bigon faces of $F$ are not identified and not both boundary,
in which case the truncated pseudomanifolds of $\mathcal{D}_0$ and $\mathcal{D}_1$ are homeomorphic.
\item\label{case:bigonPillowBall}
$F$ forms a (bounded) cell decomposition of a $3$-ball,
in which case $\mathcal{D}_1$ is obtained from $\mathcal{D}_0$ by deleting this $3$-ball component.
\item\label{case:bigonPillowClosed}
$F$ forms a (closed) cell decomposition of either $S^3$ (the $3$-sphere) or $\mathbb{R}P^3$ (real projective space),
in which case $\mathcal{D}_1$ is obtained from $\mathcal{D}_0$ by deleting this closed component.
\item\label{case:bigonPillowIdeal}
$F$ forms an ideal cell decomposition of $\mathbb{R}P^2\times[0,1]$,
in which case $\mathcal{D}_1$ is obtained from $\mathcal{D}_0$ by deleting this ideal component.
\item\label{case:bigonPillowInvalid}
$F$ forms a one-vertex component of $\mathcal{D}_0$ with exactly two invalid edges,
in which case $\mathcal{D}_1$ is obtained from $\mathcal{D}_0$ by deleting this invalid component.
\end{enumerate}
\end{lemma}

\begin{proof}
Throughout this proof, let $b$ and $b'$ denote the bigon faces that bound the bigon pillow $F$.
We have several cases to consider, depending on how $b$ and $b'$ are glued to other faces of $\mathcal{D}_0$ (if at all).

First, suppose $b$ and $b'$ are not glued to each other.
In this case, the bigon pillow $F$ forms a $3$-ball.
If $b$ and $b'$ are not both boundary, then this ball lives inside some larger component of $\mathcal{D}_0$,
and flattening $F$ does not change the truncated pseudomanifold;
this corresponds to case~\ref{case:bigonPillowHomeo}.
On the other hand, if $b$ and $b'$ are both boundary, then $F$ forms the entirety of a $3$-ball component of $\mathcal{D}_0$,
and flattening $F$ deletes this $3$-ball component;
this corresponds to case~\ref{case:bigonPillowBall}.

With that out of the way, suppose $b$ and $b'$ \emph{are} glued to each other.
There are two possibilities for an orientation-reversing gluing:
\begin{itemize}
\item If $b$ and $b'$ are glued without a twist, then $F$ forms a cell decomposition of $S^3$
(see Figure~\ref{subfig:bigonPillowSphere}).
\item If $b$ and $b'$ are glued with a twist, then $F$ forms a cell decomposition of $\mathbb{R}P^3$
(see Figure~\ref{subfig:bigonPillowProjective}).
\end{itemize}
In either case, we see that $F$ forms a closed component of $\mathcal{D}_0$.
Moreover, flattening $F$ has the effect of deleting this closed component.
This corresponds to case~\ref{case:bigonPillowClosed}.

Finally, for an orientation-preserving gluing, we again have two possibilities:
\begin{itemize}
\item One of these gluings causes the two edges of $F$ to be identified together, and does not create invalid edges.
In this case, $F$ forms an ideal cell decomposition of $\mathbb{R}P^2\times[0,1]$
(see Figure~\ref{subfig:bigonPillowIdeal}),
and flattening $F$ has the effect of deleting this ideal component.
This corresponds to case~\ref{case:bigonPillowIdeal}.
\item The other orientation-preserving gluing causes each edge of $F$ to be identified with itself in reverse,
so that $F$ forms a one-vertex component of $\mathcal{D}_0$ with exactly two invalid edges
(see Figure~\ref{subfig:bigonPillowInvalid}).
Flattening $F$ has the effect of deleting this invalid component.
This corresponds to case~\ref{case:bigonPillowInvalid}.
\qedhere
\end{itemize}
\end{proof}

\begin{figure}[htbp]
\centering
	\begin{subfigure}[t]{0.22\textwidth}
	\centering
		\includegraphics[scale=1]{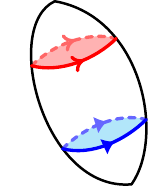}
	\caption{A bigon pillow that forms a (closed) cell decomposition of $S^3$.}
	\label{subfig:bigonPillowSphere}
	\end{subfigure}
	\hfill
	\begin{subfigure}[t]{0.22\textwidth}
	\centering
		\includegraphics[scale=1]{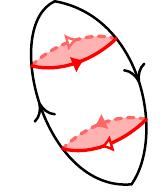}
	\caption{A bigon pillow that forms a (closed) cell decomposition of $\mathbb{R}P^3$.}
	\label{subfig:bigonPillowProjective}
	\end{subfigure}
	\hfill
	\begin{subfigure}[t]{0.22\textwidth}
	\centering
		\includegraphics[scale=1]{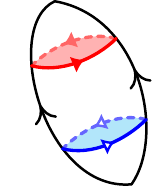}
	\caption{A bigon pillow that forms an ideal cell decomposition of $\mathbb{R}P^2\times[0,1]$.}
	\label{subfig:bigonPillowIdeal}
	\end{subfigure}
	\hfill
	\begin{subfigure}[t]{0.22\textwidth}
	\centering
		\includegraphics[scale=1]{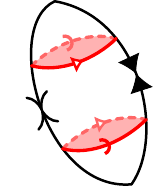}
	\caption{A bigon pillow that forms a component with two invalid edges.}
	\label{subfig:bigonPillowInvalid}
	\end{subfigure}
\caption{Cases where the two faces of a bigon pillow are glued to each other.}
\label{fig:bigonPillowGlued}
\end{figure}

\subsection{Flattening bigon faces}\label{subsec:bigonFaces}

We now study the effect of flattening a bigon face $F$.
As mentioned earlier, our main goal is to give a detailed analysis in the case where $F$ belongs to
a valid cell decomposition whose vertices are all either internal or ideal.
Our arguments only rely on the following properties:
\begin{enumerate}[label={(\alph*)}]
\item\label{cond:bigonInt}
$F$ is an internal face.
\item\label{cond:edgeInt}
Each edge incident to $F$ is internal.
\item\label{cond:vertIntIdeal}
Each vertex incident to $F$ is either internal or ideal.
\end{enumerate}
Provided these properties hold, our analysis will apply even when $F$ belongs to an invalid cell decomposition.

With this in mind, we assume throughout Section~\ref{subsec:bigonFaces} that $F$ is an internal bigon face.
However, for the sake of generality, we do \emph{not} assume that
conditions~\ref{cond:edgeInt} and~\ref{cond:vertIntIdeal} are satisfied;
instead, we carefully enumerate the cases where these conditions hold,
and for each such case we subsequently give a detailed description of the effect of flattening $F$.

We present our analysis in four parts.
First, in Section~\ref{subsubsec:userGuide}, we give a brief user guide for the reader seeking to apply our results.
Then, in Section~\ref{subsubsec:bigonFaceVertexLinks}, we make some preliminary observations by
examining how flattening $F$ interacts with the vertices incident to $F$.
Finally, we partition the main analysis into two broad cases that we handle separately in
Sections~\ref{subsubsec:bigonSinglePath} and~\ref{subsubsec:bigonSeparatePaths}.

Before we dive into the details, we make some general comments about our proof strategy, and we introduce some notation and terminology to support this.
One of the key ideas throughout our analysis is that, under our assumption that $F$ is internal,
flattening $F$ has the side-effect that we lose the face-gluing along $F$.
This means that flattening $F$ has the same result as the following two-step procedure:
\begin{enumerate}[label={(\arabic*)}]
\item\label{step:undoGluing} Undo the gluing along $F$, which yields two new boundary bigons $F^\dagger_0$ and $F^\dagger_1$.
\item\label{step:flattenBdryBigons} Flatten $F^\dagger_0$ and $F^\dagger_1$;
since these are boundary faces, flattening these faces has no side-effects (unlike the original face $F$).
\end{enumerate}
We will see that step~\ref{step:undoGluing} often corresponds to cutting along a properly embedded surface $S$,
and that step~\ref{step:flattenBdryBigons} often corresponds to filling the remnants of $S$,
so that the overall topological effect of flattening $F$ is often to decompose along $S$ (as defined in Section~\ref{subsec:decomposing}).
With this in mind, we introduce the following notation (also see
Figure~\ref{fig:notnBigonFaces}), which we will use throughout the rest of this section:

\begin{atomicnotn}\label{notn:bigonFaces}
As above, let $F$ be an internal bigon face in a (possibly invalid) cell decomposition $\mathcal{D}_0$.
Let $\mathcal{D}_1$ be the cell decomposition obtained by flattening $F$, and let $\varphi$ denote the associated flattening map.
For each $i\in\{0,1\}$, let $V_i$ denote the set of ideal and invalid vertices in $\mathcal{D}_i$,
and let $\mathcal{P}_i$ denote the truncated pseudomanifold of $\mathcal{D}_i$;
recall that $\mathcal{P}_i$ is obtained from $\mathcal{D}_i$ by truncating the vertices in $V_i$.

As in step~\ref{step:undoGluing} above, let $F^\dagger_0$ and $F^\dagger_1$ denote the two new boundary bigons
that we obtain after undoing the face-gluing along $F$, and let $\mathcal{D}^\dagger$ denote the cell decomposition that we obtain after undoing this gluing.
Let $g:\mathcal{D}^\dagger\to\mathcal{D}_0$ be the quotient map associated to the operation of regluing $F^\dagger_0$ and $F^\dagger_1$ to recover the original bigon face $F$.
\end{atomicnotn}

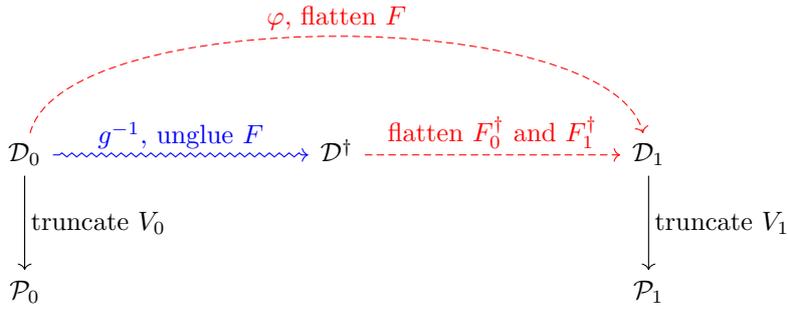
\begin{figure}[htbp]
\centering
\begin{tikzcd}[row sep=huge, column sep=huge,]
\mathcal{D}_0
\arrow[rrrr, "\text{$\varphi$, flatten $F$}", red, dashed, controls={+(0.5,2) and +(-0.5,2)}]
\arrow[rr, "\text{$g^{-1}$, unglue $F$}", blue, squiggly]
\arrow[d, "\text{truncate $V_0$}"]
&& \mathcal{D}^\dagger
\arrow[rr, "\text{flatten $F^\dagger_0$ and $F^\dagger_1$}", red, dashed]
&& \mathcal{D}_1
\arrow[d, "\text{truncate $V_1$}"] \\
\mathcal{P}_0
&&
&& \mathcal{P}_1
\end{tikzcd}
\caption{The protagonists introduced in Notation~\ref{notn:bigonFaces}.}
\label{fig:notnBigonFaces}
\end{figure}

We also introduce the following terminology, which will be useful not only
for flattening bigon faces, but also for proving our main theorem in Section~\ref{sec:genus}:

\begin{definitions}\label{defs:bigonPath}
Let $\mathcal{D}$ be a (possibly invalid) cell decomposition, and let $\mathcal{P}$ be the truncated pseudomanifold of $\mathcal{D}$.
Since $\mathcal{P}$ is obtained from $\mathcal{D}$ by truncating the ideal and invalid vertices of $\mathcal{D}$, we can view $\mathcal{P}$ as a subset of $\mathcal{D}$;
using this viewpoint, the \textbf{truncated bigon} associated to a bigon face $B$ in $\mathcal{D}$ is
given by $B\cap\mathcal{P}$ (see Figure~\ref{fig:truncatedBigonCases});
we will see that in many cases, the truncated bigon forms a properly embedded surface in $\mathcal{P}$.

For some positive integer $n$, consider an embedded curve $\gamma$ in $\mathcal{D}$ that:
\begin{itemize}[nosep]
\item starts at the midpoint of an edge $e_0$;
\item ends at the midpoint of an edge $e_n$ (possibly equal to $e_0$, to allow for the possibility that $\gamma$ is a closed curve); and
\item passes through the midpoints of a sequence $e_0,\ldots,e_n$ of edges,
such that for each $i\in\{0,\ldots,n-1\}$,
the edges $e_i$ and $e_{i+1}$ together bound a single bigon face $B_i$ that is bisected by $\gamma$.
\end{itemize}
This is illustrated in Figure~\ref{fig:bigonPathCurve}.
We call the union $\mathcal{U}:=B_0\cup\cdots\cup B_{n-1}$ a \textbf{bigon path} of \textbf{length} $n$ in $\mathcal{D}_0$,
and we call the edges $e_0$ and $e_n$ the \textbf{ends} of $\mathcal{U}$.
If the bigon faces $B_0,\ldots,B_{n-1}$ are all boundary, then we say that $\mathcal{U}$ is \textbf{boundary};
similarly, if $B_0,\ldots,B_{n-1}$ are all internal, then we say that $\mathcal{U}$ is \textbf{internal}.

For each $i\in\{0,\ldots,n-1\}$, let $S_i$ denote the truncated bigon associated to $B_i$.
We call the union $S_0\cup\cdots\cup S_{n-1}$ the \textbf{truncated bigon path} associated to $\mathcal{U}$;
similar to individual truncated bigons, truncated bigon paths often form properly embedded surfaces in $\mathcal{P}$.
\end{definitions}

\begin{figure}[htbp]
\centering
	\begin{subfigure}[t]{0.31\textwidth}
	\centering
		\includegraphics[scale=1]{FlattenBigon.pdf}
	\caption{The case where both vertices of $B$ are internal or boundary
	(either forming two distinct vertices, or identified to form a single such vertex).}
	\label{subfig:truncatedBigon0}
	\end{subfigure}
	\hfill
	\begin{subfigure}[t]{0.31\textwidth}
	\centering
		\includegraphics[scale=1]{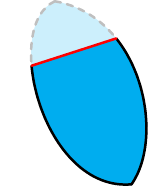}
	\caption{The case where one vertex of $B$ is ideal or invalid, while the other is internal or boundary.}
	\label{subfig:truncatedBigon1}
	\end{subfigure}
	\hfill
	\begin{subfigure}[t]{0.31\textwidth}
	\centering
		\includegraphics[scale=1]{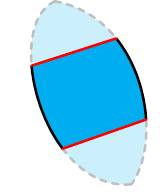}
	\caption{The case where both vertices of $B$ are ideal or invalid
	(either forming two distinct vertices, or identified to form a single such vertex).}
	\label{subfig:truncatedBigon2}
	\end{subfigure}
\caption{There are three possibilities for the truncated bigon associated to a bigon face $B$.
The portion of $B$ that lives outside the truncated bigon is indicated by dashed edges and faint shading.}
\label{fig:truncatedBigonCases}
\end{figure}

\begin{figure}[htbp]
\centering
	\includegraphics[scale=1]{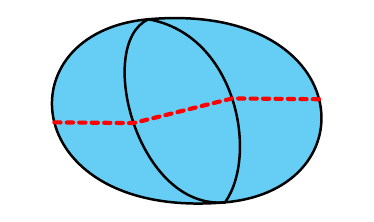}
\caption{Three bigon faces bisected by a curve (drawn as a dashed red line) passing through midpoints of edges.
These bigon faces together form a bigon path of length $3$.}
\label{fig:bigonPathCurve}
\end{figure}

We mentioned earlier that we divide our analysis into two cases that we handle separately in
Sections~\ref{subsubsec:bigonSinglePath} and~\ref{subsubsec:bigonSeparatePaths}.
We now have the terminology to describe these two cases.
Specifically, after ungluing $F$, the two new boundary bigons $F^\dagger_0$ and $F^\dagger_1$ could either:
\begin{itemize}
\item share at least one common edge, so that they together form a single boundary bigon path of length two; or
\item have no common edges, in which case they form two separate boundary bigon paths of length one.
\end{itemize}
There is no technical reason for dividing our analysis according to these two cases;
we make this choice simply to help organise our analysis into smaller, more manageable pieces.

\subsubsection{User guide}\label{subsubsec:userGuide}

We split the effects of flattening $F$ into several parts:
\begin{itemize}
\item The effect on the vertices incident to $F$ is described in Claim~\ref{claim:flattenBigonVertexLink}.
\item The effect on the edges incident to $F$ is described in Claims~\ref{claim:bigonSinglePath} and~\ref{claim:bigonSeparatePaths}.
\item The effect on the truncated pseudomanifold $\mathcal{P}_0$ is described in
Claims~\ref{claim:bigonSinglePath:S2:glueD2},~\ref{claim:bigonSinglePath:S2:glueRP2} and~\ref{claim:bigonSinglePath:S2:glueS2},
and in Claims~\ref{claim:bigonSeparatePaths:glueS2} and~\ref{claim:bigonSeparatePaths:glueRP2}.
\end{itemize}
Claims~\ref{claim:bigonSinglePath},~\ref{claim:bigonSinglePath:S2:glueD2},~\ref{claim:bigonSinglePath:S2:glueRP2}
and~\ref{claim:bigonSinglePath:S2:glueS2} all deal with the case where $F^\dagger_0$ and $F^\dagger_1$
form a single boundary bigon path, so they can be found in Section~\ref{subsubsec:bigonSinglePath};
on the other hand, Claims~\ref{claim:bigonSeparatePaths},~\ref{claim:bigonSeparatePaths:glueS2}
and~\ref{claim:bigonSeparatePaths:glueRP2} all deal with the case where $F^\dagger_0$ and $F^\dagger_1$
form two separate boundary bigon paths, so they can be found in Section~\ref{subsubsec:bigonSeparatePaths}.
The intended way to use all these results is to begin by referring to Claims~\ref{claim:bigonSinglePath} and~\ref{claim:bigonSeparatePaths},
as these two overarching claims will indicate which of the other claims are relevant for any given application.

The only other result that we prove is Claim~\ref{claim:flattenBigonVertexCone}.
This is a useful tool for our proofs in Sections~\ref{subsubsec:bigonSinglePath} and~\ref{subsubsec:bigonSeparatePaths},
but it is otherwise not a crucial part of our description of the effects of flattening $F$.
Having said this, Claim~\ref{claim:flattenBigonVertexCone} might be useful for the reader seeking to
extend our analysis of flattening $F$ to the cases that we do not study in detail.

\subsubsection{Interaction with vertices}\label{subsubsec:bigonFaceVertexLinks}

Let $v$ denote a vertex incident to $F$, and consider a small regular neighbourhood $N$ of $v$.
To describe how flattening $F$ interacts with $v$, we will find it useful to view $N$ as a cone over the link $L$ of $v$;
that is, we view $N$ as a union of lines, with each point in $L$ being joined to $v$
by one such line, and with any two such lines intersecting only at the vertex $v$.
Under this viewpoint, any subset $S$ of $L$ defines a subset $C_S$ of $N$ consisting of the lines joining $S$ to $v$
(for example, see Figure~\ref{fig:coneOverCurve});
we will call $C_S$ the \textbf{$v$-cone} over $S$.
We will use this notion of $v$-cones to prove two claims:
\begin{itemize}
\item In Claim~\ref{claim:flattenBigonVertexLink}, we describe how flattening $F$ changes the vertex $v$.
\item In Claim~\ref{claim:flattenBigonVertexCone}, we give conditions under which we can, in some sense, ``push $F$ away from $v$'';
we will give a more precise formulation of this later.
Roughly, the purpose of this is that it gives us a unified method to deal with some of the more inconvenient ways in which flattening $F$ interacts with $v$;
this will become clearer when we see Claim~\ref{claim:flattenBigonVertexCone} in action in
Sections~\ref{subsubsec:bigonSinglePath} and~\ref{subsubsec:bigonSeparatePaths}.
\end{itemize}

\begin{figure}[htbp]
\centering
	\includegraphics[scale=1]{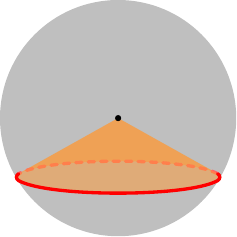}
\caption{A simple example of a $v$-cone (orange) over a subset $S$ of a vertex link $L$.
Here, $L$ is a $2$-sphere, and $S$ is an embedded closed curve (red) in $L$.}
\label{fig:coneOverCurve}
\end{figure}

Since we are interested in the effect of flattening $F$, we devote particular attention to the subset of $L$ given by $F\cap L$.
Assuming that each edge incident to $F$ is internal, we have the following possibilities (see Figure~\ref{fig:curveConeCases}):
\begin{itemize}
\item Suppose the edges of $F$ are not identified, so that $F$ forms a disc.
In this case, $F\cap L$ consists of either one or two arcs:
	\begin{itemize}
	\item If the two vertices of $F$ form two distinct vertices in $\mathcal{D}_0$,
	then $v$ is one of these two vertices,
	and $F\cap L$ consists of a single arc in $L$.
	\item If the two vertices of $F$ are identified to form a single vertex in $\mathcal{D}_0$,
	then $F\cap L$ consists of two disjoint arcs in $L$.
	\end{itemize}
\item Suppose the edges of $F$ are identified to form a single edge $e$, and suppose this identification causes $F$ to form a $2$-sphere.
In this case, $F\cap L$ consists of either one or two closed curves:
	\begin{itemize}
	\item If the two vertices of $F$ form two distinct vertices in $\mathcal{D}_0$,
	then $v$ is one of these two vertices,
	and $F\cap L$ consists of a single closed curve in $L$.
	\item If the two vertices of $F$ are identified to form a single vertex in $\mathcal{D}_0$,
	then $F\cap L$ consists of two disjoint closed curves in $L$.
	\end{itemize}
\item Suppose the edges of $F$ are identified to form a single edge $e$, and suppose this identification causes $F$ to form a projective plane.
In this case, the two vertices of $F$ are identified to form a single vertex in $\mathcal{D}_0$,
and $F\cap L$ consists of a single closed curve in $L$.
\end{itemize}

\begin{figure}[htbp]
\centering
	\begin{subfigure}[t]{0.29\textwidth}
	\centering
		\includegraphics[scale=1]{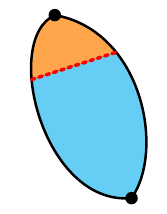}
	\caption{When $F$ forms a disc, $\gamma$ forms an arc.}
	\label{subfig:curveConeD2}
	\end{subfigure}
	\hfill
	\begin{subfigure}[t]{0.29\textwidth}
	\centering
		\includegraphics[scale=1]{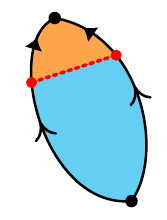}
	\caption{When $F$ forms a $2$-sphere, $\gamma$ forms a closed curve.}
	\label{subfig:curveConeS2}
	\end{subfigure}
	\hfill
	\begin{subfigure}[t]{0.29\textwidth}
	\centering
		\includegraphics[scale=1]{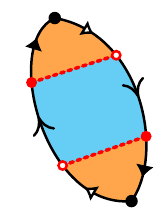}
	\caption{When $F$ forms a projective plane, $\gamma$ forms a closed curve.}
	\label{subfig:curveConeRP2}
	\end{subfigure}
\caption{The three possibilities for a component $\gamma$ (dotted red) of $F\cap L$.
In each case, the $v$-cone (orange) over $\gamma$ forms part of $F\cap N$.}
\label{fig:curveConeCases}
\end{figure}

In each of the above cases, observe that the $v$-cone over $F\cap L$ coincides exactly with $F\cap N$.
Intuitively, this means that flattening $F$ changes $N$ in a way that ``respects the cone structure''.
This idea allows us to give a fairly straightforward description of how flattening $F$ affects the vertex $v$:

\begin{atomic}\label{claim:flattenBigonVertexLink}
Assume that each edge incident to $F$ is internal.
Let $v$ be a vertex incident to $F$, and let $L$ denote the link of $v$.
We have the following possibilities:
\begin{enumerate}[nosep,label={(\alph*)}]
\item\label{case:bigonVertexLinkArc}
If the edges of $F$ are not identified, then $\varphi(v)$ consists of a single vertex whose link is topologically equivalent to $L$.
\item\label{case:bigonVertexLinkClosed}
If the edges of $F$ are identified, then $F\cap L$ consists of either one or two closed curves in $L$.
Let $L'_0,\ldots,L'_{k-1}$ denote the components of the surface obtained by decomposing $L$ along the curves in $F\cap L$;
there could be up to three such components (that is, we have $1\leqslant k\leqslant3$).
After flattening $F$, the image $\varphi(v)$ consists of $k$ vertices $v'_0,\ldots,v'_{k-1}$
such that for each $i\in\{0,\ldots,k-1\}$, the vertex $v'_i$ has link $L'_i$.
\end{enumerate}
\end{atomic}

\begin{proof}
As above, let $N$ denote a small regular neighbourhood of $v$, and view $N$ as the $v$-cone over the vertex link $L$.
We first consider the case where the two edges of $F$ are not identified, so that $F\cap L$ consists of either one or two arcs in $L$.
For each such arc $\gamma$, flattening $F$ has the effect of collapsing $\gamma$ to a single point $p_\gamma$, which leaves $L$ topologically unchanged;
the corresponding effect on $N$ is to flatten the $v$-cone over $\gamma$ to a single line joining $p_\gamma$ to $v$,
which means that we can continue to view $N$ as the $v$-cone over $L$.
As a result, we see that $\varphi(v)$ consists of a single vertex whose link is topologically equivalent to $L$.
This proves case~\ref{case:bigonVertexLinkArc}.

In the case where the edges of $F$ \emph{are} identified, recall that $F\cap L$ consists of either one or two closed curves in $L$.
This time, we study the effect of flattening $F$ by first ungluing $F$, and then flattening $F^\dagger_0$ and $F^\dagger_1$:
\begin{description}[font=\normalfont]
\item[\ref{step:undoGluing}]
Ungluing $F$ changes $L$ by cutting along the curves in $F\cap L$.
Since $F\cap L$ could have up to two components, each of which could possibly form a \emph{separating} curve in $L$,
cutting along $F\cap L$ could split $L$ into up to three components;
let $k$ be the number of such components, and denote these components by $L'_0,\ldots,L'_{k-1}$.
The corresponding change to $N$ is to cut along the $v$-cone over $F\cap L$, which has the following effects:
	\begin{itemize}
	\item $v$ gets split into $k$ new vertices $v'_0,\ldots,v'_{k-1}$; and
	\item $N$ gets split into $k$ components $N'_0,\ldots,N'_{k-1}$ such that
	for each $i\in\{0,\ldots,k-1\}$, $N'_i$ forms the $v'_i$-cone over $L'_i$.
	\end{itemize}
At this stage, the vertices $v'_i$ and the surfaces $L'_i$ are not yet the same as the corresponding objects that appear in the claim statement.
However, this situation will soon be fixed when we flatten $F^\dagger_0$ and $F^\dagger_1$, which will have the consequence of modifying the vertices $v'_i$ and the surfaces $L'_i$;
to keep notation as simple as possible, we will continue to use the same notation for these objects even after they are modified.
Likewise, the neighbourhoods $N'_i$ will also be modified, but we will continue to denote the modified neighbourhoods by $N'_i$.
\item[\ref{step:flattenBdryBigons}]
For each $i\in\{0,\ldots,k-1\}$, flattening $F^\dagger_0$ and $F^\dagger_1$ modifies $L'_i$ by collapsing
each remnant $\gamma$ of $F\cap L$ to a single point $p_\gamma$, which is topologically equivalent to filling $\gamma$ with a disc;
the corresponding effect on $N'_i$ is to flatten the $v'_i$-cone over $\gamma$ to a single line joining $p_\gamma$ to $v'_i$,
which means that we can continue to view $N'_i$ as the $v'_i$-cone over $L'_i$.
The end result of all this is that $\varphi(v)$ consists precisely of the (now modified) vertices $v'_0,\ldots,v'_{k-1}$,
and that the links of these vertices are given by the (now modified) surfaces $L'_0,\ldots,L'_{k-1}$, respectively.
We also note that, topologically, $L'_0,\ldots,L'_{k-1}$ form the components of the surface obtained by decomposing $L$ along $F\cap L$.
\end{description}
This proves case~\ref{case:bigonVertexLinkClosed}.
\end{proof}

We now turn our attention to the idea of ``pushing $F$ away from $v$'';
we will build up to this idea in a slightly roundabout way.
Assume that the two edges of $F$ are identified to form a single internal edge.
As observed earlier, this means that $F\cap L$ consists of either one or two closed curves in $L$.
Suppose a component $\gamma$ of $F\cap L$ forms a separating curve that bounds a disc in $L$.
Under these conditions, rather than beginning the process of flattening $F$ by ungluing the entirety of $F$ all at once,
we will find it useful to follow a more fine-grained procedure for flattening $F$
(see Figure~\ref{fig:flattenConeCases}):
\begin{enumerate}[label={(\roman*)}]
\item\label{coneStep:cut}
Cut along the subset of $F$ given by the $v$-cone $C_\gamma$ over $\gamma$.
Since $\gamma$ is a separating curve in $L$, this has the following effects:
	\begin{itemize}
	\item $v$ gets split into two new vertices $v'_0$ and $v'_1$; and
	\item $C_\gamma$ gets split into two remnants $C^\dagger_0$ and $C^\dagger_1$ such that
	for each $i\in\{0,1\}$, $C^\dagger_i$ forms the $v'_i$-cone over $\gamma$.
	\end{itemize}
We postpone ungluing or cutting along $F-C_\gamma$ (i.e., the rest of $F$) until later;
as a result, the curve $\gamma$ does not yet fall apart into two pieces because it is still ``held together'' by $F-C_\gamma$.
\item\label{coneStep:flattenCone}
Flatten $\gamma$ to a single point $p_\gamma$, and for each $i\in\{0,1\}$ flatten
the remnant $C^\dagger_i$ to a single line $\alpha_i$ joining $p_\gamma$ to $v'_i$.
Intuitively, the lines $\alpha_0$ and $\alpha_1$ will eventually form segments of the edges in $\varphi(F)$.
Viewing $p_\gamma$ as a temporary vertex, this step causes $F-C_\gamma$ to become a new bigon $F'$.
\item\label{coneStep:flattenBigon}
Treating $F'$ as if it were an internal bigon face in a cell decomposition, flatten $F'$ by first cutting along it, and then flattening each of its remnants.
Since $\gamma$ was originally a separating curve in $L$, this step splits the temporary vertex $p_\gamma$ into a pair of new points.
We no longer view these new points as vertices (which is why we called $p_\gamma$ a ``temporary'' vertex);
instead, each of these new points will occur in the interior of an edge in $\varphi(F)$.
\end{enumerate}
Together, we refer to steps~\ref{coneStep:cut} and~\ref{coneStep:flattenCone} as the operation of \textbf{flattening} $C_\gamma$.
We emphasise that after performing this operation, the intermediate object that we obtain from $\mathcal{D}_0$ might not be a cell decomposition anymore;
however, this problem is only temporary, since we will recover a cell decomposition once we complete step~\ref{coneStep:flattenBigon}.

\begin{figure}[htbp]
\raggedleft
	\begin{subfigure}[t]{0.6\textwidth}
	\centering
		\begin{tikzpicture}

		\node[inner sep=0pt] (before) at (0,0) {
			\includegraphics[scale=1]{FlattenS2ConeBefore.pdf}
		};
		\node[inner sep=0pt] (cut) at (3.55,0) {
			\includegraphics[scale=1]{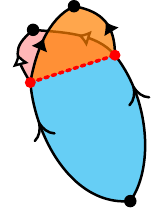}
		};
		\node[inner sep=0pt] (after) at (6.75,0) {
			\includegraphics[scale=1]{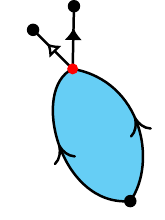}
		};
		\draw[thick, line cap=round, -Stealth] ($(before.east)+(-0.2,0)$) --++ (1.4,0);
		\draw[thick, line cap=round, -Stealth] ($(cut.east)+(-0.2,0)$) --++ (1.4,0);

		\end{tikzpicture}
	\caption{The case where $F$ forms a $2$-sphere.
	In this case, the new bigon $F'$ also forms a $2$-sphere, and one of its two vertices is given by $p_\gamma$.}
	\label{subfig:flattenConeS2}
	\end{subfigure}
	\hfill
	\bigskip
	\hfill
	\begin{subfigure}[t]{0.6\textwidth}
	\centering
		\begin{tikzpicture}

		\node[inner sep=0pt] (before) at (0,0) {
			\includegraphics[scale=1]{FlattenRP2ConeBefore.pdf}
		};
		\node[inner sep=0pt] (cut) at (3.55,0) {
			\includegraphics[scale=1]{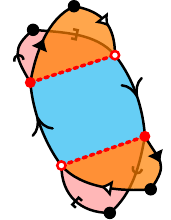}
		};
		\node[inner sep=0pt] (after) at (6.75,0) {
			\includegraphics[scale=1]{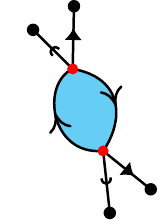}
		};
		\draw[thick, line cap=round, -Stealth] ($(before.east)+(-0.25,0)$) --++ (1.4,0);
		\draw[thick, line cap=round, -Stealth] ($(cut.east)+(-0.3,0)$) --++ (1.4,0);

		\end{tikzpicture}
	\caption{The case where $F$ forms a projective plane.
	In this case, the new bigon $F'$ also forms a projective plane, and its two vertices are identified together to form $p_\gamma$.}
	\label{subfig:flattenConeRP2}
	\end{subfigure}
\caption{Flattening $C_\gamma$ (orange) by first cutting along it, and then flattening the two remnants (orange and pink).
This collapses $\gamma$ (dotted red) to a point $p_\gamma$ that we temporarily view as a vertex;
it also causes $F-C_\gamma$ (blue) to become a new bigon $F'$.
(These illustrations do not accurately reflect how everything is embedded in $\mathcal{D}_0$.)}
\label{fig:flattenConeCases}
\end{figure}

In describing the operation of flattening $C_\gamma$, we used the assumption that
$\gamma$ is a separating curve, but made no mention of the assumption that $\gamma$ bounds a disc in $L$.
The purpose of the latter assumption is that it allows us to show that flattening $C_\gamma$ leaves the topology of $\mathcal{D}_0$ unchanged,
which means that flattening $F$ is topologically equivalent to the operation of flattening the new bigon $F'$.
Moreover, we can view $F'$ as the bigon that results from ``pushing $F$ away from $v$'':

\begin{atomic}\label{claim:flattenBigonVertexCone}
Assume that each edge incident to $F$ is internal.
Let $v$ be a vertex incident to $F$, and let $L$ denote the link of $v$.
Suppose a component $\gamma$ of $F\cap L$ forms a separating curve that bounds a disc $E$ in $L$, and suppose the interior of $E$ is disjoint from $F$.
Consider the subset of $F$ given by the $v$-cone $C_\gamma$ over $\gamma$.
Flattening $C_\gamma$ creates a new internal vertex without changing the topology of $\mathcal{D}_0$,
and reduces the operation of flattening $F$ to the operation of flattening a new bigon $F'$ that is:
\begin{itemize}[nosep]
\item topologically equivalent to a bigon obtained from $F$ by an isotopy that takes $C_\gamma$ to $E$; and
\item incident to a temporary \emph{internal} vertex given by the point $p_\gamma$ that results from flattening $\gamma$.
\end{itemize}
\end{atomic}

\begin{proof}
To see how flattening $C_\gamma$ affects $\mathcal{D}_0$ topologically, we first claim that since $\gamma$ bounds a disc $E$ in $L$,
we can find a $3$-ball $\mathcal{B}$ such that $C_\gamma-v$ lies in the interior of $\mathcal{B}$.
If $v$ were internal or boundary, we could simply take $\mathcal{B}$ to be a small regular neighbourhood of the $v$-cone $C_E$ over $E$.
However, to account for the possibility that $v$ is ideal or invalid, we instead construct $\mathcal{B}$ as follows:
\begin{enumerate}[label={(\alph*)}]
\item Consider a regular neighbourhood $N^\ast$ of $v$ that is ``large enough'' so that $C_E$ lies entirely in the interior of $N^\ast$.
\item Slightly isotope the disc $E$ so that it lies in the frontier of $N^\ast$,
and then enlarge this disc slightly so that the $v$-cone over this disc forms the desired $3$-ball $\mathcal{B}$
(see Figures~\ref{subfig:flattenConeBefore} and~\ref{subfig:flattenConeSchematicBefore}).
\end{enumerate}

\begin{figure}[htbp]
\centering
	\begin{subfigure}[t]{0.3\textwidth}
	\centering
		\begin{tikzpicture}

		\node[inner sep=0pt] (pic) at (0,0) {
			\includegraphics[scale=1]{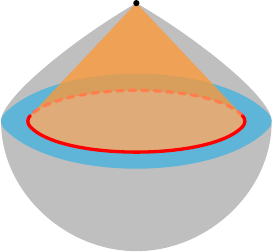}
		};
		\node[right,inner sep=2pt] at (pic.north) {$v$};

		\end{tikzpicture}
	\caption{The $3$-ball $\mathcal{B}$ (grey) contains
	the entirety of the $v$-cone $C_\gamma$ (orange), as well as a portion of $F-C_\gamma$ (blue).}
	\label{subfig:flattenConeBefore}
	\end{subfigure}
	\hfill
	\begin{subfigure}[t]{0.3\textwidth}
	\centering
		\begin{tikzpicture}

		\node[inner sep=0pt] (pic) at (0,0) {
			\includegraphics[scale=1]{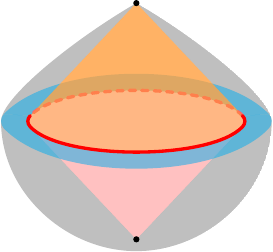}
		};
		\node[right,inner sep=2pt] at (pic.north) {$v'_0$};
		\node[above right,inner sep=2pt] at (pic.south) {$v'_1$};

		\end{tikzpicture}
	\caption{Cutting along the $v$-cone $C_\gamma$ yields two remnants (orange and pink),
	and creates a void.}
	\label{subfig:flattenConeCut}
	\end{subfigure}
	\hfill
	\begin{subfigure}[t]{0.3\textwidth}
	\centering
		\begin{tikzpicture}

		\node[inner sep=0pt] (pic) at (0,0) {
			\includegraphics[scale=1]{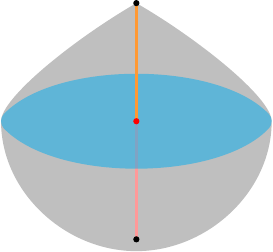}
		};
		\node[right,inner sep=2pt] at (pic.north) {$v'_0$};
		\node[above right,inner sep=2pt] at (pic.south) {$v'_1$};
		\node[right,inner sep=2pt,red] at (pic.center) {$p_\gamma$};

		\end{tikzpicture}
	\caption{Flattening the remnants of $C_\gamma$ fills the void back in, so we recover a $3$-ball.}
	\label{subfig:flattenConeAfter}
	\end{subfigure}
\caption{Since $\gamma$ (red) bounds a disc in $L$, flattening $C_\gamma$ has no topological effect on $\mathcal{D}_0$.
The intersection of the $3$-ball $\mathcal{B}$ (grey) with the ``unflattened'' part of $F$ is shaded blue.}
\label{fig:coneCutFlatten}
\end{figure}
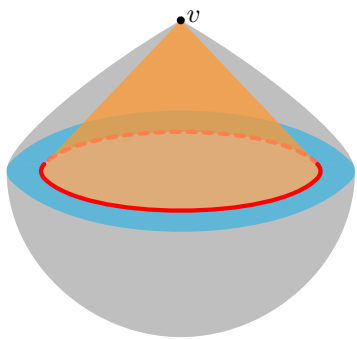
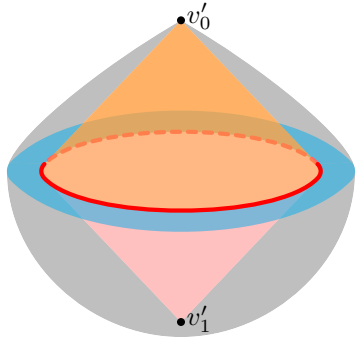
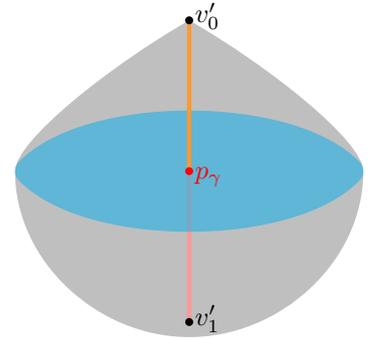

\begin{figure}[htbp]
\centering
	\begin{subfigure}[t]{0.3\textwidth}
	\centering
		\begin{tikzpicture}

		\node[inner sep=0pt] (pic) at (0,0) {
			\includegraphics[scale=1]{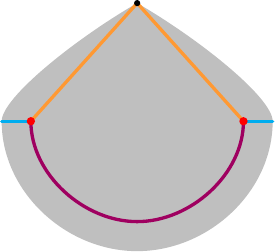}
		};
		\node[right,inner sep=2pt] at (pic.north) {$v$};
		\node[above,inner sep=2pt,purple!50!violet] at ($(pic.south)+(90:0.1)$) {$E$};

		\end{tikzpicture}
	\caption{Cross-section of Figure~\ref{subfig:flattenConeBefore}.}
	\label{subfig:flattenConeSchematicBefore}
	\end{subfigure}
	\hfill
	\begin{subfigure}[t]{0.3\textwidth}
	\centering
		\begin{tikzpicture}

		\node[inner sep=0pt] (pic) at (0,0) {
			\includegraphics[scale=1]{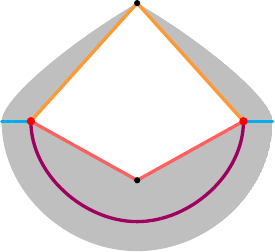}
		};
		\node[right,inner sep=2pt] at (pic.north) {$v'_0$};
		\node[above right,inner sep=2pt] at ($(pic.south)+(90:0.85)$) {$v'_1$};
		\node[above,inner sep=2pt,purple!50!violet] at ($(pic.south)+(90:0.1)$) {$E$};

		\end{tikzpicture}
	\caption{Cross-section of Figure~\ref{subfig:flattenConeCut}.}
	\label{subfig:flattenConeSchematicCut}
	\end{subfigure}
	\hfill
	\begin{subfigure}[t]{0.3\textwidth}
	\centering
		\begin{tikzpicture}

		\node[inner sep=0pt] (pic) at (0,0) {
			\includegraphics[scale=1]{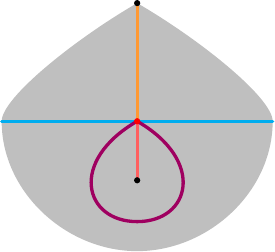}
		};
		\node[right,inner sep=2pt] at (pic.north) {$v'_0$};
		\node[above right,inner sep=2pt] at ($(pic.south)+(90:0.85)$) {$v'_1$};
		\node[above right,inner sep=2pt,red] at (pic.center) {$p_\gamma$};
		\node[above,inner sep=2pt,purple!50!violet] at ($(pic.south)$) {$E/\partial{E}$};

		\end{tikzpicture}
	\caption{Cross-section of Figure~\ref{subfig:flattenConeAfter}.}
	\label{subfig:flattenConeSchematicAfter}
	\end{subfigure}
\caption{Schematic cross-sections of the $3$-dimensional pictures in Figure~\ref{fig:coneCutFlatten}.
Here we also include the disc $E$ (purple), which is not shown in the $3$-dimensional pictures;
the vertex $v'_1$ is repositioned slightly to accommodate this inclusion.}
\label{fig:coneCutFlattenSchematic}
\end{figure}

The operation of flattening $C_\gamma$ leaves everything outside of $\mathcal{B}$ untouched,
so we just need to understand how this operation affects $\mathcal{B}$ topologically.
We follow steps~\ref{coneStep:cut} and~\ref{coneStep:flattenCone} from above:
\begin{description}[font=\normalfont]
\item[\ref{coneStep:cut}]
As illustrated in Figures~\ref{subfig:flattenConeCut} and~\ref{subfig:flattenConeSchematicCut}, cutting along $C_\gamma$ has the following effects:
	\begin{itemize}
	\item The vertex $v$ gets split into two new vertices $v'_0$ and $v'_1$.
	One of the new vertices, say $v'_0$, has link given by $L$ minus a disc.
	The other new vertex $v'_1$ has link given by the disc $E$, as illustrated in Figure~\ref{subfig:flattenConeSchematicCut}.
	\item The $v$-cone $C_\gamma$ gets split into two remnants $C^\dagger_0$ and $C^\dagger_1$
	such that for each $i\in\{0,1\}$, $C^\dagger_i$ forms the $v'_i$-cone over $\gamma$.
	These two remnants bound a newly-created void inside our $3$-ball $\mathcal{B}$.
	\end{itemize}
\item[\ref{coneStep:flattenCone}]
As illustrated in Figures~\ref{subfig:flattenConeAfter} and~\ref{subfig:flattenConeSchematicAfter}, flattening $C^\dagger_0$ and $C^\dagger_1$ has the following effects:
	\begin{itemize}
	\item The link of $v'_0$ gets ``closed up'' so that it becomes topologically equivalent to $L$.
	Thus, we can equate $v'_0$ with the original vertex $v$.
	\item The link of $v'_1$ gets ``closed up'' to become the $2$-sphere corresponding to $E/\partial{E}$,
	as illustrated in Figure~\ref{subfig:flattenConeSchematicAfter}.
	Thus, we can view $v'_1$ as a newly-created internal vertex.
	\item The void that we created in the previous step gets filled in, so that we once again have a $3$-ball $\mathcal{B}'$.
	\item The curve $\gamma$ gets flattened to a single point $p_\gamma$ that we temporarily view as a vertex.
	(Recall that $p_\gamma$ is only a ``temporary'' vertex because after performing step~\ref{coneStep:flattenBigon},
	$p_\gamma$ gets split into two new points that we no longer view as vertices.)
	Since $p_\gamma$ lies in the interior of the $3$-ball $\mathcal{B}'$, we can think of $p_\gamma$ as an \emph{internal} vertex.
	\end{itemize}
\end{description}
Topologically, all we have done is replaced the $3$-ball $\mathcal{B}$ with
another $3$-ball $\mathcal{B}'$, so we have not changed the topology of $\mathcal{D}_0$.

To finish this proof, consider $F-C_\gamma$; this is the part of $F$ that is being left ``unflattened''.
Recall that after flattening $C_\gamma$, this unflattened part of $F$ becomes a new bigon $F'$,
and that the operation of flattening $F$ is reduced to the operation of flattening $F'$.
Topologically, observe that $F'$ is equivalent to a bigon obtained from $F$ by an isotopy that replaces $C_\gamma$ with the disc $E$;
this can be seen by equating the vertices $v$ and $v'_0$, and then comparing how $F$ intersects the grey $3$-ball in
Figure~\ref{subfig:flattenConeBefore} with how $F'$ intersects the grey $3$-ball in Figure~\ref{subfig:flattenConeAfter}.
\end{proof}

\subsubsection{The case where ungluing gives a single boundary bigon path}\label{subsubsec:bigonSinglePath}

We are now ready to present the main analysis of the effect of flattening $F$.
We first consider the case where $F^\dagger_0$ and $F^\dagger_1$ together form a single boundary bigon path $\mathcal{F}^\dagger$.
Depending on how the ends of $\mathcal{F}^\dagger$ are identified (if at all),
$\mathcal{F}^\dagger$ could form a $2$-sphere, projective plane or disc in the boundary of $\mathcal{D}^\dagger$.
We refine this list of cases as follows:

\begin{atomic}\label{claim:bigonSinglePath}
If $F^\dagger_0$ and $F^\dagger_1$ together form a single boundary bigon path $\mathcal{F}^\dagger$, then one of the following holds:
\begin{itemize}[nosep]
\item The boundary bigon path $\mathcal{F}^\dagger$ forms a $2$-sphere,
in which case the result $\varphi(F)$ of flattening $F$ is a single internal edge in $\mathcal{D}_1$.
Moreover, letting $e^\dagger_0$ and $e^\dagger_1$ denote the edges incident to $\mathcal{F}^\dagger$, we have four cases depending on the behaviour of the gluing map $g$
\emph{(see Figure~\ref{fig:bigonSinglePath:S2})}:
	\begin{enumerate}[nosep,label={(\arabic*)}]
	\item\label{case:bigonSinglePath:S2:glueD2}
	The map $g$ realises an orientation-reversing gluing of $F^\dagger_0$ and $F^\dagger_1$ such that for each $i\in\{0,1\}$,
	the edge $e^\dagger_i$ gets identified with itself to form an internal edge of $\mathcal{D}_0$.
	In this case, $F$ forms a \textbf{disc} in $\mathcal{D}_0$.\\
	\emph{(See Claim~\ref{claim:bigonSinglePath:S2:glueD2} for
	details about the effect of flattening $F$ in this case.)}
	\item\label{case:bigonSinglePath:S2:glueRP2}
	The map $g$ realises an orientation-reversing gluing of $F^\dagger_0$ and $F^\dagger_1$ that causes
	$e^\dagger_0$ and $e^\dagger_1$ to be identified together to form a single internal edge of $\mathcal{D}_0$.
	In this case, $F$ forms a \textbf{projective plane} in $\mathcal{D}_0$.\\
	\emph{(See Claim~\ref{claim:bigonSinglePath:S2:glueRP2} for
	details about the effect of flattening $F$ in this case.)}
	\item\label{case:bigonSinglePath:S2:glueInvalid}
	The map $g$ realises an orientation-preserving gluing of $F^\dagger_0$ and $F^\dagger_1$ such that for each $i\in\{0,1\}$,
	the edge $e^\dagger_i$ gets identified with itself in reverse to form an invalid edge of $\mathcal{D}_0$.
	\item\label{case:bigonSinglePath:S2:glueS2}
	The map $g$ realises an orientation-preserving gluing of $F^\dagger_0$ and $F^\dagger_1$ that causes
	$e^\dagger_0$ and $e^\dagger_1$ to be identified together to form a single internal edge of $\mathcal{D}_0$.
	In this case, $F$ forms a \textbf{\textup{2}-sphere} in $\mathcal{D}_0$.\\
	\emph{(See Claim~\ref{claim:bigonSinglePath:S2:glueS2} for
	details about the effect of flattening $F$ in this case.)}
	\end{enumerate}
\item The boundary bigon path $\mathcal{F}^\dagger$ forms a projective plane, in which case $F$ is incident to an invalid edge in $\mathcal{D}_0$.
\item The boundary bigon path $\mathcal{F}^\dagger$ forms a disc, in which case $F$ is incident to a boundary edge in $\mathcal{D}_0$.
\end{itemize}
\end{atomic}

\begin{figure}[htbp]
\centering
	\begin{subfigure}[t]{0.45\textwidth}
	\centering
		\begin{tikzpicture}

		\node[inner sep=0pt] (before) at (0,0) {
			\includegraphics[scale=1]{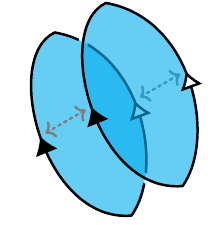}
		};
		\node[inner sep=0pt] (after) at (4,0) {
			\includegraphics[scale=1]{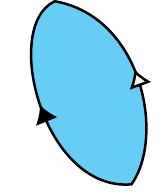}
		};
		\draw[thick, line cap=round, -Stealth] ($(before.east)+(-0.2,0)$) -- ($(after.west)+(0.5,0)$);

		\end{tikzpicture}
	\caption{A gluing that results in $F$ forming a disc.}
	\label{subfig:bigonSinglePath:S2:glueD2}
	\end{subfigure}
	\hfill
	\begin{subfigure}[t]{0.45\textwidth}
	\centering
		\begin{tikzpicture}

		\node[inner sep=0pt] (before) at (0,0) {
			\includegraphics[scale=1]{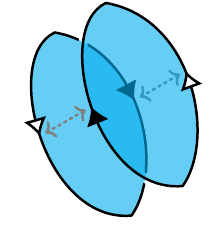}
		};
		\node[inner sep=0pt] (after) at (4,0) {
			\includegraphics[scale=1]{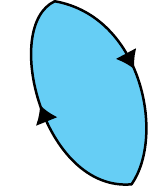}
		};
		\draw[thick, line cap=round, -Stealth] ($(before.east)+(-0.2,0)$) -- ($(after.west)+(0.5,0)$);

		\end{tikzpicture}
	\caption{A gluing that results in $F$ forming a projective plane.}
	\label{subfig:bigonSinglePath:S2:glueRP2}
	\end{subfigure}
	\bigskip
	\begin{subfigure}[t]{0.45\textwidth}
	\centering
		\begin{tikzpicture}

		\node[inner sep=0pt] (before) at (0,0) {
			\includegraphics[scale=1]{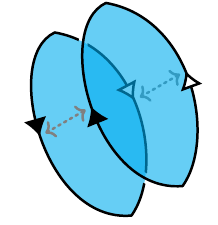}
		};
		\node[inner sep=0pt] (after) at (4,0) {
			\includegraphics[scale=1]{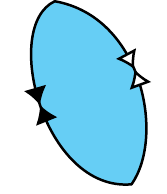}
		};
		\draw[thick, line cap=round, -Stealth] ($(before.east)+(-0.2,0)$) -- ($(after.west)+(0.4,0)$);

		\end{tikzpicture}
	\caption{A gluing that results in $F$ being incident to two invalid edges.}
	\label{subfig:bigonSinglePath:S2:glueInvalid}
	\end{subfigure}
	\hfill
	\begin{subfigure}[t]{0.45\textwidth}
	\centering
		\begin{tikzpicture}

		\node[inner sep=0pt] (before) at (0,0) {
			\includegraphics[scale=1]{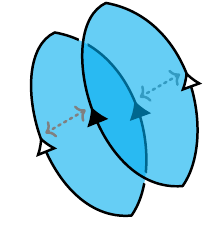}
		};
		\node[inner sep=0pt] (after) at (4,0) {
			\includegraphics[scale=1]{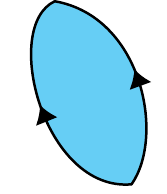}
		};
		\draw[thick, line cap=round, -Stealth] ($(before.east)+(-0.2,0)$) -- ($(after.west)+(0.5,0)$);

		\end{tikzpicture}
	\caption{A gluing that results in $F$ forming a $2$-sphere.}
	\label{subfig:bigonSinglePath:S2:glueS2}
	\end{subfigure}
\caption{The four ways to glue $F^\dagger_0$ and $F^\dagger_1$ together when $\mathcal{F}^\dagger$ forms a $2$-sphere.}
\label{fig:bigonSinglePath:S2}
\end{figure}
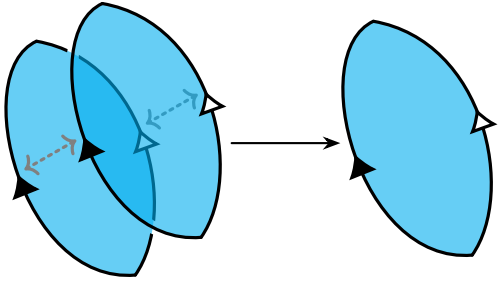
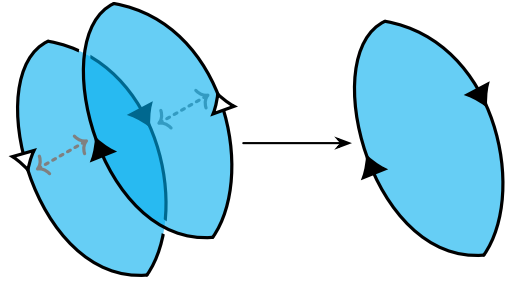
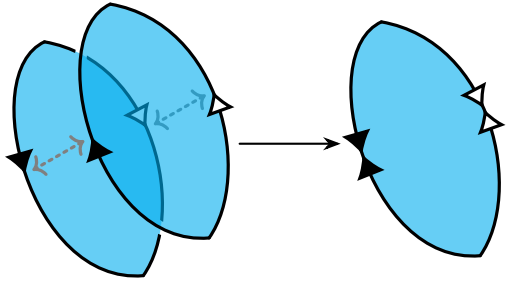
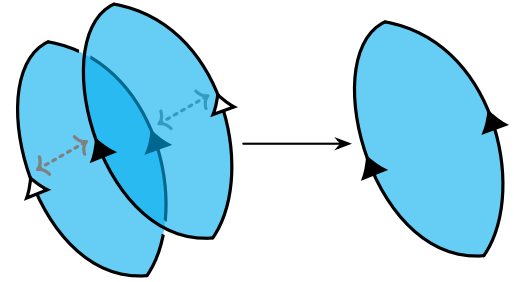

\begin{proof}
We first consider the case where the ends of $\mathcal{F}^\dagger$ are identified in such a way that
$\mathcal{F}^\dagger$ forms a $2$-sphere in the boundary of $\mathcal{D}^\dagger$.
In this case, observe that after flattening $F^\dagger_0$ and $F^\dagger_1$, the image $\varphi(F)$ is a single internal edge in $\mathcal{D}_1$.
Let $e^\dagger_0$ and $e^\dagger_1$ denote the two edges of $\mathcal{D}^\dagger$ that are incident to $\mathcal{F}^\dagger$.
As illustrated in Figure~\ref{fig:bigonSinglePath:S2}, there are four ways to glue $F^\dagger_0$ and $F^\dagger_1$
together to recover the bigon face $F$, depending on whether the gluing is orientation-reversing,
and on whether the gluing causes $e^\dagger_0$ and $e^\dagger_1$ to be identified together:
\begin{itemize}
\item The two orientation-reversing gluings are shown in
Figures~\ref{subfig:bigonSinglePath:S2:glueD2} and~\ref{subfig:bigonSinglePath:S2:glueRP2};
these correspond to cases~\ref{case:bigonSinglePath:S2:glueD2} and~\ref{case:bigonSinglePath:S2:glueRP2}, respectively.
\item The two orientation-preserving gluings are shown in
Figures~\ref{subfig:bigonSinglePath:S2:glueInvalid} and~\ref{subfig:bigonSinglePath:S2:glueS2};
these correspond to cases~\ref{case:bigonSinglePath:S2:glueInvalid} and~\ref{case:bigonSinglePath:S2:glueS2}, respectively.
\end{itemize}

Suppose now that the ends of $\mathcal{F}^\dagger$ are identified in such a way that
$\mathcal{F}^\dagger$ forms a projective plane in the boundary of $\mathcal{D}^\dagger$.
Let $e^\dagger_0$ and $e^\dagger_1$ denote the two edges of $\mathcal{D}^\dagger$ that are incident to $\mathcal{F}^\dagger$.
Up to symmetry, there are two ways to glue $F^\dagger_0$ and $F^\dagger_1$ back together, depending on whether the gluing causes $e^\dagger_0$ and $e^\dagger_1$ to be identified together.
As illustrated in Figure~\ref{fig:bigonSinglePath:RP2}, $F$ is incident to an invalid edge in both cases.

\begin{figure}[htbp]
\centering
	\begin{subfigure}[t]{0.45\textwidth}
	\centering
		\begin{tikzpicture}

		\node[inner sep=0pt] (before) at (0,0) {
			\includegraphics[scale=1]{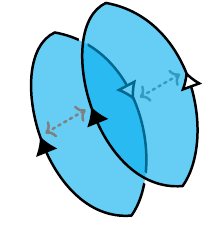}
		};
		\node[inner sep=0pt] (after) at (4,0) {
			\includegraphics[scale=1]{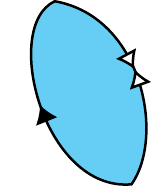}
		};
		\draw[thick, line cap=round, -Stealth] ($(before.east)+(-0.2,0)$) -- ($(after.west)+(0.5,0)$);

		\end{tikzpicture}
	\caption{A gluing that results in $F$ being incident to one internal edge and one invalid edge.}
	\label{subfig:bigonSinglePath:RP2:glueInvalid1}
	\end{subfigure}
	\hfill
	\begin{subfigure}[t]{0.45\textwidth}
	\centering
		\begin{tikzpicture}

		\node[inner sep=0pt] (before) at (0,0) {
			\includegraphics[scale=1]{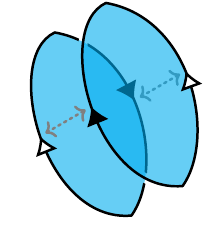}
		};
		\node[inner sep=0pt] (after) at (4,0) {
			\includegraphics[scale=1]{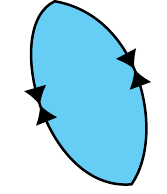}
		};
		\draw[thick, line cap=round, -Stealth] ($(before.east)+(-0.2,0)$) -- ($(after.west)+(0.4,0)$);

		\end{tikzpicture}
	\caption{A gluing that results in $F$ being incident to an invalid edge.}
	\label{subfig:bigonSinglePath:RP2:glueInvalid2}
	\end{subfigure}
\caption{The two ways to glue $F^\dagger_0$ and $F^\dagger_1$ together when $\mathcal{F}^\dagger$ forms a projective plane.}
\label{fig:bigonSinglePath:RP2}
\end{figure}

Finally, suppose the ends of $\mathcal{F}^\dagger$
are not identified, so that $\mathcal{F}^\dagger$ forms a disc in the boundary of $\mathcal{D}^\dagger$.
Observe that each end of $\mathcal{F}^\dagger$ must be incident to a boundary face of $\mathcal{D}^\dagger$ that is not part of $\mathcal{F}^\dagger$.
This means that regardless of how we glue $F^\dagger_0$ and $F^\dagger_1$ back together,
the bigon face $F$ will always be incident to an edge lying in the boundary of $\mathcal{D}_0$.
\end{proof}

As mentioned earlier, we only give a detailed analysis of the effect of flattening $F$ in the cases where $F$ is not incident to any boundary or invalid edges.
This corresponds to cases~\ref{case:bigonSinglePath:S2:glueD2},~\ref{case:bigonSinglePath:S2:glueRP2}
and~\ref{case:bigonSinglePath:S2:glueS2} of Claim~\ref{claim:bigonSinglePath}.

\begin{subatomic}[Disc]\label{claim:bigonSinglePath:S2:glueD2}
In case~\ref{case:bigonSinglePath:S2:glueD2} of Claim~\ref{claim:bigonSinglePath},
the truncated pseudomanifolds $\mathcal{P}_0$ and $\mathcal{P}_1$ are homeomorphic.
\end{subatomic}

\begin{proof}
Recall that in case~\ref{case:bigonSinglePath:S2:glueD2} of Claim~\ref{claim:bigonSinglePath},
the edges of $F$ are not identified, so that $F$ forms a disc.
In this case, we can flatten $F$ without changing the topology of $\mathcal{D}_0$;
in particular, as we saw in Claim~\ref{claim:flattenBigonVertexLink}, the links of the vertices incident to $F$ remain unchanged after flattening $F$.
Thus, we see that $\mathcal{P}_0$ and $\mathcal{P}_1$ are homeomorphic.
\end{proof}

\begin{subatomic}[Projective plane]\label{claim:bigonSinglePath:S2:glueRP2}
In case~\ref{case:bigonSinglePath:S2:glueRP2} of Claim~\ref{claim:bigonSinglePath},
the two vertices of $F$ are identified to form a single vertex $v$, and one of the following holds:
\begin{enumerate}[nosep,label={(\alph*)}]
\item\label{case:bigonSinglePath:S2:glueRP2:internal}
The vertex $v$ is internal, in which case $F$ forms a one-sided properly embedded projective plane in $\mathcal{P}_0$,
and $\mathcal{P}_1$ is obtained from $\mathcal{P}_0$ by decomposing along this projective plane.
\item\label{case:bigonSinglePath:S2:glueRP2:ideal}
The vertex $v$ is ideal, in which case the truncated bigon associated to $F$ forms
a one-sided properly embedded M\"{o}bius band $S$ in $\mathcal{P}_0$;
the boundary curve $\gamma$ of $S$ forms a two-sided curve in $\partial\mathcal{P}_0$.
In this case, flattening $F$ has one of the following effects:
	\begin{itemize}[nosep]
	\item If $\gamma$ bounds a disc $E$ in $\partial\mathcal{P}_0$,
	then $\mathcal{P}_1$ is obtained from $\mathcal{P}_0$ by
	decomposing along a one-sided projective plane given by
	isotoping $S\cup E$ slightly off the boundary of $\mathcal{P}_0$.
	\item If $\gamma$ does not bound a disc in $\partial\mathcal{P}_0$,
	then $\mathcal{P}_1$ is obtained from $\mathcal{P}_0$ by
	first decomposing along the M\"{o}bius band $S$, and then
	filling any new $2$-sphere boundary components with $3$-balls.
	\end{itemize}
\item\label{case:bigonSinglePath:S2:glueRP2:bdryInvalid}
The vertex $v$ is boundary or invalid.
\end{enumerate}
\end{subatomic}

\begin{proof}
Recall that in case~\ref{case:bigonSinglePath:S2:glueRP2} of Claim~\ref{claim:bigonSinglePath},
the edges of $F$ are identified so that $F$ forms a projective plane.
In particular, this means that the two vertices of $F$ are identified to form a single vertex $v$ in $\mathcal{D}_0$.
We start by getting the easy cases out of the way:
\begin{itemize}
\item If $v$ is boundary or invalid, then we are in case~\ref{case:bigonSinglePath:S2:glueRP2:bdryInvalid}.
\item If $v$ is internal, then $F$ forms an embedded projective plane that lies entirely in the interior of $\mathcal{P}_0$.
In this case, observe that ungluing $F$ corresponds topologically to cutting along this projective plane;
this yields a single $2$-sphere remnant, corresponding precisely to the boundary bigon path $\mathcal{F}^\dagger$.
This means, in particular, that $F$ forms a \emph{one-sided} projective plane in $\mathcal{P}_0$.
Moreover, as illustrated in Figure~\ref{fig:bigonSinglePath:S2:glueRP2:internal:fillS2}, flattening $F^\dagger_0$ and $F^\dagger_1$ is
topologically equivalent to filling the $2$-sphere $\mathcal{F}^\dagger$ with a $3$-ball.
Altogether, we see that flattening $F$ is topologically equivalent to decomposing along $F$.
This proves case~\ref{case:bigonSinglePath:S2:glueRP2:internal}.
\end{itemize}

\begin{figure}[htbp]
\centering
	\begin{tikzpicture}

	\node[inner sep=3pt] (before) at (0,0) {
		\includegraphics[scale=1]{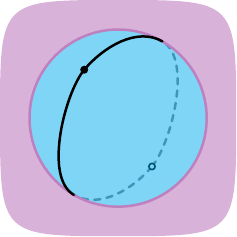}
	};
	\node[inner sep=3pt] (after) at (5.6,0) {
		\includegraphics[scale=1]{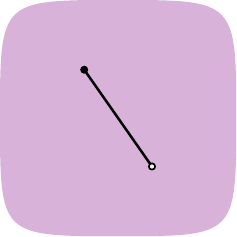}
	};
	\draw[thick, line cap=round, -Stealth] (before.east) -- (after.west);

	\end{tikzpicture}
\caption{When $\mathcal{F}^\dagger$ forms a $2$-sphere,
flattening $\mathcal{F}^\dagger$ is equivalent to filling it with a $3$-ball.}
\label{fig:bigonSinglePath:S2:glueRP2:internal:fillS2}
\end{figure}

The rest of this proof is devoted to the case where $v$ is ideal;
we need to prove all the conclusions stated in case~\ref{case:bigonSinglePath:S2:glueRP2:ideal}.
For this, we first observe that the truncated bigon associated to $F$ forms a properly embedded M\"{o}bius band $S$ in $\mathcal{P}_0$.
Consider the pseudomanifold $\mathcal{P}^\dagger$ obtained from $\mathcal{D}^\dagger$ by truncating the vertices in $g^{-1}(V_0)$;
viewing $\mathcal{P}^\dagger$ as a subset of $\mathcal{D}^\dagger$, let $S^\dagger$ denote the annulus in $\partial\mathcal{P}^\dagger$
given by $\mathcal{F}^\dagger\cap\mathcal{P}^\dagger$.
Topologically, observe that $\mathcal{P}^\dagger$ is obtained from $\mathcal{P}_0$ by cutting along the M\"{o}bius band $S$;
as shown in Figure~\ref{fig:bigonSinglePath:S2:glueRP2:ideal:cutMobius},
this yields a single remnant---namely, the annulus $S^\dagger$---so $S$ must be a \emph{one-sided} M\"{o}bius band in $\mathcal{P}_0$.

\begin{figure}[htbp]
\centering
	\begin{tikzpicture}

	\node[inner sep=-5pt] (before) at (0,0) {
		\includegraphics[scale=1]{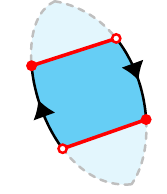}
	};
	\node[inner sep=-10pt] (after) at (4.2,0) {
		\includegraphics[scale=1]{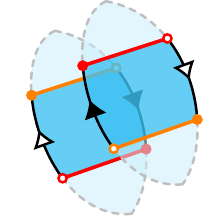}
	};
	\draw[thick, line cap=round, -Stealth] (before.east) -- (after.west);

	\end{tikzpicture}
\caption{The truncated bigon associated to $F$ forms a one-sided M\"{o}bius band $S$.
Cutting along $S$ yields a single annulus remnant.}
\label{fig:bigonSinglePath:S2:glueRP2:ideal:cutMobius}
\end{figure}

Consider the ideal boundary component $L$ of $\mathcal{P}_0$ given by truncating the vertex $v$, and let $\gamma$ denote the boundary curve of $S$.
To see that $\gamma$ forms a two-sided curve in $L$, observe that cutting along $S$
splits $\gamma$ into the two disjoint curves that bound the annulus $S^\dagger$.

All that remains is to understand the overall topological effect of flattening $F$.
We begin with the case where $\gamma$ bounds a disc $E$ in $L$.
In this case, we use Claim~\ref{claim:flattenBigonVertexCone} to flatten the $v$-cone over $\gamma$.
This reduces the operation of flattening $F$ to the operation of flattening a new bigon $F'$ given by pushing $F$ slightly away from $v$;
topologically, $F'$ is equivalent to a projective plane given by isotoping $S\cup E$ slightly off the boundary of $\mathcal{P}_0$.
Since the vertices of $F'$ are identified to form the temporary internal vertex that results from flattening the curve $\gamma$,
flattening $F'$ has the same topological effect as flattening $F$ in the case
where $v$ is internal (case~\ref{case:bigonSinglePath:S2:glueRP2:internal}).
That is, $F'$ forms a \emph{one-sided} projective plane in $\mathcal{P}_0$,
and $\mathcal{P}_1$ is obtained from $\mathcal{P}_0$ by decomposing along this projective plane.
This completes the case where $\gamma$ bounds a disc in $L$.

For the case where $\gamma$ does \emph{not} bound a disc in $L$, we flatten $F$ by first ungluing $F$, and then flattening $F^\dagger_0$ and $F^\dagger_1$.
Earlier, we observed that ungluing $F$ has the effect of cutting $\mathcal{P}_0$ along $S$,
which yields a single annulus remnant $S^\dagger$ in a new pseudomanifold $\mathcal{P}^\dagger$.
With this in mind, consider the pseudomanifold $\mathcal{P}^\ast$ obtained from $\mathcal{D}_1$ by truncating the vertices in $\varphi(V_0)$.
Topologically, observe that $\mathcal{P}^\ast$ is obtained from $\mathcal{P}^\dagger$ by filling the annulus $S^\dagger$ with a thickened disc;
see Figure~\ref{fig:bigonSinglePath:S2:glueRP2:ideal:fillAnnulus}.
In other words, $\mathcal{P}^\ast$ is obtained from $\mathcal{P}_0$ by decomposing along the M\"{o}bius band $S$.

\begin{figure}[htbp]
\centering
	\begin{tikzpicture}

	\node[inner sep=3pt] (before) at (0,0) {
		\includegraphics[scale=1]{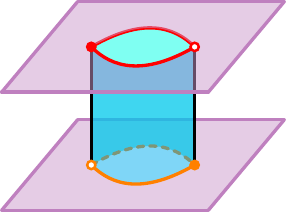}
	};
	\node[inner sep=3pt] (after) at (6.5,0) {
		\includegraphics[scale=1]{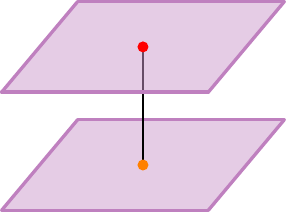}
	};
	\draw[thick, line cap=round, -Stealth] (before.east) -- (after.west);

	\end{tikzpicture}
\caption{Flattening $\mathcal{F}^\dagger$ has the effect of filling the annulus remnant $S^\dagger$ with a thickened disc.}
\label{fig:bigonSinglePath:S2:glueRP2:ideal:fillAnnulus}
\end{figure}

To see how $\mathcal{P}^\ast$ is related to $\mathcal{P}_1$, we need to compare the truncated vertex sets $\varphi(V_0)$ and $V_1$.
The only way these vertex sets can differ is if $\varphi(v)$ contains a vertex that is
neither ideal nor invalid---such a vertex would be in $\varphi(V_0)$ but not in $V_1$.
We can use Claim~\ref{claim:flattenBigonVertexLink} to determine the composition of $\varphi(v)$,
and hence determine the relationship between $\mathcal{P}^\ast$ and $\mathcal{P}_1$:
\begin{itemize}
\item If $\gamma$ is a non-separating curve in $L$, then decomposing $L$ along $\gamma$ gives a single new closed surface $L^\ast$,
and $\varphi(v)$ consists of a single vertex whose link is given by $L^\ast$.
If $L^\ast$ is not a $2$-sphere, then $\varphi(v)$ is an ideal vertex, so $\mathcal{P}^\ast$ is homeomorphic to $\mathcal{P}_1$.
However, if $L^\ast$ is a $2$-sphere, then $\varphi(v)$ is an internal vertex;
topologically, $L^\ast$ corresponds to a $2$-sphere boundary component of $\mathcal{P}^\ast$,
and we need to fill this $2$-sphere with a $3$-ball to recover $\mathcal{P}_1$ from $\mathcal{P}^\ast$.
\item If $\gamma$ is a separating curve in $L$, then decomposing along $\gamma$ gives two new closed surfaces,
and $\varphi(v)$ consists of two vertices whose links are given by these two new surfaces.
Since $\gamma$ does not bound a disc in $L$, both vertices in $\varphi(v)$ are ideal, so $\mathcal{P}^\ast$ is homeomorphic to $\mathcal{P}_1$.
\end{itemize}
This completes the proof of case~\ref{case:bigonSinglePath:S2:glueRP2:ideal}.
\end{proof}

\stepcounter{subatomic}
\begin{subatomic}[$2$-sphere]\label{claim:bigonSinglePath:S2:glueS2}
In case~\ref{case:bigonSinglePath:S2:glueS2} of Claim~\ref{claim:bigonSinglePath},
each vertex incident to $F$ is either ideal or invalid.
Moreover, the truncated bigon associated to $F$ forms a one-sided properly embedded annulus $S$ in $\mathcal{P}_0$,
and each boundary curve of $S$ forms a one-sided curve in $\partial\mathcal{P}_0$.

Suppose $F$ is only incident to ideal vertices
(the two vertices of $F$ could either be identified to form a single ideal vertex, or they could form two distinct ideal vertices).
In this case, $\mathcal{P}_1$ is obtained from $\mathcal{P}_0$ by first decomposing along the annulus $S$,
and then filling any new $2$-sphere boundary components with $3$-balls.
\end{subatomic}

\begin{proof}
Recall that in case~\ref{case:bigonSinglePath:S2:glueS2} of Claim~\ref{claim:bigonSinglePath},
the edges of $F$ are identified so that $F$ forms a $2$-sphere.
The two vertices of $F$ could either be identified to form a single vertex of $\mathcal{D}_0$,
or they could form two distinct vertices of $\mathcal{D}_0$.
Let $L$ denote the union of the links of the vertices incident to $F$,
and consider the two curves $\gamma_0$ and $\gamma_1$ in which $F$ meets $L$.
For each $i\in\{0,1\}$, observe that ungluing $F$ causes the curve $\gamma_i$ to ``unravel'' to form a single new curve;
thus, $\gamma_i$ forms a one-sided curve in $L$.
In particular, the vertices incident to $F$ must have non-orientable vertex links,
which implies that these vertices must be either ideal or invalid.
This means that the truncated bigon associated to $F$ forms a properly embedded annulus $S$ in $\mathcal{P}_0$.

To see that $S$ is one-sided, consider the pseudomanifold $\mathcal{P}^\dagger$ obtained
from $\mathcal{D}^\dagger$ by truncating the vertices in $g^{-1}(V_0)$;
viewing $\mathcal{P}^\dagger$ as a subset of $\mathcal{D}^\dagger$, let $S^\dagger$ denote the annulus in $\partial\mathcal{P}^\dagger$
given by $\mathcal{F}^\dagger\cap\mathcal{P}^\dagger$.
Topologically, $\mathcal{P}^\dagger$ is obtained from $\mathcal{P}_0$ by cutting along the annulus $S$;
as shown in Figure~\ref{fig:bigonSinglePath:S2:glueS2:cutAnnulus},
this yields a single remnant---namely, the annulus $S^\dagger$---which tells us that $S$ is a one-sided annulus in $\mathcal{P}_0$.

\begin{figure}[htbp]
\centering
	\begin{tikzpicture}

	\node[inner sep=-5pt] (before) at (0,0) {
		\includegraphics[scale=1]{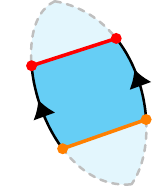}
	};
	\node[inner sep=-10pt] (after) at (4.2,0) {
		\includegraphics[scale=1]{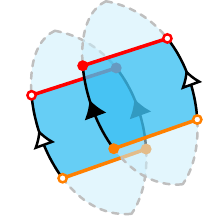}
	};
	\draw[thick, line cap=round, -Stealth] (before.east) -- (after.west);

	\end{tikzpicture}
\caption{The truncated bigon associated to $F$ forms a one-sided annulus $S$.
Cutting along $S$ yields a single annulus remnant.}
\label{fig:bigonSinglePath:S2:glueS2:cutAnnulus}
\end{figure}

Suppose now that the vertices incident to $F$ are all ideal.
Consider the pseudomanifold $\mathcal{P}^\ast$ obtained from $\mathcal{D}_1$ by truncating the vertices in $\varphi(V_0)$.
Topologically, $\mathcal{P}^\ast$ is obtained from $\mathcal{P}^\dagger$ by filling the annulus $S^\dagger$ with a thickened disc;
see Figure~\ref{fig:bigonSinglePath:S2:glueRP2:ideal:fillAnnulus}.
In other words, $\mathcal{P}^\ast$ is obtained from $\mathcal{P}_0$ by decomposing along the annulus $S$.

To see how $\mathcal{P}^\ast$ and $\mathcal{P}_1$ are related, we need to compare the truncated vertex sets $\varphi(V_0)$ and $V_1$.
For this, we first note that since $F$ is only incident to ideal vertices, each component of $L$ must be a closed surface other than a $2$-sphere.
Let $L^\ast$ denote the (possibly disconnected) surface obtained by decomposing $L$ along $\gamma_0$ and $\gamma_1$;
each component of $L^\ast$ must be a closed surface, but could possibly be a $2$-sphere.
By Claim~\ref{claim:flattenBigonVertexLink}, the components of $L^\ast$ correspond precisely to
the boundary components of $\mathcal{P}^\ast$ given by truncating the vertices in $\varphi(v)$.
The only way $\mathcal{P}^\ast$ can differ from $\mathcal{P}_1$ is if $L^\ast$ has $2$-sphere components;
we need to fill each such $2$-sphere with a $3$-ball to recover $\mathcal{P}_1$ from $\mathcal{P}^\ast$.
\end{proof}

\subsubsection{The case where ungluing gives two separate boundary bigon paths}\label{subsubsec:bigonSeparatePaths}

We now consider the case where $F^\dagger_0$ and $F^\dagger_1$ form two separate boundary bigon paths.
We have the following cases:

\begin{atomic}\label{claim:bigonSeparatePaths}
If $F^\dagger_0$ and $F^\dagger_1$ form two separate boundary bigon paths, then one of the following holds
\emph{(see Figure~\ref{fig:bigonSeparatePaths})}:
\begin{enumerate}[nosep,label={(\arabic*)}]
\item\label{case:bigonSeparatePaths:glueS2}
For each $i\in\{0,1\}$, the boundary bigon path $F^\dagger_i$ forms a $2$-sphere.
In this case, $\varphi(F)$ consists of two distinct internal edges in $\mathcal{D}_1$.
Moreover, the edges of $F$ are identified to form a single internal edge in $\mathcal{D}_0$,
and $F$ itself forms a \textbf{\textup{2}-sphere} in $\mathcal{D}_0$.\\
\emph{(See Claim~\ref{claim:bigonSeparatePaths:glueS2} for details about the effect of flattening $F$ in this case.)}
\item\label{case:bigonSeparatePaths:glueRP2}
For each $i\in\{0,1\}$, the boundary bigon path $F^\dagger_i$ forms a projective plane.
In this case, $\varphi(F)$ consists of two distinct invalid edges in $\mathcal{D}_1$.
Moreover, the edges of $F$ are identified to form a single internal edge in $\mathcal{D}_0$,
and $F$ itself forms a \textbf{projective plane} in $\mathcal{D}_0$.\\
\emph{(See Claim~\ref{claim:bigonSeparatePaths:glueRP2} for details about the effect of flattening $F$ in this case.)}
\item\label{case:bigonSeparatePaths:glueInvalid}
For some $i\in\{0,1\}$, the boundary bigon path $F^\dagger_i$ forms a $2$-sphere,
but the boundary bigon path $F^\dagger_{1-i}$ forms a projective plane.
In this case, the edges of $F$ are identified to form a single invalid edge in $\mathcal{D}_0$.
\item\label{case:bigonSeparatePaths:glueBdry}
For some $i\in\{0,1\}$, the boundary bigon path $F^\dagger_i$ forms a disc,
in which case (regardless of the behaviour of $F^\dagger_{1-i}$) $F$ is incident to a boundary edge in $\mathcal{D}_0$.
\end{enumerate}
\end{atomic}

\begin{figure}[htbp]
\centering
	\begin{subfigure}[t]{0.45\textwidth}
	\centering
		\begin{tikzpicture}

		\node[inner sep=0pt] (before) at (0,0) {
			\includegraphics[scale=1]{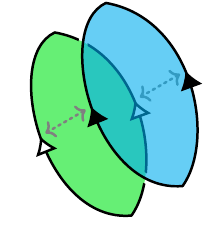}
		};
		\node[inner sep=0pt] (after) at (4,0) {
			\includegraphics[scale=1]{BigonSingleS2_S2After.pdf}
		};
		\draw[thick, line cap=round, -Stealth] ($(before.east)+(-0.2,0)$) -- ($(after.west)+(0.5,0)$);

		\end{tikzpicture}
	\caption{Gluing two separate $2$-spheres results in $F$ forming a $2$-sphere.}
	\label{subfig:bigonSeparatePaths:glueS2}
	\end{subfigure}
	\hfill
	\begin{subfigure}[t]{0.45\textwidth}
	\centering
		\begin{tikzpicture}

		\node[inner sep=0pt] (before) at (0,0) {
			\includegraphics[scale=1]{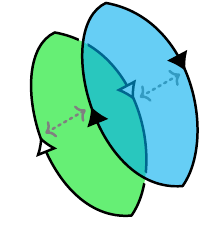}
		};
		\node[inner sep=0pt] (after) at (4,0) {
			\includegraphics[scale=1]{BigonSingleS2_RP2After.pdf}
		};
		\draw[thick, line cap=round, -Stealth] ($(before.east)+(-0.2,0)$) -- ($(after.west)+(0.5,0)$);

		\end{tikzpicture}
	\caption{Gluing two separate projective planes results in $F$ forming a projective plane.}
	\label{subfig:bigonSeparatePaths:glueRP2}
	\end{subfigure}
	\bigskip
	\begin{subfigure}[t]{0.45\textwidth}
	\centering
		\begin{tikzpicture}

		\node[inner sep=0pt] (before) at (0,0) {
			\includegraphics[scale=1]{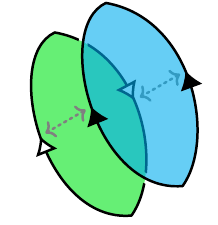}
		};
		\node[inner sep=0pt] (after) at (4,0) {
			\includegraphics[scale=1]{BigonSingleRP2_Invalid2After.pdf}
		};
		\draw[thick, line cap=round, -Stealth] ($(before.east)+(-0.2,0)$) -- ($(after.west)+(0.4,0)$);

		\end{tikzpicture}
	\caption{Gluing a $2$-sphere and a projective plane results in $F$ being incident to an invalid edge.}
	\label{subfig:bigonSeparatePaths:glueInvalid}
	\end{subfigure}
\caption{Gluing $F^\dagger_0$ and $F^\dagger_1$ together when each forms a separate $2$-sphere or projective plane in the boundary of $\mathcal{D}^\dagger$.}
\label{fig:bigonSeparatePaths}
\end{figure}

\begin{proof}
For each $i\in\{0,1\}$, depending on how the ends of $F^\dagger_i$ are identified (if at all),
$F^\dagger_i$ could form a $2$-sphere, projective plane or disc in the boundary of $\mathcal{D}^\dagger$.
Observe that:
\begin{itemize}
\item if $F^\dagger_i$ forms a $2$-sphere, then flattening $F^\dagger_i$ yields a single internal edge; and
\item if $F^\dagger_i$ forms a projective plane, then flattening $F^\dagger_i$ yields a single invalid edge.
\end{itemize}
With this in mind, we have the following four cases, which correspond precisely to the cases stated in the claim:
\begin{description}[font=\normalfont]
\item[\ref{case:bigonSeparatePaths:glueS2}]
Suppose that $F^\dagger_0$ and $F^\dagger_1$ both form $2$-spheres.
After the gluing these bigon faces back together, the edges of $F$ are identified so that $F$ forms a $2$-sphere,
as illustrated in Figure~\ref{subfig:bigonSeparatePaths:glueS2}.
\item[\ref{case:bigonSeparatePaths:glueRP2}]
Suppose that $F^\dagger_0$ and $F^\dagger_1$ both form projective planes.
After gluing these bigon faces back together, the edges of $F$ are identified so that $F$ forms a projective plane,
as illustrated in Figure~\ref{subfig:bigonSeparatePaths:glueRP2}.
\item[\ref{case:bigonSeparatePaths:glueInvalid}]
Suppose that one of $F^\dagger_0$ or $F^\dagger_1$ forms a $2$-sphere, whilst the other forms a projective plane.
After gluing these bigon faces back together, the edges of $F$ are identified to form a single invalid edge,
as illustrated in Figure~\ref{subfig:bigonSeparatePaths:glueInvalid}.
\item[\ref{case:bigonSeparatePaths:glueBdry}]
Suppose that for some $i\in\{0,1\}$, $F^\dagger_i$ forms a disc in the boundary of $\mathcal{D}^\dagger$;
each end of $F^\dagger_i$ must therefore be incident to a boundary face other than $F^\dagger_0$ or $F^\dagger_1$.
Thus, regardless of how we glue $F^\dagger_0$ and $F^\dagger_1$ back together,
the bigon face $F$ will always be incident to at least one boundary edge of $\mathcal{D}_0$.
\qedhere
\end{description}
\end{proof}

As before, we only give a detailed analysis of the effect of flattening $F$ in the cases where $F$ is not incident to any boundary or invalid edges.
This corresponds to cases~\ref{case:bigonSeparatePaths:glueS2}
and~\ref{case:bigonSeparatePaths:glueRP2} of Claim~\ref{claim:bigonSeparatePaths}.

\begin{subatomic}[$2$-sphere]\label{claim:bigonSeparatePaths:glueS2}
In case~\ref{case:bigonSeparatePaths:glueS2} of Claim~\ref{claim:bigonSeparatePaths},
one of the following holds:
	\begin{enumerate}[nosep,label={(\alph*)}]
	\item\label{case:bigonSeparatePaths:glueS2:onlyInternal}
	The bigon face $F$ is only incident to internal vertices
	(the two vertices of $F$ could either be identified to form a single internal vertex,
	or they could form two distinct internal vertices).
	In this case, $\mathcal{P}_1$ is obtained from $\mathcal{P}_0$ by decomposing along
	a properly embedded $2$-sphere given by pushing $F$ slightly away from its incident vertices.
	\item\label{case:bigonSeparatePaths:glueS2:internalIdeal}
	The bigon face $F$ is incident to one internal vertex and one ideal vertex.
	In this case, the truncated bigon associated to $F$ forms a properly embedded disc $S$ in $\mathcal{P}_0$;
	the boundary curve $\gamma$ of $S$ is a two-sided curve in $\partial\mathcal{P}_0$.
	Flattening $F$ has one of the following effects:
		\begin{itemize}[nosep]
		\item If $\gamma$ bounds a disc $E$ in $\partial\mathcal{P}_0$,
		then $\mathcal{P}_1$ is obtained from $\mathcal{P}_0$ by decomposing along
		a properly embedded $2$-sphere given by isotoping $S\cup E$ slightly off the boundary.
		\item If $\gamma$ does not bound a disc in $\mathcal{P}_0$, then $\mathcal{P}_1$ is obtained from $\mathcal{P}_0$ by first cutting along the disc $S$, and then filling any new $2$-sphere boundary components with $3$-balls.
		\end{itemize}
	\item\label{case:bigonSeparatePaths:glueS2:onlyIdeal}
	The bigon face $F$ is only incident to ideal vertices
	(the two vertices of $F$ could either be identified to form a single ideal vertex,
	or they could form two distinct ideal vertices).
	In this case, the truncated bigon associated to $F$ forms a two-sided properly embedded annulus $S$ in $\mathcal{P}_0$;
	the boundary curves $\gamma_0$ and $\gamma_1$ of $S$ form two-sided curves in $\partial\mathcal{P}_0$.
	Flattening $F$ has one of the following effects:
		\begin{itemize}[nosep]
		\item If $\gamma_0$ and $\gamma_1$ respectively bound discs $E_0$ and $E_1$ in $\partial\mathcal{P}_0$,
		then either these discs are disjoint or one of these discs lies entirely in the interior of the other;
		choose $i\in\{0,1\}$ so that $E_i$ either lies entirely inside or entirely outside $E_{1-i}$.
		In this case, $\mathcal{P}_1$ is obtained from $\mathcal{P}_0$ by
		decomposing along a properly embedded $2$-sphere $S^\ast$ constructed as follows:
			\begin{enumerate}[nosep,label={(\roman*)}]
			\item Isotope $S\cup E_i$ slightly off the boundary to obtain
			a properly embedded disc $S'$ in $\mathcal{P}_0$.
			\item Isotope $S'\cup E_{1-i}$ slightly off the boundary to obtain the desired $2$-sphere $S^\ast$.
			\end{enumerate}
		\item If for some $i\in\{0,1\}$, the curve $\gamma_i$ bounds a disc $E_i$ in $\partial\mathcal{P}_0$,
		but $\gamma_{1-i}$ does not bound a disc in $\partial\mathcal{P}_0$,
		then $\mathcal{P}_1$ is obtained from $\mathcal{P}_0$ by first cutting along
		a disc given by isotoping $S\cup E_i$ slightly away from $E_i$,
		and then filling any new $2$-sphere boundary components with $3$-balls.
		\item If neither $\gamma_0$ nor $\gamma_1$ bounds a disc in $\partial\mathcal{P}_0$,
		then $\mathcal{P}_1$ is obtained from $\mathcal{P}_0$ by first decomposing along the annulus $S$,
		and then filling any new $2$-sphere boundary components with $3$-balls.
		\end{itemize}
	\item\label{case:bigonSeparatePaths:glueS2:bdryInvalid}
	There is a boundary or invalid vertex incident to $F$.
	\end{enumerate}
\end{subatomic}

\begin{proof}
Recall that in case~\ref{case:bigonSeparatePaths:glueS2} of Claim~\ref{claim:bigonSeparatePaths},
the edges of $F$ are identified so that $F$ forms a $2$-sphere.
The two vertices of $F$ could either be identified to form a single vertex of $\mathcal{D}_0$,
or they could form two distinct vertices of $\mathcal{D}_0$.
Let $L$ denote the union of the links of the vertices incident to $F$,
and let $\gamma_0$ and $\gamma_1$ denote the two curves in which $F$ meets $L$.
For each $i\in\{0,1\}$, let $v_i$ denote the vertex of $F$ that is cut off by the curve $\gamma_i$, and let $L_i$ denote the link of $v_i$;
if $v_i$ is ideal, then we also think of $L_i$ as the ideal boundary component of $\mathcal{P}_0$ given by truncating $v_i$.
If $v_0$ and $v_1$ are identified, then $L=L_0=L_1$;
otherwise, $L$ is the disjoint union of $L_0$ and $L_1$.
With all this setup in mind, we start by getting the easy cases out of the way:
\begin{itemize}
\item If there is a boundary or invalid vertex incident to $F$, then we are in case~\ref{case:bigonSeparatePaths:glueS2:bdryInvalid}.
\item If the vertices incident to $F$ are all internal, then $F$ forms a $2$-sphere in $\mathcal{P}_0$.
This $2$-sphere might not be embedded, since the two vertices of $F$ could be identified to form a single internal vertex.
Thus, to ensure that we have a properly embedded $2$-sphere, we use
Claim~\ref{claim:flattenBigonVertexCone} to flatten the $v_0$-cone over $\gamma_0$ and the $v_1$-cone over $\gamma_1$, one at a time.
This reduces the operation of flattening $F$ to the operation of flattening a new bigon $F'$ given by pushing $F$ slightly away from its incident vertices.
Since the vertices of $F'$ form two distinct temporary internal vertices, $F'$ forms the desired properly embedded $2$-sphere in $\mathcal{P}_0$.
As shown in Figure~\ref{fig:bigonSeparatePaths:glueS2:onlyInternal:decomposeS2},
flattening $F'$ is topologically equivalent to decomposing along this $2$-sphere.
This proves case~\ref{case:bigonSeparatePaths:glueS2:onlyInternal}.
\end{itemize}

\begin{figure}[htbp]
\centering
	\begin{tikzpicture}

	\node[inner sep=3pt] (before0) at (0,0) {
		\includegraphics[scale=0.8]{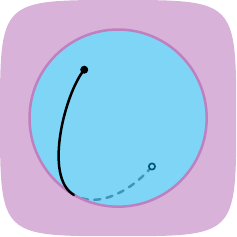}
	};
	\node[inner sep=3pt] (before1) at (3.6,0) {
		\includegraphics[scale=0.8]{SeparatePath_fillS2Before.pdf}
	};
	\node[inner sep=3pt] (after0) at (8.4,0) {
		\includegraphics[scale=0.8]{SinglePath_fillS2After.pdf}
	};
	\node[inner sep=3pt] (after1) at (12,0) {
		\includegraphics[scale=0.8]{SinglePath_fillS2After.pdf}
	};
	\draw[thick, line cap=round, -Stealth] (before1.east) -- (after0.west);
	\draw[very thick, line cap=round, <->, gray, dotted] (0.2,0.9) .. controls
		++(45:1.5) and ++(135:1.5)
		.. node[midway, above, inner sep=0pt]{spheres identified} (3.4,0.9);

	\end{tikzpicture}
\caption{When $F'$ forms a properly embedded $2$-sphere, flattening $F'$ is equivalent to decomposing along this $2$-sphere.}
\label{fig:bigonSeparatePaths:glueS2:onlyInternal:decomposeS2}
\end{figure}

We now consider the case where $F$ is incident to one internal vertex and one ideal vertex;
without loss of generality, suppose that $v_0$ is the internal vertex and $v_1$ is the ideal vertex.
We need to prove the conclusions stated in case~\ref{case:bigonSeparatePaths:glueS2:internalIdeal}.
Observe that the truncated bigon associated to $F$ forms a properly embedded disc $S$ in $\mathcal{P}_0$.
Consider the pseudomanifold $\mathcal{P}^\dagger$ obtained from $\mathcal{D}^\dagger$ by truncating the vertices in $g^{-1}(V_0)$;
viewing $\mathcal{P}^\dagger$ as a subset of $\mathcal{D}^\dagger$, for each $i\in\{0,1\}$ let $S^\dagger_i$ denote the disc in $\partial\mathcal{P}^\dagger$
given by $\mathcal{F}^\dagger_i\cap\mathcal{P}^\dagger$.
Topologically, $\mathcal{P}^\dagger$ is obtained from $\mathcal{P}_0$ by cutting along the disc $S$;
the two discs $S^\dagger_0$ and $S^\dagger_1$ form the remnants of cutting along $S$.
This is illustrated in Figure~\ref{fig:bigonSeparatePaths:glueS2:internalIdeal:cutDisc}.
We also note that the boundary of the disc $S$ is given by the curve $\gamma_1$;
since cutting along $S$ splits $\gamma_1$ into two remnants, one bounding each of the discs $S^\dagger_0$ and $S^\dagger_1$,
we see that $\gamma_1$ is a two-sided curve in $L_1$.
It remains to describe how flattening $F$ changes $\mathcal{P}_0$.
This depends on whether $\gamma_1$ bounds a disc in $L_1$:
\begin{itemize}
\item Suppose that $\gamma_1$ bounds a disc $E$ in $L_1$.
We can use Claim~\ref{claim:flattenBigonVertexCone} to flatten the $v_1$-cone over $\gamma_1$.
This reduces the operation of flattening $F$ to the operation of flattening a new bigon $F'$ such that:
	\begin{itemize}
	\item one of the vertices of $F'$ is the temporary internal vertex that results from flattening $\gamma_1$; and
	\item the other vertex of $F'$ is the internal vertex $v_0$.
	\end{itemize}
Topologically, $F'$ is equivalent to a properly embedded $2$-sphere given by isotoping $S\cup E$ slightly off the boundary of $\mathcal{P}_0$.
Moreover, by analogy with case~\ref{case:bigonSeparatePaths:glueS2:onlyInternal}, we see that $\mathcal{P}_1$ is obtained from $\mathcal{P}_0$ by decomposing along this $2$-sphere $F'$.
\item Suppose that $\gamma_1$ does not bound a disc in $L_1$.
Consider the pseudomanifold $\mathcal{P}^\ast$ obtained from $\mathcal{D}_1$ by truncating the vertices in $\varphi(V_0)$.
As shown in Figure~\ref{fig:bigonSeparatePaths:glueS2:internalIdeal:collapseDiscs}, $\mathcal{P}^\ast$ is obtained
from $\mathcal{P}^\dagger$ by collapsing $S^\dagger_0$ and $S^\dagger_1$ to arcs, which has no topological effect;
in other words, $\mathcal{P}^\ast$ is homeomorphic to $\mathcal{P}^\dagger$.
To see how $\mathcal{P}^\ast$ is related to $\mathcal{P}_1$, consider the surface $L^\ast$ obtained from $L_1$ by decomposing along $\gamma_1$.
Using Claim~\ref{claim:flattenBigonVertexLink}, we see that each component of $L^\ast$ corresponds to
a boundary component of $\mathcal{P}^\ast$ given by truncating one of the vertices in $\varphi(v_1)$.
Thus, if $\varphi(v_1)$ contains any internal vertices (since $\gamma$ is two-sided and does not bound a disc in $L_1$,
this is only possible if $L_1$ is a torus or a Klein bottle), then we need to fill each corresponding
$2$-sphere boundary component of $\mathcal{P}^\ast$ with a $3$-ball to recover $\mathcal{P}_1$.
\end{itemize}
This proves case~\ref{case:bigonSeparatePaths:glueS2:internalIdeal}.

\begin{figure}[htbp]
\centering
	\begin{tikzpicture}

	\node[inner sep=-5pt] (before) at (0,0) {
		\includegraphics[scale=1]{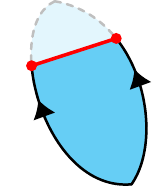}
	};
	\node[inner sep=-10pt] (after) at (6,0) {
		\includegraphics[scale=1]{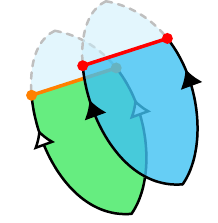}
	};
	\draw[thick, line cap=round, -Stealth] (before.east) -- (after.west);

	\node[inner sep=-6pt] (beforeEmb) at (0,-3.25) {
		\includegraphics[scale=1]{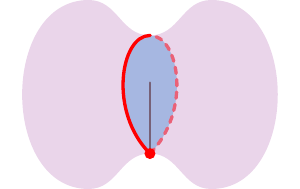}
	};
	\node[inner sep=-6pt] (afterEmb) at (6.5,-3.25) {
		\includegraphics[scale=1]{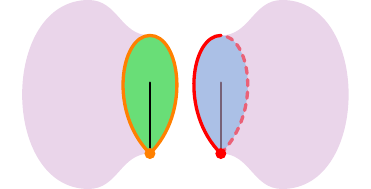}
	};
	\draw[thick, line cap=round, -Stealth] (beforeEmb.east) -- (afterEmb.west);

	\end{tikzpicture}
\caption{The truncated bigon associated to $F$ forms a properly embedded disc $S$.
Cutting along $S$ yields two disc remnants.}
\label{fig:bigonSeparatePaths:glueS2:internalIdeal:cutDisc}
\end{figure}

\begin{figure}[htbp]
\centering
	\begin{tikzpicture}

	\node[inner sep=-6pt] (before) at (0,0) {
		\includegraphics[scale=1]{EmbeddedDisc_Cut2.pdf}
	};
	\node[inner sep=-6pt] (after) at (7.2,0) {
		\includegraphics[scale=1]{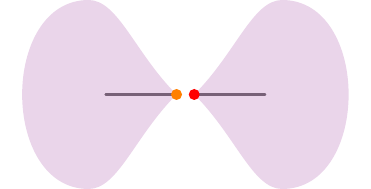}
	};
	\draw[thick, line cap=round, -Stealth] (before.east) -- (after.west);

	\end{tikzpicture}
\caption{Collapsing the discs $S^\dagger_0$ and $S^\dagger_1$ to arcs has no topological effect.}
\label{fig:bigonSeparatePaths:glueS2:internalIdeal:collapseDiscs}
\end{figure}

All that remains is to consider the case where the vertices incident to $F$ are all ideal;
we need to prove the conclusions stated in case~\ref{case:bigonSeparatePaths:glueS2:onlyIdeal}.
This time, the truncated bigon associated to $F$ forms a properly embedded annulus $S$ in $\mathcal{P}_0$.
Similar to before, consider the pseudomanifold $\mathcal{P}^\dagger$ obtained
from $\mathcal{D}^\dagger$ by truncating the vertices in $g^{-1}(V_0)$;
viewing $\mathcal{P}^\dagger$ as a subset of $\mathcal{D}^\dagger$, for each $i\in\{0,1\}$ let $S^\dagger_i$ denote the annulus in $\partial\mathcal{P}^\dagger$
given by $\mathcal{F}^\dagger_i\cap\mathcal{P}^\dagger$.
Topologically, $\mathcal{P}^\dagger$ is obtained from $\mathcal{P}_0$ by cutting along the annulus $S$;
as shown in Figure~\ref{fig:bigonSeparatePaths:glueS2:onlyIdeal:cutAnnulus},
this yields a pair of remnants---namely, the annuli $S^\dagger_0$ and $S^\dagger_1$---which
means that $S$ is a \emph{two-sided} annulus in $\mathcal{P}_0$.
We also note that for each $i\in\{0,1\}$, the curve $\gamma_i$ forms a two-sided curve in $L_i$,
since cutting along $S$ causes this curve to split into two remnants, one meeting $S^\dagger_0$ and the other meeting $S^\dagger_1$.

\begin{figure}[htbp]
\centering
	\begin{tikzpicture}

	\node[inner sep=-5pt] (before) at (0,0) {
		\includegraphics[scale=1]{TruncatedAnnulus_Uncut.pdf}
	};
	\node[inner sep=-10pt] (after) at (4.2,0) {
		\includegraphics[scale=1]{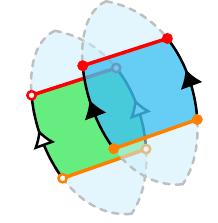}
	};
	\draw[thick, line cap=round, -Stealth] (before.east) -- (after.west);

	\end{tikzpicture}
\caption{The truncated bigon associated to $F$ forms a two-sided annulus $S$.
Cutting along $S$ yields two annulus remnants.}
\label{fig:bigonSeparatePaths:glueS2:onlyIdeal:cutAnnulus}
\end{figure}

Depending on whether $\gamma_0$ and $\gamma_1$ bound discs in $\partial\mathcal{P}_0$,
we have the following possibilities for how flattening $F$ changes $\mathcal{P}_0$:
\begin{itemize}
\item Suppose that for some $i\in\{0,1\}$, $\gamma_i$ bounds a disc $E_i$ in $L_i$.
We can assume without loss of generality that the interior of $E_i$ is disjoint from $F$.
To see why, note that the only way this assumption can fail is if $\gamma_{1-i}$ lies in the interior of $E_i$;
in this case, $\gamma_{1-i}$ bounds a ``smaller'' disc, so we can simply exchange the roles of $\gamma_i$ and $\gamma_{1-i}$.
This allows us to use Claim~\ref{claim:flattenBigonVertexCone} to flatten the $v_i$-cone over $\gamma_i$.
This reduces the operation of flattening $F$ to the operation of flattening a new bigon $F'$ such that:
	\begin{itemize}
	\item one of the vertices of $F'$ is given by the temporary internal vertex that results from flattening $\gamma_i$; and
	\item the other vertex of $F'$ is given by the ideal vertex $v_{1-i}$.
	\end{itemize}
Thus, flattening $F'$ has the same topological effect as flattening $F$ in case~\ref{case:bigonSeparatePaths:glueS2:internalIdeal}.
In more detail, after truncating the ideal vertex $v_{1-i}$, we see that $F'$ becomes a properly embedded disc $S'$ in $\mathcal{P}_0$;
we can view $\gamma_{1-i}$ as the boundary curve of this disc $S'$.
Topologically, $S'$ is obtained by isotoping $S\cup E_i$ slightly off the boundary of $\mathcal{P}_0$,
and the effect of flattening $F'$ depends on whether $\gamma_{1-i}$ bounds a disc in $L_{1-i}$:
	\begin{itemize}
	\item If $\gamma_{1-i}$ bounds a disc $E_{1-i}$ in $L_{1-i}$, then $\mathcal{P}_1$ is obtained from $\mathcal{P}_0$
	by decomposing along a properly embedded $2$-sphere given by isotoping $S'\cup E_{1-i}$ slightly off the boundary of $\mathcal{P}_0$.
	\item If $\gamma_{1-i}$ does not bound a disc in $L_{1-i}$, then $\mathcal{P}_1$ is obtained from $\mathcal{P}_0$
	by first cutting along the disc $S'$, and then filling any new $2$-sphere boundary components with $3$-balls.
	\end{itemize}
\item Suppose that neither $\gamma_0$ nor $\gamma_1$ bounds a disc in $\partial\mathcal{P}_0$.
Consider the pseudomanifold $\mathcal{P}^\ast$ obtained from $\mathcal{D}_1$ by truncating the vertices in $\varphi(V_0)$.
Topologically, $\mathcal{P}^\ast$ is obtained from $\mathcal{P}^\dagger$ by
filling the annuli $S^\dagger_0$ and $S^\dagger_1$ with thickened discs;
see Figure~\ref{fig:bigonSeparatePaths:glueS2:onlyIdeal:fillAnnulus}.
In other words, $\mathcal{P}^\ast$ is obtained from $\mathcal{P}_0$ by decomposing along the annulus  $S$.
To see how $\mathcal{P}^\ast$ is related to $\mathcal{P}_1$, consider the surface $L^\ast$
obtained from $L_0\cup L_1$ by decomposing along $\gamma_0$ and $\gamma_1$.
Using Claim~\ref{claim:flattenBigonVertexLink},
we see that each component of $L^\ast$ corresponds to a boundary component of $\mathcal{P}^\ast$ given by truncating one of the vertices in $\varphi(v_0)$ or $\varphi(v_1)$.
Thus, if there are any internal vertices in $\varphi(v_0)$ or $\varphi(v_1)$, then we need to
fill each corresponding $2$-sphere boundary component of $\mathcal{P}^\ast$ with a $3$-ball to recover $\mathcal{P}_1$.
\end{itemize}
This proves case~\ref{case:bigonSeparatePaths:glueS2:onlyIdeal}.
\end{proof}

\begin{figure}[htbp]
\centering
	\begin{tikzpicture}

	\node[inner sep=3pt] (before) at (0,0) {
		\includegraphics[scale=1]{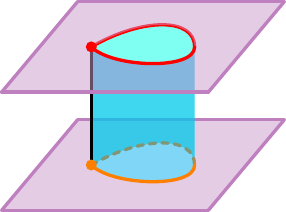}
	};
	\node[inner sep=3pt] (after) at (6.5,0) {
		\includegraphics[scale=1]{SinglePath_fillAnnulusAfter.pdf}
	};
	\draw[thick, line cap=round, -Stealth] (before.east) -- (after.west);

	\end{tikzpicture}
\caption{For each $i\in\{0,1\}$, flattening $F^\dagger_i$ has the effect of filling the annulus remnant $S^\dagger_i$ with a thickened disc.}
\label{fig:bigonSeparatePaths:glueS2:onlyIdeal:fillAnnulus}
\end{figure}

\begin{subatomic}[Projective plane]\label{claim:bigonSeparatePaths:glueRP2}
In case~\ref{case:bigonSeparatePaths:glueRP2} of Claim~\ref{claim:bigonSeparatePaths},
the two vertices of $F$ are identified to form a single vertex $v$, and one of the following holds:
	\begin{enumerate}[nosep,label={(\alph*)}]
	\item\label{case:bigonSeparatePaths:glueRP2:internal}
	The vertex $v$ is internal, in which case $F$ forms
	a two-sided properly embedded projective plane in $\mathcal{P}_0$,
	and $\mathcal{P}_1$ is obtained from $\mathcal{P}_0$ by decomposing along this projective plane.
	\item\label{case:bigonSeparatePaths:glueRP2:ideal}
	The vertex $v$ is ideal, in which case the truncated bigon associated to $F$ forms
	a two-sided properly embedded M\"{o}bius band $S$ in $\mathcal{P}_0$;
	the boundary curve $\gamma$ of $S$ forms a two-sided curve in $\partial\mathcal{P}_0$.
	In this case, flattening $F$ has one of the following effects:
		\begin{itemize}[nosep]
		\item If $\gamma$ bounds a disc $E$ in $\partial\mathcal{P}_0$,
		then $\mathcal{P}_1$ is obtained from $\mathcal{P}_0$ by
		decomposing along a two-sided projective plane given by
		isotoping $S\cup E$ slightly off the boundary of $\mathcal{P}_0$.
		\item If $\gamma$ does not bound a disc in $\partial\mathcal{P}_0$,
		then $\mathcal{P}_1$ is obtained from $\mathcal{P}_0$ by
		first decomposing along the M\"{o}bius band $S$, and then
		filling any new $2$-sphere boundary components with $3$-balls.
		\end{itemize}
	\item\label{case:bigonSeparatePaths:glueRP2:bdryInvalid}
	The vertex $v$ is boundary or invalid.
	\end{enumerate}
\end{subatomic}

Before we prove Claim~\ref{claim:bigonSeparatePaths:glueRP2}, we need to understand the topological effect of flattening
a boundary bigon face whose corresponding truncated bigon forms either a projective plane or a M\"{o}bius band.
This is the purpose of the following lemma:

\begin{lemma}\label{lem:fillInvalidCone}
Let $B$ be a boundary bigon face, with edges identified so that $B$ forms a projective plane.
Thus, the associated truncated bigon $B'$ forms either a projective plane or a M\"{o}bius band.
Topologically, flattening $B$ has the effect of filling $B'$ with an invalid cone, as described in Section~\ref{subsec:decomposing}.
\end{lemma}

\begin{proof}
Let $\mathcal{P}$ denote the truncated pseudomanifold corresponding to the cell decomposition containing $B$.
We first consider the case where the truncated bigon $B'$ forms a projective plane boundary component of $\mathcal{P}$;
in other words, we are considering the case where the vertex incident to $B$ does not get truncated, and hence $B'=B$.
We produce a copy $B_0$ of $B$ by an isotopy of $B$ that:
\begin{itemize}
\item
fixes the vertex; and
\item
pushes the rest of $B$ slightly into the interior of $\mathcal{P}$.
\end{itemize}
Let $\mathcal{R}$ denote the region that is swept out by this isotopy, and let $\mathcal{C}$ denote the quotient of $\mathcal{R}$ obtained by flattening $B$.

By construction, deleting $\mathcal{R}-B_0$ preserves the ambient space $\mathcal{P}$ up to homeomorphism, and replaces the boundary component $B$ with its copy $B_0$.
Thus, we see that flattening $B$ has the same topological effect as attaching a copy of $\mathcal{C}$ to $B$.
It therefore suffices to show that $\mathcal{C}$ is actually an invalid cone.
Our strategy for this is to express $\mathcal{R}$ as a union of lines such that after flattening $B$, these lines exhibit the structure of a cone over the projective plane $B_0$.

To this end, begin with $\mathcal{Q} := [0,1]^3$.
Let $\Lambda$ denote the set of lines of the form $\{x\}\times\{y\}\times[0,1]$, so that $\mathcal{Q}$ is a (disjoint) union of the lines in $\Lambda$.
Take $\sim$ to be the minimal equivalence relation satisfying the following:
\begin{itemize}
\item
$(x,0,0) \sim (x',0,0)$ for all $x,x'\in[0,1]$.
\item
$(x,1,0) \sim (x',1,0)$ for all $x,x'\in[0,1]$.
\item
$(x,y,1) \sim (x,y',1)$ for all $x,y,y'\in[0,1]$.
\item
$(0,y,z) \sim (1,1-y,z)$ for all $y,z\in[0,1]$.
\end{itemize}
In the quotient $\mathcal{Q}/{\sim}$, the $z=0$ rectangle becomes a projective plane $P_0$,
and the $y=0$ and $y=1$ rectangles become a pair of triangles that together form a projective plane $P_1$;
see Figure~\ref{fig:invalidConeQuotient}.
Indeed, we can identify this quotient with $\mathcal{R}$ in such a way that $P_0$ is identified with $B_0$ and $P_1$ is identified with $B$.

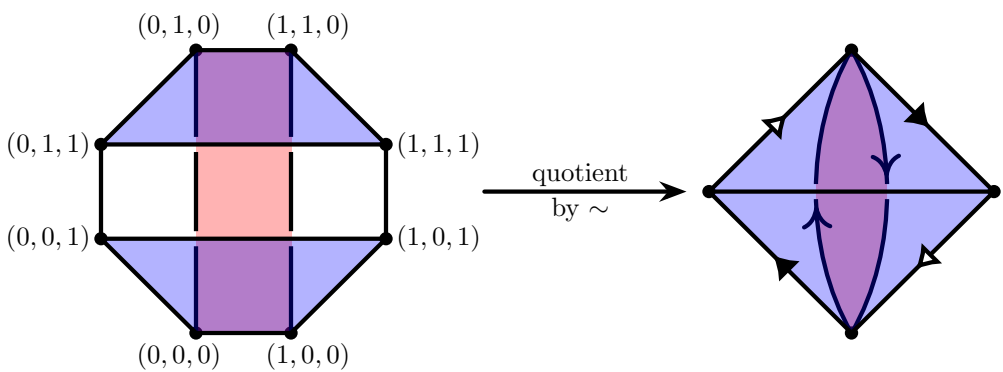
\begin{figure}[htbp]
\centering
	\begin{tikzpicture}

	\node[inner sep=0pt] (box) at (0,0) {
		\includegraphics[scale=1]{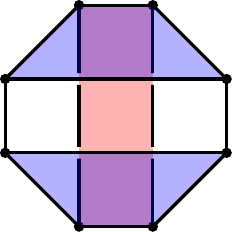}
	};
	\node[inner sep=0pt] (quotient) at (8,0) {
		\includegraphics[scale=1]{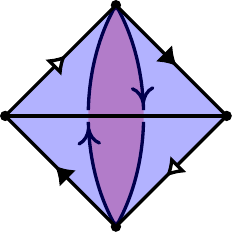}
	};

	\node[inner sep=1pt, below left] at ($(box.south)+(180:0.25)$) {$(0,0,0)$};
	\node[inner sep=1pt, below right] at ($(box.south)+(0:0.25)$) {$(1,0,0)$};
	\node[inner sep=1pt, above left] at ($(box.north)+(180:0.25)$) {$(0,1,0)$};
	\node[inner sep=1pt, above right] at ($(box.north)+(0:0.25)$) {$(1,1,0)$};
	\node[inner sep=1pt, below left] at ($(box.west)+(270:0.4)$) {$(0,0,1)$};
	\node[inner sep=1pt, below right] at ($(box.east)+(270:0.4)$) {$(1,0,1)$};
	\node[inner sep=1pt, above left] at ($(box.west)+(90:0.4)$) {$(0,1,1)$};
	\node[inner sep=1pt, above right] at ($(box.east)+(90:0.4)$) {$(1,1,1)$};

	\draw[ultra thick, line cap=round, -Stealth] ($(box.east)+(0:1.2)$)
		-- node[pos=0.475] (label){} ($(quotient.west)+(180:0.2)$);
	\node[above, inner sep=2pt] at (label.center) {quotient};
	\node[below, inner sep=2pt] at (label.center) {by $\sim$};

	\end{tikzpicture}
\caption{The region $\mathcal{R}$ can be constructed as a quotient of the box $\mathcal{Q}:=[0,1]^3$.
The $z=0$ rectangle becomes the projective plane $B_0$ (shaded red).
The $y=0$ and $y=1$ rectangles together become the projective plane $B$ (shaded blue).}
\label{fig:invalidConeQuotient}
\end{figure}

With this identification, we can express $\mathcal{R}$ as a union of the lines given by $\Lambda/{\sim}$.
Each of these lines joins a point in $B_0$ to a point on the curve $\{(x,\ast,1)\}$.
Moreover, we observe the following:
\begin{itemize}
\item
For each point of the form $(x,\ast,1)$, we have exactly two lines that join this point to the vertex of $B_0$.
The union of these lines is precisely the projective plane $B$.
\item
For every point in $B_0$ other than the vertex, there is a unique line that joins this point to the curve $\{(x,\ast,1)\}$.
\end{itemize}
After flattening $B$, the curve $\{(x,\ast,1)\}$ gets collapsed to a single ``apex'' point $a$, and each point in the projective plane $B_0$ is now joined to $a$ by a unique line.
This exhibits precisely the required cone structure on $\mathcal{C}$, and hence we conclude that flattening $B$ is topologically equivalent to filling $B'$ with this invalid cone $\mathcal{C}$.

What remains is the case where $B'$ forms a M\"{o}bius band in $\partial\mathcal{P}$.
In this case, flattening $B$ is still equivalent to attaching an invalid cone, except we need to account for the fact that the vertex incident to $B$ is truncated.
This truncation only removes a small neighbourhood of the vertex, and this removed neighbourhood is disjoint from the apex of the cone.
Thus, we still have an invalid cone $\mathcal{C}$, but the truncation means that there is a small disc $D\subset\partial\mathcal{C}$ that lies inside $\partial\mathcal{P}$,
and $\mathcal{C}$ is attached by identifying the M\"{o}bius bands $B'$ and $\partial\mathcal{C}-D$.
In other words, we again conclude that flattening $B$ is topologically equivalent to filling $B'$ with the invalid cone $\mathcal{C}$.
\end{proof}

\begin{proof}[Proof of Claim~\ref{claim:bigonSeparatePaths:glueRP2}]
Recall that in case~\ref{case:bigonSeparatePaths:glueRP2} of Claim~\ref{claim:bigonSeparatePaths},
the edges of $F$ are identified so that $F$ forms a projective plane.
We will see that the proof is almost identical to the proof of Claim~\ref{claim:bigonSinglePath:S2:glueRP2};
the main difference is that here, we end up working with embedded surfaces that are two-sided rather than one-sided.

Since $F$ forms a projective plane, the two vertices of $F$ are identified to form a single vertex $v$ in $\mathcal{D}_0$.
When $v$ is not ideal, the claim is easy to prove:
\begin{itemize}
\item If $v$ is boundary or invalid, then we are in case~\ref{case:bigonSeparatePaths:glueRP2:bdryInvalid}.
\item If $v$ is internal, then $F$ forms an embedded projective plane in the interior of $\mathcal{P}_0$.
Ungluing $F$ yields two projective plane remnants corresponding to the boundary bigon paths $F^\dagger_0$ and $F^\dagger_1$,
which tells us that $F$ forms a \emph{two-sided} projective plane in $\mathcal{P}_0$.
Moreover, by Lemma~\ref{lem:fillInvalidCone}, flattening $F^\dagger_0$ and $F^\dagger_1$
corresponds to filling these two projective plane boundary components with invalid cones.
Altogether, we see that flattening $F$ is topologically equivalent to decomposing along $F$.
This proves case~\ref{case:bigonSeparatePaths:glueRP2:internal}.
\end{itemize}

With that out of the way, suppose for the rest of this proof that $v$ is ideal;
we need to prove all the conclusions stated in case~\ref{case:bigonSeparatePaths:glueRP2:ideal}.
Observe that the truncated bigon associated to $F$ forms a properly embedded M\"{o}bius band $S$ in $\mathcal{P}_0$.
Consider the pseudomanifold $\mathcal{P}^\dagger$ obtained from $\mathcal{D}^\dagger$ by truncating the vertices in $g^{-1}(V_0)$;
viewing $\mathcal{P}^\dagger$ as a subset of $\mathcal{D}^\dagger$, for each $i\in\{0,1\}$ let $S^\dagger_i$ denote
the M\"{o}bius band in $\partial\mathcal{P}^\dagger$ given by $F^\dagger_i\cap\mathcal{P}^\dagger$.
Topologically, $\mathcal{P}^\dagger$ is obtained from $\mathcal{P}_0$ by cutting along the M\"{o}bius band $S$;
as shown in Figure~\ref{fig:bigonSeparatePaths:glueRP2:ideal:cutMobius},
this yields two remnants---namely, the M\"{o}bius bands $S^\dagger_0$ and $S^\dagger_1$---so $S$ must be a \emph{two-sided} Mobius band in $\mathcal{P}_0$.

\begin{figure}[htbp]
\centering
	\begin{tikzpicture}

	\node[inner sep=-5pt] (before) at (0,0) {
		\includegraphics[scale=1]{TruncatedMobius_Uncut.pdf}
	};
	\node[inner sep=-10pt] (after) at (4.2,0) {
		\includegraphics[scale=1]{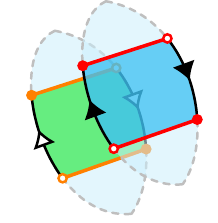}
	};
	\draw[thick, line cap=round, -Stealth] (before.east) -- (after.west);

	\end{tikzpicture}
\caption{The truncated bigon associated to $F$ forms a two-sided M\"{o}bius band $S$.
Cutting along $S$ yields two M\"{o}bius band remnants.}
\label{fig:bigonSeparatePaths:glueRP2:ideal:cutMobius}
\end{figure}

Consider the ideal boundary component $L$ of $\mathcal{P}_0$ given by truncating the vertex $v$, and let $\gamma$ denote the boundary curve of $S$;
we need to show that $\gamma$ forms a two-sided curve in $L$.
For this, it suffices to observe that cutting along $S$ has the effect of splitting $\gamma$ into two curves,
one bounding the M\"{o}bius band $S^\dagger_0$ and the other bounding the M\"{o}bius band $S^\dagger_1$.

All that remains is to understand the overall effect of flattening $F$.
We begin with the case where $\gamma$ bounds a disc $E$ in $L$.
In this case, we use Claim~\ref{claim:flattenBigonVertexCone} to flatten the $v$-cone over $\gamma$.
This reduces the operation of flattening $F$ to the operation of flattening a new bigon $F'$ given by pushing $F$ slightly away from $v$;
topologically, $F'$ is equivalent a properly embedded projective plane given by isotoping $S\cup E$ slightly off the boundary of $\mathcal{P}_0$.
Since the vertices of $F'$ are identified to form a single temporary internal vertex, flattening $F'$ has the same topological effect as
flattening $F$ in the case where $v$ is internal (case~\ref{case:bigonSeparatePaths:glueRP2:internal}).
In other words, $F'$ forms a \emph{two-sided} projective plane in $\mathcal{P}_0$,
and $\mathcal{P}_1$ is obtained from $\mathcal{P}_0$ by decomposing along this projective plane.
This completes the case where $\gamma$ bounds a disc in $L$.

For the case where $\gamma$ does \emph{not} bound a disc in $L$,
consider the pseudomanifold $\mathcal{P}^\ast$ obtained from $\mathcal{D}_1$ by truncating the vertices in $\varphi(V_0)$.
Topologically, by Lemma~\ref{lem:fillInvalidCone}, $\mathcal{P}^\ast$ is obtained from $\mathcal{P}^\dagger$ by filling
the M\"{o}bius bands $S^\dagger_0$ and $S^\dagger_1$ with invalid cones;
in other words, $\mathcal{P}^\ast$ is obtained from $\mathcal{P}_0$ by decomposing along the M\"{o}bius band $S$.
To see how $\mathcal{P}^\ast$ is related to $\mathcal{P}_1$, we compare the truncated vertex sets $\varphi(V_0)$ and $V_1$.
For this, let $L^\ast$ denote the surface obtained by decomposing $L$ along $\gamma$.
Claim~\ref{claim:flattenBigonVertexLink} tells us that the components of $L^\ast$ correspond to
boundary components of $\mathcal{P}^\ast$ given by truncating the vertices in $\varphi(v)$.
The only way $\mathcal{P}^\ast$ can differ from $\mathcal{P}_1$ is if $L^\ast$ has $2$-sphere components;
we need to fill each such $2$-sphere with a $3$-ball to recover $\mathcal{P}_1$ from $\mathcal{P}^\ast$.
This completes the proof of case~\ref{case:bigonSeparatePaths:glueRP2:ideal}.
\end{proof}

\section{Crushing surfaces of positive genus}\label{sec:genus}

Consider a normal surface $S$ in a $3$-manifold $\mathcal{M}$.
Roughly, our goal in this section is to give sufficient conditions under which crushing $S$ gives an ideal triangulation of a component of $\mathcal{M}-S$.
For this, we fix the following notation throughout this section:
\begin{itemize}
\item Let $\mathcal{M}$ be a (compact) $3$-manifold with no $2$-sphere boundary components.
If $\mathcal{M}$ is closed, let $\mathcal{T}$ be a closed triangulation of $\mathcal{M}$;
otherwise, if $\mathcal{M}$ is bounded, let $\mathcal{T}$ be an ideal triangulation of $\mathcal{M}$.
\item Let $S$ be a (possibly disconnected) separating normal surface in $\mathcal{T}$.
(Since $\mathcal{T}$ has no real boundary components, note that $S$ must be a closed surface.)
Assume that every component of $S$ is two-sided, and that none of these components are $2$-spheres.
\item Fix any particular component of $\mathcal{M}-S$ that meets each component of $S$ on exactly one side, and call it the \textbf{chosen region} for $S$;
also, for any (not necessarily normal) surface $E$ isotopic to $S$,
call the corresponding component of $\mathcal{M}-E$ the chosen region for $E$.
Let $X$ denote the compact $3$-manifold given by the closure of the chosen region for $S$.
\end{itemize}
The assumptions that we have made on $S$ and on the chosen region are not as restrictive as they might appear at first glance.
If $S$ has a component $C$ that is either non-separating or one-sided (or both), then we can always ``repair'' $S$ as follows:
build a new surface $\Sigma$ by replacing $C$ with the frontier of a regular neighbourhood of $C$.
Up to homeomorphism, each component of $\mathcal{M}-S$ appears as a component of
$\mathcal{M}-\Sigma$, so we lose nothing by repairing $S$ in this way.
A similar trick allows us to deal with components of $\mathcal{M}-S$ that meet a component of $S$ on both sides.

With this in mind, let $\mathcal{T}'$ denote the triangulation obtained by (destructively) crushing $S$.
As we hinted earlier, our goal is to give sufficient conditions for $\mathcal{T}'$ to include an ideal triangulation of $X$ as one of its components
(the precise statement is given in Theorem~\ref{thm:crushPositiveGenus}).

To this end, consider the cell decomposition $\mathcal{D}'$ given by \emph{non-destructively} crushing $S$.
One of the components $\mathcal{D}^\ast$ of $\mathcal{D}'$ gives a destructible ideal cell decomposition of $X$;
see Figure~\ref{fig:nonDestructiveCartoon}.
Call a $3$-cell in $\mathcal{D}'$ \textbf{benign} if and only if it belongs to the component $\mathcal{D}^\ast$.

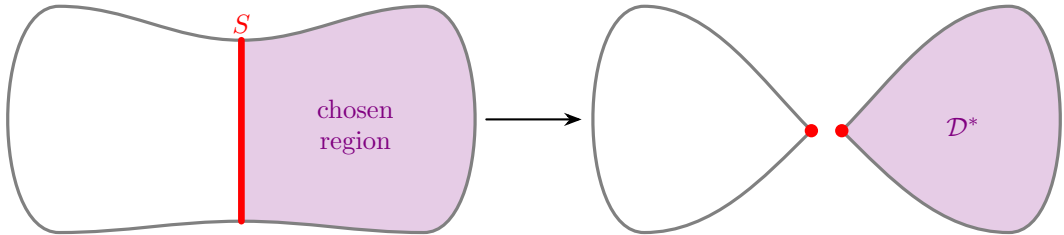
\begin{figure}[htbp]
\centering
	\begin{tikzpicture}

	\node[inner sep=-3pt] (before) at (0,0) {
		\includegraphics[scale=1]{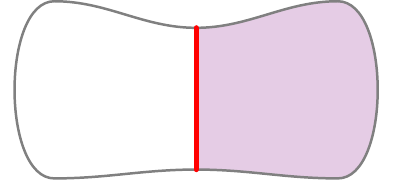}
	};
	\node[inner sep=-3pt] (after) at (7.7,0) {
		\includegraphics[scale=1]{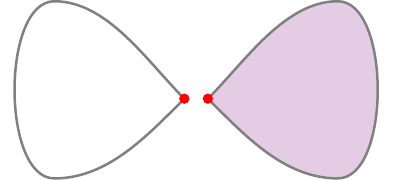}
	};
	\draw[thick, line cap=round, -Stealth] (before.east) -- (after.west);

	\node[violet, align=center] at (1.5,-0.1) {chosen\\region};
	\node[violet] at (9.5,-0.1) {$\mathcal{D}^\ast$};
	\node[red, below, inner sep=1pt] at (before.north) {$S$};

	\end{tikzpicture}
\caption{Schematic illustration of non-destructively crushing $S$, in the case where $S$ is connected.}
\label{fig:nonDestructiveCartoon}
\end{figure}

From the formulation of crushing given in Section~\ref{subsec:atomic}
(in particular, recall Definitions~\ref{defs:stages} and Lemma~\ref{lem:benCrushing}),
$\mathcal{T}'$ is obtained from $\mathcal{D}'$ by a finite sequence of atomic moves.
Since the atomic moves only either destroy or modify the existing $3$-cells in $\mathcal{D}'$,
we can naturally speak about benign $3$-cells in all of the intermediate
cell decompositions that we encounter as we perform the atomic moves.

We will also call a component of a cell decomposition \textbf{benign} if this component is built entirely from benign $3$-cells.
Whilst $\mathcal{D}'$ initially has exactly one benign component, namely $\mathcal{D}^\ast$,
this number can change as we perform atomic moves.

With this in mind, let $\mathcal{T}^\ast$ denote the triangulation consisting only of the benign components of $\mathcal{T}'$;
see Figure~\ref{fig:notationSecGenus} for a visual summary of all the notation we have just introduced.
We can now give a more precise statement of our main goal in this section:
we want to show that one of the components of $\mathcal{T}^\ast$ gives an ideal triangulation of $X$.

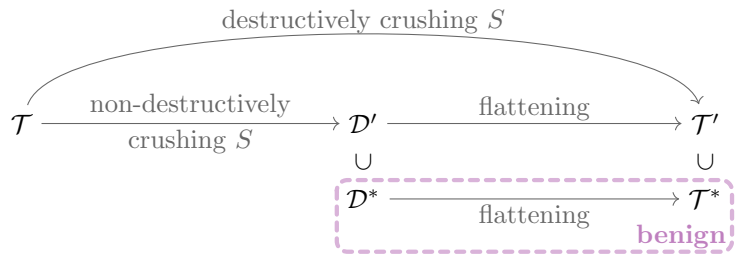
\begin{figure}[htbp]
\centering
	\begin{tikzpicture}[commutative diagrams/every diagram]

	\newcommand{\Shift}{-1}

	\node (t) at (0,0) {$\mathcal{T}$};
	\node (dp) at (4.5,0) {$\mathcal{D}'$};
	\node (tp) at (9,0) {$\mathcal{T}'$};
	\node (ds) at (4.5,\Shift) {$\mathcal{D}^\ast$};
	\node (ts) at (9,\Shift) {$\mathcal{T}^\ast$};

	\path[commutative diagrams/.cd, every arrow, every label]
		(t) edge[darkgray!40!gray] node[above]{non-destructively} node[below]{crushing $S$} (dp)
		(t) edge[darkgray!40!gray, controls={+(0.5,1.5) and +(-0.5,1.5)}] node[inner sep=0pt] {destructively crushing $S$} (tp)
		(dp) edge[darkgray!40!gray] node[inner sep=1pt] {flattening} (tp)
		(ds) edge[darkgray!40!gray] node[inner sep=2pt, swap] {flattening} (ts)
	;

	\node at ($(dp)!0.5!(ds)$) {\rotatebox{90}{$\subset$}};
	\node at ($(tp)!0.5!(ts)$) {\rotatebox{90}{$\subset$}};

	\draw[ultra thick, rounded corners, line cap=round, violet!30!white, dashed] (4.15,0.25+\Shift) rectangle (9.35,-0.7+\Shift);
	\node[left, violet!50!white, inner sep=2pt] at (9.35,-0.5+\Shift) {\textbf{benign}};

	\end{tikzpicture}
\caption{Some notation that we use throughout Section~\ref{sec:genus}.}
\label{fig:notationSecGenus}
\end{figure}

The proof boils down to checking that we do not make ``drastic'' topological changes
when performing atomic moves on an ideal cell decomposition of $X$.
In all but one of the cases, it is enough to require that $X$ satisfies the following conditions:
\begin{itemize}
\item it is irreducible and $\partial$-irreducible; and
\item it contains no essential annuli and no two-sided properly embedded M\"{o}bius bands.
\end{itemize}
The one difficult case is when we flatten a bigon whose corresponding truncated bigon forms a boundary-parallel annulus in $X$;
we will discuss how we circumvent this difficulty in Section~\ref{subsec:badBigonPaths}.
We then put everything together to prove Theorem~\ref{thm:crushPositiveGenus} in Section~\ref{subsec:crushBenign}.

\subsection{Avoiding bad bigon paths}\label{subsec:badBigonPaths}

Throughout the rest of Section~\ref{sec:genus}, call a bigon path (recall Definitions~\ref{defs:bigonPath}) \textbf{bad} if
it is internal and its corresponding truncated bigon path forms a boundary-parallel annulus.
Using this terminology, the difficult case that we mentioned earlier is when we flatten a bad bigon path of length one.
In Section~\ref{subsec:crushBenign}, we will see that the only way to have a bad bigon path of length one is if
the cell decomposition $\mathcal{D}^\ast$ initially contained a bad bigon path (of some arbitrary length).
With this in mind, our goal is to cut this problem off at the source:
we will give conditions on the surface $S$ that will ensure that $\mathcal{D}^\ast$ does not contain any bad bigon paths.

Roughly, the idea is that if $\mathcal{D}^\ast$ contains a bad bigon path, then we can ``push'' or ``expand'' $S$ further into the chosen region;
thus, we would like to ensure that $S$ cannot be ``expanded'' in this way.
To make this idea precise, we introduce the following terminology for any normal surface $E$ in the isotopy class of $S$:
\begin{itemize}
\item Let $E'$ be any normal surface that is, up to normal isotopy, disjoint from $E$.
We call $E'$ an \textbf{expansion} of $E$ if it is isotopic to $E$, but cannot be normally isotoped to lie entirely outside the chosen region for $E$;
see Figure~\ref{fig:expansionCartoon}.
\item Call $E$ \textbf{maximal} if it does not admit such an expansion.
\end{itemize}

\begin{figure}[htbp]
\centering
	\begin{tikzpicture}

	\node at (0,0) {
		\includegraphics[scale=1]{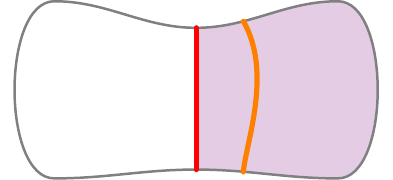}
	};

	\node[violet, align=center] at (2.1,-0.1) {chosen\\region};
	\node[red, below, inner sep=1pt] at (before.north) {$E$};
	\node[orange!60!red, below, inner sep=0pt] at ($(before.north)+(0.8,0.1)$) {$E'$};

	\end{tikzpicture}
\caption{Let $E$ and $E'$ be two disjoint normal surfaces in the same isotopy class.
If $E'$ cannot be \emph{normally} isotoped to lie outside the chosen region for $E$, then it is an expansion of $E$.}
\label{fig:expansionCartoon}
\end{figure}
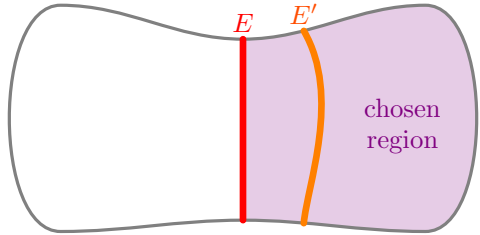

In Lemma~\ref{lem:badBigonPaths}, we will show that to avoid bad bigon paths
in $\mathcal{D}^\ast$, it is enough to assume that $X$ is irreducible, and that $S$ is incompressible and maximal.
Before we do so, it is worth noting that the maximality assumption is not very restrictive.
Specifically, the following result says that, up to isotopy, we can always choose $S$ to be maximal:

\begin{lemma}\label{lem:maximal}
There is a maximal normal surface in the isotopy class of $S$.
\end{lemma}

One way to prove Lemma~\ref{lem:maximal} would be to appeal to Kneser's finiteness theorem
(as is done, for instance, in~\cite[pp.~160--161]{JacoRubinstein2003});
however, we do not actually need the full strength of this theorem.
The following simple proof distils precisely the part of Kneser's finiteness theorem that is necessary:

\begin{proof}[Proof of Lemma~\ref{lem:maximal}]
If $S$ is itself maximal, then there is nothing to prove;
thus, assume for the rest of this proof that $S$ is not maximal.
Let $E_0=S$, and consider any sequence $E_0,\ldots,E_n$ of normal surfaces
such that for each $i\in\{1,\ldots,n\}$, $E_i$ is an expansion of $E_{i-1}$.
The idea is to show that such a sequence cannot be extended indefinitely,
which means that after extending as much as possible, the final entry of the sequence will be maximal.

To this end, consider the surface $E_0\cup\cdots\cup E_n$;
this is a normal surface, since $E_0,\ldots,E_n$ are mutually disjoint, up to normal isotopy.
Let $\mathcal{C}$ denote the induced cell decomposition obtained by cutting along $E_0\cup\cdots\cup E_n$.
For each $i\in\{1,\ldots,n\}$, let $B_i$ denote the component of $\mathcal{C}$ given by the trivial $I$-bundle between $E_i$ and $E_{i-1}$.
Since $E_i$ and $E_{i-1}$ are not normally isotopic, $B_i$ must contain at least one non-parallel cell.
However, each tetrahedron of $\mathcal{T}$ gives rise to at most six non-parallel cells,
so we see that $n\leqslant6\lvert\mathcal{T}\rvert$.
This implies that there is a sequence $E_0,\ldots,E_k$ of expansions whose length $k$ is maximum among all such sequences;
the surface $E_k$ must therefore be a maximal normal surface in the isotopy class of $S$,
otherwise we would be able to extend the sequence by adding an expansion of $E_k$.
\end{proof}

\begin{lemma}\label{lem:badBigonPaths}
If $X$ is irreducible, $S$ is maximal, and the remnants of $S$ are incompressible in $X$, then $\mathcal{D}^\ast$ contains no bad bigon paths.
\end{lemma}

\begin{proof}
Suppose $\mathcal{D}^\ast$ contains a bad bigon path $\mathcal{F}$.
We will show that $S$ cannot be maximal by using $\mathcal{F}$ to construct an expansion of $S$.

To do this, let $\mathcal{D}$ denote the cell decomposition of $X$ induced by $S$,
and let $q:\mathcal{D}\to\mathcal{D}^\ast$ be the quotient map given by non-destructively crushing $S$.
Observe that $q^{-1}(\mathcal{F})$ realises an annulus $A$ that:
\begin{itemize}
\item consists of bridge faces (as defined in Definitions~\ref{defs:inducedCells}) in the $2$-skeleton of $\mathcal{D}$; and
\item is parallel to an annulus $A^\|$ lying entirely inside a component of $S$.
\end{itemize}
We now aim to build a sequence $S_1,S_2,S_3,S_4$ of surfaces isotopic to $S$, each of which is ``expanded further'' into the chosen region than the last.
Our goal is for $S_4$ to be the required expansion of $S$.
We achieve this as follows:
\begin{enumerate}[label={(\arabic*)}]
\item Let $S_1$ be the surface obtained from $S$ by replacing $A^\|$ with $A$.
\item Consider the boundary $B$ of a regular neighbourhood of $S\cup A$, and let $S_2$ be
the union of the components of $B$ that lie inside the chosen region for $S_1$;
see Figure~\ref{fig:maximalProofCell}.
Since the chosen region for $S_1$ meets each component of $S_1$ on exactly one of its two sides, observe that $S_2$ is isotopic to $S$.
By Theorem~\ref{thm:subcomplexBarrier}, $B$ is a barrier for any component of $\mathcal{M}-B$ that does not meet $S\cup A$;
observe that the chosen region $\mathcal{N}$ for $S_2$ is one such component of $\mathcal{M}-B$.
\item Let $S_3$ be the surface given by isotoping $S_2$ slightly into $\mathcal{N}$.
\item Using the barrier $B$, normalise $S_3$ to obtain a normal surface $E$ in $\mathcal{N}$.
Since the remnants of $S$ are incompressible in $X$, we know that $S_3$ must be incompressible in $\mathcal{N}$.
This, together with the fact that $X$ is irreducible, implies that after deleting any $2$-sphere components of $E$, we must be left with a normal surface $S_4$ isotopic to $S$.
By construction, $S_4$ is disjoint from $S$, and it cannot be normally isotoped to lie outside the chosen region for $S$, so it is an expansion of $S$.
\qedhere
\end{enumerate}
\end{proof}

\begin{figure}[htbp]
\centering
	\includegraphics[scale=1]{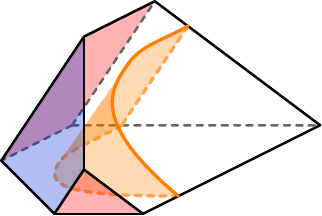}
\caption{An example of an induced cell $\Delta$ that meets the chosen region for $S_1$.
The red faces are parts of $S_1$ that originated from $S$, and the blue bridge face is a part of $S_1$ that originated from $A$.
The intersection of $S_2$ with $\Delta$ is shaded orange.
In this example, the orange piece forms an elementary triangle, which means that this piece is preserved by the normalisation procedure;
thus, the final surface $S_4$ will include this orange triangle among its elementary discs.}
\label{fig:maximalProofCell}
\end{figure}

\subsection{Crushing the benign components}\label{subsec:crushBenign}

In Theorem~\ref{thm:crushPositiveGenus} below, we give sufficient conditions so that after crushing $S$,
one of the components of the triangulation $\mathcal{T}^\ast$ gives an ideal triangulation of $X$.
Specifically, our proof relies on the following:
\begin{itemize}
\item We require that $S$ is maximal.
As discussed in Section~\ref{subsec:badBigonPaths}, this is not a serious restriction, thanks to Lemma~\ref{lem:maximal}.
\item We require that $X$ is irreducible, $\partial$-irreducible and \textbf{anannular}
(i.e., $X$ contains no essential annuli).
These are quite common ``niceness'' conditions for $3$-manifolds.
It is worth noting that we do \emph{not} need to assume that $X$ is \textbf{atoroidal} (i.e., $X$ contains no essential tori).
\item We require that $X$ contains no two-sided properly embedded M\"{o}bius bands.
This condition holds for all orientable $3$-manifolds, but sometimes fails for non-orientable $3$-manifolds.
\item We require that $X$ contains no two-sided properly embedded projective planes.
This follows from the previous condition, together with the fact that $X$ has at least one boundary component (given by a remnant of $S$).
To see why, suppose $X$ contains a two-sided projective plane $E$.
Consider a path $\gamma$ that starts at a point $p_0$ in $E$ and ends at a point $p_1$ in $\partial X$.
Remove a small disc around $p_0$ from $E$, and replace it with a thin tube that ``follows'' the path $\gamma$
and ends with a curve that bounds a small disc around $p_1$ in $\partial X$;
see Figure~\ref{fig:twoSidedRP2Tube}.
This turns $E$ into a two-sided properly embedded M\"{o}bius band in $X$.
\end{itemize}

\begin{figure}[htbp]
\centering
	\begin{tikzpicture}

	\node[inner sep=0pt] (before) at (0,0) {
		\includegraphics[scale=1]{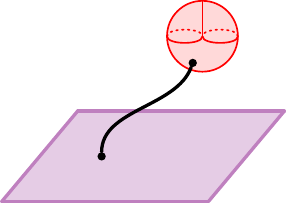}
	};
	\node[inner sep=0pt] (after) at (6,0) {
		\includegraphics[scale=1]{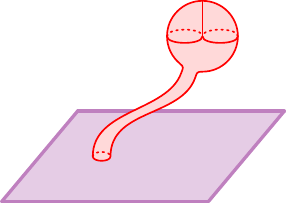}
	};
	\draw[thick, line cap=round, -Stealth] ($(before.east)+(0,-0.5)$) -- ($(after.west)+(0.7,-0.5)$);

	\end{tikzpicture}
\caption{Schematic illustration of using a tube to turn a two-sided projective plane into a two-sided M\"{o}bius band.}
\label{fig:twoSidedRP2Tube}
\end{figure}

\crushPositiveGenus*

\begin{proof}
Throughout this proof, call a cell decomposition \textbf{acceptable} if:
\begin{itemize}
\item one of its components is an ideal cell decomposition of $X$; and
\item every other component is a (closed) cell decomposition of the $3$-sphere.
\end{itemize}
Recall from Lemma~\ref{lem:benCrushing} that the procedure of flattening $\mathcal{D}^\ast$ to get $\mathcal{T}^\ast$
can be realised by a finite sequence of atomic moves.
Thus, our strategy will be to inductively prove that each atomic move preserves the property of being acceptable.

For the proof to work, we actually need to prove slightly more than this.
The problem is that flattening a bad bigon path (as defined in Section~\ref{subsec:badBigonPaths})
has the topological effect of decomposing along a boundary-parallel annulus,
which could potentially yield a cell decomposition that is no longer acceptable.
To circumvent this problem, we will show by induction that performing any number of atomic moves on $\mathcal{D}^\ast$
always yields an acceptable cell decomposition \emph{that contains no bad bigon paths}.

For the base case, consider the initial cell decomposition $\mathcal{D}^\ast$
(which is obtained by performing zero atomic moves).
Recall that $\mathcal{D}^\ast$ is an ideal cell decomposition of $X$
(with no extra $3$-sphere components), so it is acceptable.
By Lemma~\ref{lem:badBigonPaths}, we also know that $\mathcal{D}^\ast$ contains no bad bigon paths.

For the inductive step, assume that we have some acceptable cell decomposition $\mathcal{D}_0$ that contains no bad bigon paths.
We need to show that performing any atomic move on $\mathcal{D}_0$ yields
a new acceptable cell decomposition $\mathcal{D}_1$ that has no bad bigon paths.
In particular, to show that $\mathcal{D}_1$ remains acceptable, we will show that an atomic move always either:
has no topological effect on the truncated pseudomanifold, or changes the truncated pseudomanifold by adding or removing a $3$-sphere component.

Throughout the rest of this proof, let $F$ denote the triangular pillow, bigon pillow or bigon face in $\mathcal{D}_0$ that we flatten to obtain $\mathcal{D}_1$.
Let $\varphi$ denote the flattening map associated to this atomic move,
and let $\psi$ denote the inverse flattening map.
For each $i\in\{0,1\}$, let $\mathcal{P}_i$ denote the truncated pseudomanifold of $\mathcal{D}_i$.
This notation is summarised in Figure~\ref{fig:notationBenignProof}.

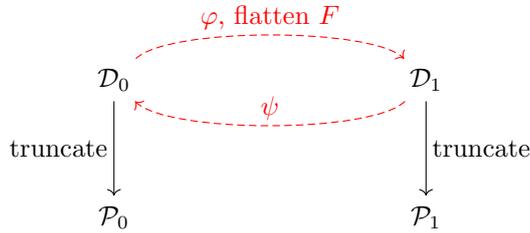
\begin{figure}[htbp]
\centering
\begin{tikzcd}[row sep=huge, column sep=huge,]
\mathcal{D}_0
\arrow[rr, "\text{$\varphi$, flatten $F$}", red, dashed, controls={+(45:1) and +(135:1)}]
\arrow[d, swap, "\text{truncate}"]
&& \mathcal{D}_1
\arrow[ll, swap, "\psi", red, dashed, controls={+(225:1) and +(315:1)}]
\arrow[d, "\text{truncate}"] \\
\mathcal{P}_0
&& \mathcal{P}_1
\end{tikzcd}
\caption{Notation for the inductive step in Theorem~\ref{thm:crushPositiveGenus}.}
\label{fig:notationBenignProof}
\end{figure}

We first consider the case where $F$ is a triangular pillow.
Recall from Lemma~\ref{lem:flattenTriangularPillows} that the effect of flattening $F$ depends on whether the two triangular faces of $F$ are identified:
\begin{itemize}
\item If the faces of $F$ are \emph{not} identified,
then we are in case~\ref{case:triangularPillowHomeo} of Lemma~\ref{lem:flattenTriangularPillows}.
Thus, the truncated pseudomanifolds of $\mathcal{D}_0$ and $\mathcal{D}_1$ are homeomorphic,
which implies that $\mathcal{D}_1$ is acceptable.
Suppose for the sake of contradiction that $\mathcal{D}_1$ contains a bad bigon path $B_1$.
Observe that $B_1$ meets the triangle $\varphi(F)$ in some (possibly empty) subset of the edges of this triangle,
which implies that $\psi(B_1)$ is a bad bigon path in $\mathcal{D}_0$;
this violates the inductive hypothesis, and hence shows that $\mathcal{D}_1$ cannot contain a bad bigon path.
\item If the faces of $F$ \emph{are} identified,
then we are in case~\ref{case:triangularPillowClosed} of Lemma~\ref{lem:flattenTriangularPillows}:
$F$ forms a (closed) cell decomposition of either $S^3$ or $L_{3,1}$.
Since the only closed components of $\mathcal{D}_0$ are $3$-spheres, $\mathcal{D}_1$ must
be obtained from $\mathcal{D}_0$ by deleting the $3$-sphere component given by $F$;
thus, $\mathcal{D}_1$ is acceptable.
Moreover, since the ideal component of $\mathcal{D}_0$ is left entirely untouched by the operation of flattening $F$,
we see that $\mathcal{D}_1$ cannot contain any bad bigon paths.
\end{itemize}
This completes the inductive step for the case where $F$ is a triangular pillow.

The case where $F$ is a bigon pillow is similar, but slightly more involved.
Recall from Lemma~\ref{lem:flattenBigonPillows} that the effect of flattening $F$ depends on whether the two bigon faces of $F$ are identified:
\begin{itemize}
\item If the faces of $F$ are \emph{not} identified,
then we are in case~\ref{case:bigonPillowHomeo} of Lemma~\ref{lem:flattenBigonPillows}.
Thus, the truncated pseudomanifolds of $\mathcal{D}_0$ and $\mathcal{D}_1$ are homeomorphic, which implies that $\mathcal{D}_1$ is acceptable.
Suppose for the sake of contradiction that $\mathcal{D}_1$ contains a bad bigon path $B_1$.
If $B_1$ is disjoint from the interior of the bigon $\varphi(F)$, then observe that $\psi(B_1)$ is a bad bigon path in $\mathcal{D}_0$.
This would violate the inductive hypothesis, so we conclude that the bigon $\varphi(F)$ must be part of the bigon path $B_1$;
this situation is illustrated in Figure~\ref{fig:bigonPathInverseBigonPillow}.
Consider the internal bigon path $B_0$ in $\mathcal{D}_0$ given by $\psi(B_1)-F$;
the ends of $B_0$ are precisely the two edges incident to the bigon pillow $F$.
Observe that augmenting $B_0$ with one of the two bigon faces of $F$ gives a bad bigon path in $\mathcal{D}_0$.
This again violates the inductive hypothesis, so we conclude that $\mathcal{D}_1$ cannot contain a bad bigon path.
\item If the faces of $F$ \emph{are} identified,
then we are in either case~\ref{case:bigonPillowClosed} or case~\ref{case:bigonPillowIdeal} of Lemma~\ref{lem:flattenBigonPillows}.
Actually, $F$ cannot form an ideal cell decomposition of $\mathbb{R}P^2\times[0,1]$ (case~\ref{case:bigonPillowIdeal})
because such a component would contain a two-sided projective plane.
Thus, we must be in case~\ref{case:bigonPillowClosed}: $F$ must form a (closed) cell decomposition of either $S^3$ or $\mathbb{R}P^3$.
By assumption, the only closed components of $\mathcal{D}_0$ are $3$-spheres,
so $\mathcal{D}_1$ must be obtained from $\mathcal{D}_0$ by deleting the $3$-sphere component given by $F$;
this shows that $\mathcal{D}_1$ is acceptable.
We also see that flattening $F$ leaves the ideal component of $\mathcal{D}_0$ entirely untouched,
which implies that $\mathcal{D}_1$ contains no bad bigon paths.
\end{itemize}
This completes the inductive step for the case where $F$ is a bigon pillow.

\begin{figure}[htbp]
\centering
	\begin{tikzpicture}

	\node[inner sep=-15pt] (before) at (0,0) {
		\includegraphics[scale=1]{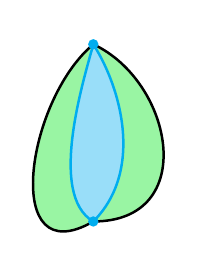}
	};
	\node[inner sep=-15pt] (after) at (3.6,0) {
		\includegraphics[scale=1]{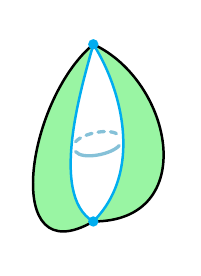}
	};
	\draw[thick, line cap=round, -Stealth] (before.east) -- node[above, inner sep=1pt]{$\psi$} (after.west);

	\end{tikzpicture}
\caption{If $F$ is a bigon pillow such that the bigon face $\varphi(F)$ forms part of a bad bigon path $B_1$,
then $\psi(B_1)$ contains a (bad) bigon path that is topologically equivalent to $B_1$.}
\label{fig:bigonPathInverseBigonPillow}
\end{figure}

With the pillow cases out of the way, all that remains is to consider the case where $F$ is a bigon face.
As in Sections~\ref{subsubsec:bigonSinglePath}
and~\ref{subsubsec:bigonSeparatePaths}, we divide our study into cases depending on whether
the two new boundary bigons given by ungluing $F$ form a single boundary bigon path or two separate boundary bigon paths.

First, suppose ungluing $F$ yields a single boundary bigon path.
Since $\mathcal{D}_0$ is valid and has no boundary edges, Claim~\ref{claim:bigonSinglePath}
tells us that $\varphi(F)$ is a single internal edge in $\mathcal{D}_1$.
Validity of $\mathcal{D}_0$ also tells us that we are in either
case~\ref{case:bigonSinglePath:S2:glueD2},~\ref{case:bigonSinglePath:S2:glueRP2}
or~\ref{case:bigonSinglePath:S2:glueS2} of Claim~\ref{claim:bigonSinglePath}.
Actually, cases~\ref{case:bigonSinglePath:S2:glueRP2} and~\ref{case:bigonSinglePath:S2:glueS2} are both impossible:
\begin{itemize}
\item Consider case~\ref{case:bigonSinglePath:S2:glueRP2} of Claim~\ref{claim:bigonSinglePath}.
In this case, $F$ forms a \textbf{projective plane} in $\mathcal{D}_0$, and Claim~\ref{claim:bigonSinglePath:S2:glueRP2}
tells us that the two vertices of $F$ are identified to form a single vertex $v$.
Moreover, since $\mathcal{D}_0$ is valid and has no boundary vertices, we must be in either case~\ref{case:bigonSinglePath:S2:glueRP2:internal} or
case~\ref{case:bigonSinglePath:S2:glueRP2:ideal} of Claim~\ref{claim:bigonSinglePath:S2:glueRP2}:
	\begin{description}[font=\normalfont]
	\item[\ref{case:bigonSinglePath:S2:glueRP2:internal}]
	If $v$ is internal, then Claim~\ref{claim:bigonSinglePath:S2:glueRP2} tells us that
	$F$ forms a one-sided properly embedded projective plane in $\mathcal{P}_0$.
	In fact, $F$ lies in $X$, since the $3$-sphere components
	of $\mathcal{P}_0$ cannot contain an embedded projective plane.
	Observe that a small regular neighbourhood $N$ of $F$ is homeomorphic to $\mathbb{R}P^3$ minus a small open $3$-ball,
	and that the frontier of $N$ forms a properly embedded $2$-sphere $E$ in $X$.
	Notice that $E$ does not bound a $3$-ball in $X$:
	the region on the ``inside'' of $E$ is $N$, and the region on the ``outside'' contains all the boundary components of $X$.
	This contradicts the assumption that $X$ is irreducible.
	\item[\ref{case:bigonSinglePath:S2:glueRP2:ideal}]
	If $v$ is ideal, then Claim~\ref{claim:bigonSinglePath:S2:glueRP2} tells us that the truncated bigon associated to $F$
	forms a one-sided properly embedded M\"{o}bius band $S$ in $\mathcal{P}_0$.
	In fact, $S$ lies in $X$, since the $3$-sphere components of $\mathcal{P}_0$ have empty boundary.
	Consider the annulus $A$ given by the frontier of a small regular neighbourhood $N$ of $S$.
	Since $X$ is anannular, $A$ must be either compressible or boundary-parallel.
	We claim that neither case is possible:
		\begin{itemize}
		\item If $A$ is compressible, then consider an essential compression disc $E$ for $A$.
		Up to isotopy, the boundary curve of $E$ coincides with the boundary curve
		of a M\"{o}bius band $B$ in $N$ given by thickening the core curve of $S$;
		see Figure~\ref{fig:oneSidedMobiusImpossible}.
		Observe that $E\cup B$ forms an embedded projective plane in $X$.
		By assumption, this projective plane cannot be two-sided.
		However, it also cannot be one-sided, since this would contradict
		the fact that $X$ is irreducible, by the same argument as before.
		Thus, we conclude that $A$ cannot be compressible.
		\item If $A$ is boundary-parallel, then isotoping $A$ into the boundary shows that
		the entire component $X$ is homeomorphic to a regular neighbourhood of the M\"{o}bius band $S$.
		But this means that $X$ is a solid torus, which contradicts the assumption that $X$ is $\partial$-irreducible.
		\end{itemize}
	\end{description}
The upshot is that, under our assumptions on $X$, case~\ref{case:bigonSinglePath:S2:glueRP2}
of Claim~\ref{claim:bigonSinglePath} can never occur.

\begin{figure}[htbp]
\centering
	\includegraphics[scale=1]{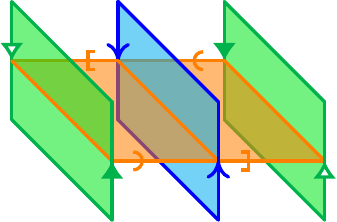}
\caption{When the truncated bigon associated to $F$ forms a one-sided M\"{o}bius band $S$ (blue),
the frontier of a small regular neighbourhood of $S$ forms an annulus $A$ (green).
If $A$ admits an essential compression disc $E$, then the boundary curve of $E$ coincides with the boundary of the M\"{o}bius band $B$ (orange).}
\label{fig:oneSidedMobiusImpossible}
\end{figure}
\item Consider case~\ref{case:bigonSinglePath:S2:glueS2} of Claim~\ref{claim:bigonSinglePath}.
In this case, $F$ forms a \textbf{\textup{2}-sphere} in $\mathcal{D}_0$, and Claim~\ref{claim:bigonSinglePath:S2:glueS2} tells us that
the truncated bigon associated to $F$ forms a one-sided properly embedded annulus in $\mathcal{P}_0$.
But this gives an essential annulus in $\mathcal{P}_0$, which is impossible since
$X$ is anannular and none of the $3$-sphere components of $\mathcal{P}_0$ can contain properly embedded annuli.
\end{itemize}
We are left with case~\ref{case:bigonSinglePath:S2:glueD2} of Claim~\ref{claim:bigonSinglePath}.
In this case, $F$ forms a \textbf{disc} in $\mathcal{D}_0$, and Claim~\ref{claim:bigonSinglePath:S2:glueD2} tells us that
the truncated pseudomanifolds $\mathcal{P}_0$ and $\mathcal{P}_1$ are homeomorphic, and hence that $\mathcal{D}_1$ is acceptable.
Suppose for the sake of contradiction that $\mathcal{D}_1$ contains a bad bigon path $B_1$.
We note that $B_1$ must contain the edge $\varphi(F)$;
otherwise, flattening $F$ would leave $B_1$ untouched, which would imply that
$\psi(B_1)$ is a bad bigon path in $\mathcal{D}_0$, contradicting the inductive hypothesis.
Thus, $\psi(B_1)$ must contain the bigon face $F$, and we are left with the following two possibilities:
\begin{itemize}
\item If $\psi(B_1)$ itself forms an internal bigon path in $\mathcal{D}_0$, then observe that flattening $F$
reduces the length of this bigon path by one, but has no topological effect;
see Figure~\ref{subfig:inverseDiscFlat3}.
Thus, $\psi(B_1)$ is bad bigon path in $\mathcal{D}_0$, contradicting the inductive hypothesis.
\item Otherwise, there must be an internal bigon path $B_0$ in $\mathcal{D}_0$
such that $B_0\cap F$ is a single edge and $B_0\cup F=\psi(B_1)$;
see Figure~\ref{subfig:inverseDiscTripod}.
In this case, we have $\varphi(B_0)=B_1$, and flattening $F$ essentially leaves $B_0$ untouched,
which means that $B_0$ is a bad bigon path in $\mathcal{D}_0$, again contradicting the inductive hypothesis.
\end{itemize}
Thus, we conclude that $\mathcal{D}_1$ contains no bad bigon paths.
This completes the inductive step for the case where ungluing $F$ yields a single boundary bigon path.

\begin{figure}[htbp]
\centering
	\begin{subfigure}[t]{0.45\textwidth}
	\centering
		\begin{tikzpicture}

		\node[inner sep=-15pt] (before) at (0,0) {
			\includegraphics[scale=1]{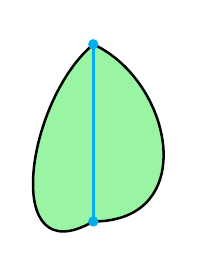}
		};
		\node[inner sep=-15pt] (after) at (3.6,0) {
			\includegraphics[scale=1]{BigonPathFlat3.pdf}
		};
		\draw[thick, line cap=round, -Stealth] (before.east) -- node[above, inner sep=1pt]{$\psi$} (after.west);

		\end{tikzpicture}
	\caption{If $\psi(B_1)$ is a bigon path, then it is topologically equivalent to $B_1$.}
	\label{subfig:inverseDiscFlat3}
	\end{subfigure}
	\hfill
	\begin{subfigure}[t]{0.45\textwidth}
	\centering
		\begin{tikzpicture}

		\node[inner sep=-15pt] (before) at (0,0) {
			\includegraphics[scale=1]{BigonPathFlat2.pdf}
		};
		\node[inner sep=-15pt] (after) at (3.6,0) {
			\includegraphics[scale=1]{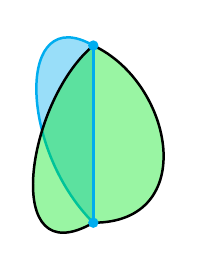}
		};
		\draw[thick, line cap=round, -Stealth] (before.east) -- node[above, inner sep=1pt]{$\psi$} (after.west);

		\end{tikzpicture}
	\caption{If there is a bigon path $B_0$ such that $B_0\cap F$ is a single edge and $B_0\cup F=\psi(B_1)$,
	then $B_0$ is topologically equivalent to $B_1$.}
	\label{subfig:inverseDiscTripod}
	\end{subfigure}
\caption{The two cases when the bigon face $F$ is a disc that forms part of $\psi(B_1)$.}
\label{fig:bigonPathInverseDisc}
\end{figure}

Suppose now that ungluing $F$ yields two separate boundary bigon paths.
Since $\mathcal{D}_0$ is valid and has no boundary edges, we are in either case~\ref{case:bigonSeparatePaths:glueS2} or
case~\ref{case:bigonSeparatePaths:glueRP2} of Claim~\ref{claim:bigonSeparatePaths}.
Actually, case~\ref{case:bigonSeparatePaths:glueRP2} is impossible,
since all three possibilities in Claim~\ref{claim:bigonSeparatePaths:glueRP2} contradict the assumption that $\mathcal{D}_0$ is acceptable:
\begin{itemize}
\item Possibility~\ref{case:bigonSeparatePaths:glueRP2:internal} requires
$\mathcal{P}_0$ to contain a two-sided properly embedded projective plane.
\item Possibility~\ref{case:bigonSeparatePaths:glueRP2:ideal} requires
$\mathcal{P}_0$ to contain a two-sided properly embedded M\"{o}bius band.
\item Possibility~\ref{case:bigonSeparatePaths:glueRP2:bdryInvalid} requires
$\mathcal{D}_0$ to contain either a boundary vertex or an invalid vertex.
\end{itemize}
We are left with case~\ref{case:bigonSeparatePaths:glueS2} of Claim~\ref{claim:bigonSeparatePaths}.
In this case, $F$ forms a \textbf{\textup{2}-sphere} in $\mathcal{D}_0$,
and $\varphi(F)$ consists of two distinct internal edges in $\mathcal{D}_1$.
Since $\mathcal{D}_0$ is valid and has no boundary vertices, we are in either
case~\ref{case:bigonSeparatePaths:glueS2:onlyInternal},~\ref{case:bigonSeparatePaths:glueS2:internalIdeal}
or~\ref{case:bigonSeparatePaths:glueS2:onlyIdeal} of Claim~\ref{claim:bigonSeparatePaths:glueS2}.
In cases~\ref{case:bigonSeparatePaths:glueS2:onlyInternal} and~\ref{case:bigonSeparatePaths:glueS2:internalIdeal},
it is relatively easy to see that $\mathcal{D}_1$ is acceptable and contains no bad bigon paths:
\begin{itemize}
\item Consider case~\ref{case:bigonSeparatePaths:glueS2:onlyInternal} of Claim~\ref{claim:bigonSeparatePaths:glueS2}.
In this case, $F$ is only incident to internal vertices, and $\mathcal{P}_1$ is
obtained from $\mathcal{P}_0$ by decomposing along a properly embedded $2$-sphere $E$.
Since $X$ is irreducible and every other component of $\mathcal{P}_0$ is a $3$-sphere, we see that $E$ bounds a $3$-ball,
which implies that flattening $F$ only changes the truncated pseudomanifold by creating a new $3$-sphere component.
Thus, $\mathcal{D}_1$ remains acceptable.
Moreover, observe that $\varphi(F)$ is not incident to any ideal vertices of $\mathcal{D}_1$,
which means that any bad bigon path $B_1$ in $\mathcal{D}_1$ must be disjoint from $\varphi(F)$;
no such $B_1$ can exist, otherwise $\psi(B_1)$ would be a bad bigon path in $\mathcal{D}_0$.
\item Consider case~\ref{case:bigonSeparatePaths:glueS2:internalIdeal} of Claim~\ref{claim:bigonSeparatePaths:glueS2}.
In this case, $F$ is incident to one internal vertex and one ideal vertex, which means that
the truncated bigon associated to $F$ forms a properly embedded disc $S$ in $X$.
Since $X$ is $\partial$-irreducible, the boundary curve of $S$ must bound a disc lying entirely in $\partial X$,
in which case Claim~\ref{claim:bigonSeparatePaths:glueS2} tells us that $\mathcal{P}_1$ is obtained from $\mathcal{P}_0$ by decomposing along a properly embedded $2$-sphere in $X$.
As before, by irreducibility of $X$, we conclude that flattening $F$
only changes the truncated pseudomanifold by creating a new $3$-sphere component.
Thus, $\mathcal{D}_1$ remains acceptable.
To see that $\mathcal{D}_1$ contains no bad bigon paths, we first note that $\varphi(F)$ is incident to exactly one ideal vertex $v$.
With this in mind, suppose for the sake of contradiction that $\mathcal{D}_1$ contains a bad bigon path $B_1$.
Observe that $B_1\cap\varphi(F)$ is either empty or consists only of the ideal vertex $v$,
which means that flattening $F$ essentially leaves $B_1$ untouched.
This implies that $\psi(B_1)$ is a bad bigon path in $\mathcal{D}_0$, contradicting the inductive hypothesis.
\end{itemize}

All that remains is to consider case~\ref{case:bigonSeparatePaths:glueS2:onlyIdeal} of Claim~\ref{claim:bigonSeparatePaths:glueS2}.
In this case, $F$ is only incident to ideal vertices, which means that
the truncated bigon associated to $F$ forms a properly embedded annulus $S$ in $X$;
let $\gamma_0$ and $\gamma_1$ denote the two boundary curves of $S$.
Observe that $S$ cannot be boundary-parallel, because this would mean that $F$ itself forms a bad bigon path in $\mathcal{D}_0$.
Combining this with the assumption that $X$ is anannular, we see that $S$ must be a compressible annulus.

Given what we know about the $3$-manifold $X$ and the annulus $S$,
we claim that flattening $F$ only changes the truncated pseudomanifold by creating a new $3$-sphere component.
To prove this, we start by compressing $S$ along an essential compression disc,
which yields two properly embedded discs $E_0$ and $E_1$ such that for each $i\in\{0,1\}$, the boundary curve of $E_i$ is $\gamma_i$.
Since $X$ is $\partial$-irreducible, each curve $\gamma_i$ must therefore bound a disc $E'_i$ lying entirely in the boundary of $X$;
see Figure~\ref{fig:benignCompressAnnulus}.
Thus, by case~\ref{case:bigonSeparatePaths:glueS2:onlyIdeal} of Claim~\ref{claim:bigonSeparatePaths:glueS2},
flattening $F$ corresponds to decomposing along a properly embedded $2$-sphere in $X$.
Since $X$ is irreducible, this $2$-sphere bounds a $3$-ball, so the only topological effect is to create a new $3$-sphere component.
This shows that $\mathcal{D}_1$ remains acceptable.

\begin{figure}[htbp]
\centering
	\begin{tikzpicture}

	\node[inner xsep=3pt] (before) at (0,0) {
		\includegraphics[scale=1]{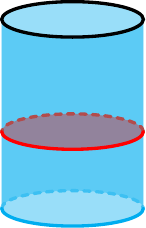}
	};
	\node[inner xsep=-6pt] (after) at (4,0) {
		\includegraphics[scale=1]{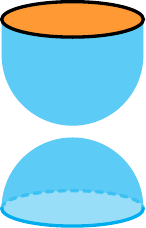}
	};
	\draw[thick, line cap=round, -Stealth] (before.east) -- (after.west);

	\node[cyan!50!black] at ($(before.west)+(0,0.5)$) {$S$};
	\node at ($(before.north)+(1,-0.1)$) {$\gamma_i$};
	\node[orange!50!black] at ($(after.north)+(0,-0.47)$) {$E'_i$};
	\node at ($(after.north)+(1,-0.1)$) {$\gamma_i$};
	\node[right, cyan!50!black] at ($(after.east)+(-0.05,0.5)$) {$E_i$};
	\node[right, cyan!50!black] at ($(after.east)+(-0.05,-1)$) {$E_{1-i}$};

	\end{tikzpicture}
\caption{For each $i\in\{0,1\}$, the curve $\gamma_i$ bounds a properly embedded disc $E_i$ in $X$ given by compressing the annulus $S$.
The fact that $X$ is $\partial$-irreducible therefore implies that $\gamma_i$ bounds a disc $E'_i$ that lies entirely in $\partial X$.}
\label{fig:benignCompressAnnulus}
\end{figure}
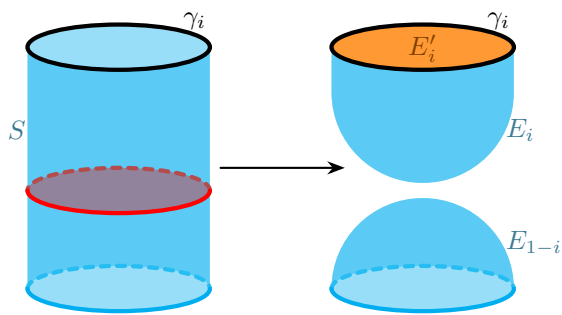

To finish, we just need to verify that $\mathcal{D}_1$ contains no bad bigon paths.
If a bad bigon path $B_1$ in $\mathcal{D}_1$ is disjoint from $\varphi(F)$ or meets $\varphi(F)$ only in (ideal) vertices,
then observe that $\psi(B_1)$ would be a bad bigon path in $\mathcal{D}_0$, which is impossible.
The only other possibility is that $B_1$ meets $\varphi(F)$ in an edge, in which case there would exist
an internal bigon path $B_0$ in $\mathcal{D}_0$ such that $B_0\cap F$ is a single edge and $B_0\cup F=\psi(B_1)$;
see Figure~\ref{fig:bigonPathInverseSphere}.
Observe that $B_0$ would be a bad bigon path in $\mathcal{D}_0$, which is again impossible.

\begin{figure}[htbp]
\centering
	\begin{tikzpicture}

	\node[inner sep=-15pt] (before) at (0,0) {
		\includegraphics[scale=1]{BigonPathFlat2.pdf}
	};
	\node[inner sep=-10pt] (after) at (3.8,0) {
		\includegraphics[scale=1]{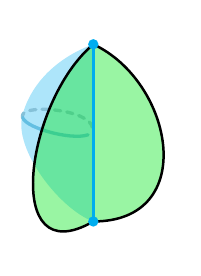}
	};
	\draw[thick, line cap=round, -Stealth] (before.east) -- node[above, inner sep=1pt]{$\psi$} (after.west);

	\end{tikzpicture}
\caption{If $F$ is a bigon face that forms a $2$-sphere, and if there is a bigon path $B_0$ such that $B_0\cap F$ is a single edge
and $B_0\cup F=\psi(B_1$), then $B_0$ is topologically equivalent to $B_1$.}
\label{fig:bigonPathInverseSphere}
\end{figure}

In summary, we have shown that in every possible case, performing an atomic move on $\mathcal{D}_0$ gives a new acceptable cell decomposition $\mathcal{D}_1$
(and also that $\mathcal{D}_1$ contains no bad bigon paths).
By induction, this shows that after performing however many atomic moves we need to flatten $\mathcal{D}^\ast$,
the triangulation $\mathcal{T}^\ast$ that results from flattening will be acceptable.
\end{proof}

\section{Triangulation complexity of \texorpdfstring{$3$}{3}-dimensional submanifolds} \label{sec:submanifold}

The purpose of this section is to showcase some applications of Theorem~\ref{thm:crushPositiveGenus} to a notion---namely,
\textbf{triangulation complexity}---whose significance is independent from crushing.
There has been substantial effort devoted to finding upper and lower bounds on triangulation complexity for various families of $3$-manifolds;
for instance, see~\cite{BFP2015,Jackson2023,JJST2016,JRST2019,JRT2009,JRST2020,Jaco-Rubinstein-Spreer-Tillmann:ComplexityDF,JRT2011,Jaco-Rubinstein-Tillmann:ThurstonNormComplexity,LackenbyPurcell2022Elliptic,Lackenby-Purcell:Complexity,RST2021Minimal,VesninFominykh2011}.
For this paper, the triangulation complexity for a closed $3$-manifold $\mathcal{M}$ will refer to
the minimum number of tetrahedra in any (closed) triangulation of $\mathcal{M}$,
and the triangulation complexity for a bounded $3$-manifold $\mathcal{M}$ will refer to
the minimum number of tetrahedra in any \emph{ideal} triangulation of $\mathcal{M}$;
in either case, we will denote this quantity by $\Delta(\mathcal{M})$.

Our main application of Theorem~\ref{thm:crushPositiveGenus} is to prove that, under quite general conditions,
the triangulation complexity of a $3$-manifold $\mathcal{M}$ is strictly bigger than the triangulation complexity of
a $3$-dimensional submanifold of $\mathcal{M}$ bounded by surfaces of positive genus.
The precise statement is given in Theorem~\ref{thm:submanifold} below.

To our knowledge, Theorem~\ref{thm:submanifold} has never previously been written down in the literature.
However, as mentioned in Section~\ref{sec:intro}, it is important to note that
a similar result can also be obtained by combining the following ideas from Matveev's book~\cite{Matveev2007}:
\begin{itemize}
\item duality of triangulations and special spines~\cite[Section~1.1]{Matveev2007};
\item conversion of almost simple spines into special spines~\cite[Section~2.1.1]{Matveev2007}; and
\item results about how the complexity of almost simple spines interacts with the operation of cutting along
normal surfaces in handle decompositions~\cite[Section~4.2]{Matveev2007}.
\end{itemize}

In any case, a good reason to explicitly write down Theorem~\ref{thm:submanifold} is that its assumptions are relatively easy to check.
This gives a way to streamline some applications by avoiding the need to directly use either our crushing machinery or Matveev's spine machinery.
We demonstrate this in Sections~\ref{subsec:jsj} to~\ref{subsec:rod} by giving some straightforward applications
of Theorem~\ref{thm:submanifold} to JSJ decompositions and satellite knots.

\submanifold*

\begin{proof}
Let $\mathcal{T}$ be a closed (if $\mathcal{M}$ is closed) or ideal (if $\mathcal{M}$ has boundary)
triangulation of $\mathcal{M}$ such that $\lvert\mathcal{T}\rvert = \Delta(\mathcal{M})$.
Our goal is to find an ideal triangulation of $\mathcal{R}$ with strictly fewer tetrahedra than $\mathcal{T}$.
We do this by constructing a suitable normal surface $S'$ in $\mathcal{T}$, and using
Theorem~\ref{thm:crushPositiveGenus} to ensure that crushing $S'$ yields the desired triangulation of $\mathcal{R}$.

In detail, let $N(S)$ denote a closed tubular neighbourhood of $S$ in $\mathcal{M}$.
Viewing $\mathcal{R}$ as the submanifold of $\mathcal{M}$ given by deleting the interior of $N(S)$,
let $S'$ be the union of all the components of $\partial N(S)$ that meet $\mathcal{R}$.
Since $S$ has no $2$-sphere or projective plane components, observe that $S'$ cannot have any $2$-sphere components.
Also, each component of $S'$ is two-sided, since it meets $N(S)$ on one side and $\mathcal{R}$ on the other side;
moreover, since $\mathcal{R}$ meets each component of $S'$ on exactly one side, we can take the interior of $\mathcal{R}$ to be
the \emph{chosen region} (as defined at the beginning of Section~\ref{sec:genus}) for $S'$.
This already establishes most of what we require to apply Theorem~\ref{thm:crushPositiveGenus};
what remains is to show that $S'$ is an essential surface in $\mathcal{M}$,
which will allow us to assume that $S'$ is a non-trivial normal surface with respect to $\mathcal{T}$.

Since $S'$ is a closed (but possibly disconnected) surface, showing that $S'$ is essential entails verifying that
every component of $S'$ is incompressible, and that at least one component of $S'$ is not boundary-parallel;
in fact, we will be able to show that every component of $S'$ is not boundary-parallel, which is stronger than we require.
To this end, consider any particular component $C'$ of $S'$.
Since $C'$ lies in the boundary of a closed tubular neighbourhood $N(C)$ of some component $C$ of $S$, we have the following two cases:
\begin{itemize}
\item If $C$ is two-sided, then $C'$ is isotopic to $C$, so the fact that $C'$ is incompressible and not boundary-parallel follows
from the assumption that these conditions are satisfied by every component of $S$.
\item If $C$ is one-sided, then $C'=\partial N(C)$.
To see that $C'$ is incompressible, consider a compression disc $D$ for $C'$ in $\mathcal{M}$;
we need to show that $D$ cannot be an \emph{essential} compression disc.
We have the following cases:
	\begin{itemize}
	\item If $D\subset\mathcal{R}$, then $D$ is a compression disc for $\mathcal{R}$.
	Since $\mathcal{R}$ is assumed to be $\partial$-irreducible, $D$ cannot be essential.
	\item If $D\subset N(C)$, then we can use a standard fundamental group argument.
	Note that $\pi_1(N(C)) \cong \pi_1(C)$, and that the double-covering map $p:\partial N(C) \to C$
	induces an injective homomorphism $p_\ast:\pi_1(\partial N(C)) \to \pi_1(C)$.
	Since $\partial D$ is homotopically trivial in $N(C)$, injectivity of $p_\ast$ tells us that
	$\partial D$ must also be homotopically trivial in $C' = \partial N(C)$.
	Thus, we again see that $D$ cannot be essential.
	\end{itemize}
The upshot is that $C'$ does not admit an essential compression disc, so it is incompressible.
To see that $C'$ is not boundary-parallel, suppose instead that this is false.
The isotopy of $C'$ into $\partial\mathcal{M}$ defines a product region $P$ in $\mathcal{M}-C'$.
Note that $C'$ meets two components of $\mathcal{M}-C'$:
the interior of $N(C)$, and the interior of $\mathcal{R}$.
The product region $P$ must coincide with $\mathcal{R}$.
However, this would contradict the assumption that $\mathcal{R}$ is anannular, so we conclude that $C'$ cannot be boundary-parallel.
\end{itemize}
As mentioned above, this suffices to show that $S'$ is a closed essential surface in $\mathcal{M}$.

By incompressibility of $S'$, we can use the normalisation procedure (recall
Section~\ref{subsec:barriers}) to ensure that $S'$ is normal with respect to $\mathcal{T}$.
Moreover, since links of ideal vertices of $\mathcal{T}$ correspond to boundary-parallel surfaces,
the fact that $S'$ is not boundary-parallel ensures that we have a \emph{non-trivial} normal surface in $\mathcal{T}$.
By Lemma~\ref{lem:maximal}, we may further assume that this normal surface $S'$ is maximal.

To recap, we are now in the setting laid out at the beginning of Section~\ref{sec:genus}:
we have a suitable normal surface $S'$, together with a suitable chosen region given by the interior of $\mathcal{R}$.
By assumption, we have that $\mathcal{R}$ is irreducible, $\partial$-irreducible and anannular, and also that $\mathcal{R}$ contains no two-sided properly embedded M\"{o}bius bands.
Thus, all the prerequisites for Theorem~\ref{thm:crushPositiveGenus} are satisfied, and applying this theorem tells us that after crushing $S'$,
one of the benign components (as defined in Section~\ref{sec:genus}) forms
an ideal triangulation $\mathcal{T}^\ast$ of $\mathcal{R}$.
Since $S'$ is a \emph{non-trivial} normal surface, we have $\lvert\mathcal{T}^\ast\rvert < \lvert\mathcal{T}\rvert$.
Hence $\Delta(\mathcal{R}) \leqslant \lvert\mathcal{T}^\ast\rvert < \lvert\mathcal{T}\rvert = \Delta(\mathcal{M})$, as required.
\end{proof}

\subsection{Application: hyperbolic JSJ pieces}\label{subsec:jsj}

Let $\mathcal{M}$ be an irreducible and $\partial$-irreducible $3$-manifold with no $2$-sphere boundary components.
Recall that by work of Jaco and Shalen~\cite{JacoShalen1978,Jaco-Shalen:SFSpacesIn3Mfds},
and independently by Johannson~\cite{Johannson:HmtpEquiOf3MfdsWithBdry},
there is a canonical collection $\{S_i\}$ of finitely many disjoint essential tori in $\mathcal{M}$ such that
each piece resulting from cutting along $\bigcup_i S_i$ is either atoroidal or Seifert fibred;
formal statements of this result can also be found in~\cite[Theorem~1.9]{Hatcher3Mfld} and~\cite[Theorem~8.23]{Purcell:HyperbolicKnotTheory}.
This collection of tori is called the \textbf{JSJ decomposition} (or the \textbf{torus decomposition}) of $\mathcal{M}$;
it is closely related to (but not exactly the same as) the decomposition along tori described by the Thurston-Perelman Geometrisation Theorem.
Theorem~\ref{thm:submanifold} almost immediately yields the following consequence for JSJ decompositions:

\begin{theorem} \label{thm:JSJDecomp}
Let $\mathcal{M}$ be an orientable $3$-manifold with no $2$-sphere boundary components.
If $\mathcal{M}$ is irreducible, $\partial$-irreducible and has non-empty JSJ decomposition $\{S_i\}$, then any hyperbolic component $\mathcal{H}$
that results from cutting $\mathcal{M}$ along $\bigcup_i S_i$ satisfies $\Delta(\mathcal{H}) < \Delta(\mathcal{M})$.
\end{theorem}

\begin{proof}
Suppose that after cutting along the tori in the JSJ decomposition of $\mathcal{M}$, (at least) one of the resulting pieces $\mathcal{H}$ is hyperbolic.
By Thurston's Hyperbolisation Theorem, $\mathcal{H}$ is irreducible, $\partial$-irreducible and anannular.
Moreover, orientability of $\mathcal{M}$ implies that $\mathcal{H}$ is orientable,
and hence that $\mathcal{H}$ contains no two-sided properly embedded M\"{o}bius bands.
Thus, by Theorem~\ref{thm:submanifold}, we have $\Delta(\mathcal{H}) < \Delta(\mathcal{M})$.
\end{proof}

\subsection{Application: satellite knots}\label{subsec:satellite}

Recall that the \textbf{exterior} of a knot or link $L$ in $S^3$ is the $3$-manifold obtained by deleting an open regular neighbourhood of $L$ from $S^3$.
We take the \textbf{triangulation complexity} of a link $L$, denoted $\Delta(L)$, to mean the triangulation complexity of the exterior of $L$.

Our goal now is to present an easy consequence of Theorem~\ref{thm:submanifold} concerning the triangulation complexity of satellite knots;
see Theorem~\ref{thm:satellite} below.
For this, we first review the definition of a satellite knot:

\begin{definitions}\label{defs:satellite}
Let $V^\ast$ denote the solid torus given by the exterior of an unknot $U^\ast$, let $C^\ast$ denote the core circle of $V^\ast$,
and let $e:V^\ast\to S^3$ be an embedding such that the image of $C^\ast$ under $e$ is a non-trivial knot $C$.
Consider a knot $K^\ast$ in the interior of $V^\ast$ such that:
\begin{itemize}[nosep]
\item $K^\ast$ is not isotopic (inside $V^\ast$) to $C^\ast$; and
\item every meridional disc of $V^\ast$ meets $K^\ast$ at least once.
\end{itemize}
The image of $K^\ast$ under $e$ is a non-trivial knot $K$ called a \textbf{satellite knot}.
We call the knot $C$ a \textbf{companion} of $K$, and we call the torus $e(\partial V^\ast)$ a \textbf{companion torus} of $K$.
We also call the link $K^\ast\cup U^\ast$ a \textbf{pattern} of $K$.
\end{definitions}

\begin{figure}[htbp]
\centering
	\begin{subfigure}[t]{0.3\textwidth}
	\centering
		\includegraphics[scale=0.8]{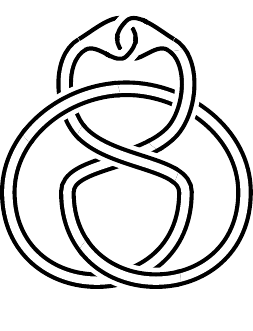}
	\caption{A satellite knot $K$: the untwisted Whitehead double of the figure-eight knot.}
	\label{subfig:whiteheadDouble}
	\end{subfigure}
	\hfill
	\begin{subfigure}[t]{0.3\textwidth}
	\centering
		\includegraphics[scale=0.8]{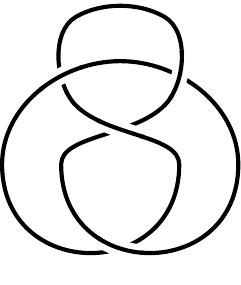}
	\caption{A companion of $K$: the figure-eight knot.}
	\label{subfig:companion}
	\end{subfigure}
	\hfill
	\begin{subfigure}[t]{0.3\textwidth}
	\centering
		\includegraphics[scale=0.8]{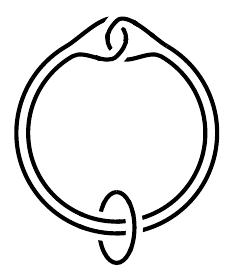}
	\caption{A pattern of $K$: the Whitehead link.}
	\label{subfig:pattern}
	\end{subfigure}
\caption{An example of a satellite knot, together with a companion and a pattern;
in this case, both the companion and the pattern are hyperbolic, and hence anannular.}
\label{fig:satellite}
\end{figure}

\begin{theorem}\label{thm:satellite}
Let $K$ be a satellite knot, and let $L$ denote either a companion or a pattern of $K$.
If the exterior of $L$ is anannular, then $\Delta(L) < \Delta(K)$.
\end{theorem}

It is worth noting that every hyperbolic link is anannular.
Thus, Theorem~\ref{thm:satellite} applies to a very large class of satellite knots;
Figure~\ref{fig:satellite} shows one example of such a satellite knot.

\begin{proof}[Proof of Theorem~\ref{thm:satellite}]
Let $\overline{K}$ and $\overline{L}$ denote the exteriors of $K$ and $L$, respectively.
Since $\overline{L}$ is one of the components given by cutting the exterior of $\overline{K}$ along a companion torus (which is an essential torus),
we can apply Theorem~\ref{thm:submanifold} provided that $\overline{L}$:
\begin{itemize}
\item is irreducible,
$\partial$-irreducible and anannular; and
\item contains no two-sided properly embedded M\"{o}bius bands.
\end{itemize}
We have assumed that $\overline{L}$ is anannular, and the fact that $\overline{L}$
is orientable implies that it contains no two-sided properly embedded M\"{o}bius bands.
Moreover, in the case where $L$ is a companion, irreducibility and $\partial$-irreducibility follow from the fact that $L$ must be a non-trivial knot;
on the other hand, when $L$ is a pattern, irreducibility and $\partial$-irreducibility follow from the fact that $L$ must be a non-split link.
The upshot is that, by Theorem~\ref{thm:submanifold},  we have
$\Delta(L) = \Delta\left(\overline{L}\right) < \Delta\left(\overline{K}\right) = \Delta(K)$.
\end{proof}

\begin{corollary}\label{cor:compositeKnot}
Consider the connected sum $K \ConnSum K'$ of two non-trivial knots $K$ and $K'$.
If $K$ is a hyperbolic knot, then $\Delta(K) < \Delta(K \ConnSum K')$.
\end{corollary}

\begin{proof}
Recall that the connected sum $K \ConnSum K'$ can be viewed as a satellite knot with companion given by either of its summands.
Thus, if the summand $K$ is hyperbolic, in which case the exterior of $K$ is anannular, then it follows immediately from Theorem~\ref{thm:satellite} that $\Delta(K) < \Delta(K \ConnSum K')$.
\end{proof}

\subsection{Application: rod complements in the \texorpdfstring{$3$}{3}-torus}\label{subsec:rod}

The crushing techniques developed in previous sections can also be used to study link exteriors in ambient spaces other than the $3$-sphere, such as the $3$-torus $\mathbb{T}^3$.
The fourth author and Purcell~\cite{Hui-Purcell:RodPackings} initiated the use of $3$-dimensional geometry and topology to study rod packing structures in crystallography.
In crystallographic chemistry, a rod packing is a packing of uniform cylinders (also called rods) that represent linear or zigzag chains of particles.
Readers may refer to \cite{OKeeffeEtAl:CubicRodPackings, OKeeffeEtAl:CylinderPackingsRevisited} for examples of rod packing structures.

Many rod packing structures exhibit translational symmetry in each dimension of $3$-dimensional Euclidean space.
Taking a quotient by this symmetry allows such rod packings to be encoded as geodesic links in the $3$-torus, which we can then study using tools from $3$-manifold geometry and topology.
Each component of such a geodesic link is called a \textbf{rod-shaped circle}, or often simply a \textbf{rod}, in $\mathbb{T}^3$;
the complement of such a link in $\mathbb{T}^3$ is a $3$-manifold called a \textbf{rod complement}.

For $\mathcal{M}$ belonging to a large family of rod complements, the JSJ decomposition gives a unique hyperbolic piece $\mathcal{H}$, and this hyperbolic piece is also a rod complement.
Theorem~\ref{thm:3-torus} states this precisely, and then (by applying Theorem~\ref{thm:JSJDecomp}) relates the triangulation complexities of $\mathcal{M}$ and $\mathcal{H}$.

To study the JSJ decompositions of rod complements, we rely on work of the fourth author~\cite{Hui:ClassifyT3Rods},
which completely classified the geometry of all rod complements using simple linear algebra conditions.
The linear algebra arises from the fact that a rod in $\mathbb{T}^3$ lifts to a straight line in the universal cover $\mathbb{R}^3$.
This linear structure leads to the following natural concepts:
\begin{itemize}
\item
\textbf{linear independence} of rods in $\mathbb{T}^3$; and
\item
\textbf{linear isotopy} of rods in $\mathbb{T}^3$, which roughly means an isotopy along a planar annulus between two parallel rods.
\end{itemize}
These notions appear in the statements of Lemma~\ref{lem:3-torus} and Theorem~\ref{thm:3-torus} below;
readers may refer to~\cite{Hui-Purcell:RodPackings, Hui:ClassifyT3Rods} for
definitions of these notions, as well as for explanations of other related terminology that we use in the proofs.

\begin{lemma}\label{lem:3-torus}
Let $\mathcal{M}$ be a toroidal rod complement in the $3$-torus with at least three linearly independent rods.
Any essential torus in $\mathcal{M}$ bounds a solid torus in $\mathbb{T}^3$ whose interior contains two or more linearly isotopic rods.
\end{lemma}

\begin{proof}
Since $\mathcal{M}$ is toroidal, by Proposition~3.8 in~\cite{Hui:ClassifyT3Rods}, either $\mathcal{M}$ is a rod complement with all rods spanning
a plane torus, or there exist disjoint parallel rods that are linearly isotopic in the complement of the other rods (or possibly both).
The assumption that $\mathcal{M}$ is a rod complement with three linearly independent rods
thus implies the existence of at least two linearly isotopic parallel rods.

Let $T_\mathrm{e}$ be an essential torus in $\mathcal{M}$.
Note that $T_\mathrm{e}$ is not in the homotopy class of a plane torus because there exist three linearly independent rods for $\mathcal{M}$.
Hence, the torus $T_\mathrm{e}$, essential in $\mathcal{M}$, has at least one generator in $\pi_1(T_\mathrm{e}) \cong \mathbb{Z} \times \mathbb{Z}$
represented by an essential loop that is homotopically trivial in $\mathbb{T}^3$.

By Lemma~3.7 in~\cite{Hui:ClassifyT3Rods}, there exists a compression disc $D$ for $T_\mathrm{e}$ in $\mathbb{T}^3$.
It then follows from the irreducibility of $\mathbb{T}^3$ that $T_\mathrm{e}$ is separating.
As $T_\mathrm{e}$ is incompressible in $\mathcal{M}$, some rod must intersect $D$.
Hence, the other generator of $\pi_1(T_\mathrm{e}) \cong \mathbb{Z} \times \mathbb{Z}$
is represented by a loop that is homotopically non-trivial in $\mathbb{T}^3$.
By Lemma~3.9 in \cite{Hui:ClassifyT3Rods}, $T_\mathrm{e}$ bounds a solid torus $V_\mathrm{e}$ in $\mathbb{T}^3$.
Since $T_\mathrm{e}$ is not boundary-parallel, the interior of $V_\mathrm{e}$ contains
two or more parallel rods that are linearly isotopic in the complement of the other rods.
\end{proof}

\begin{theorem}\label{thm:3-torus}
Let $\mathcal{M}$ be a toroidal rod complement in the $3$-torus with at least three linearly independent rods.
The JSJ decomposition of $\mathcal{M}$ gives a unique (up to homeomorphism) hyperbolic
rod complement $\mathcal{H}$ with $\Delta(\mathcal{H}) < \Delta(\mathcal{M})$.
\end{theorem}

\begin{proof}
Without loss of generality, assume that $\mathcal{M}$ is the complement of an open neighbourhood of all the rods in the $3$-torus.

We first show that $\mathcal{M}$ satisfies all the assumptions in Theorem~\ref{thm:JSJDecomp}.
Note that $\mathcal{M}$ is an orientable $3$-manifold with no $2$-sphere boundary components.
Since $\mathcal{M}$ is the complement of finitely many rods in the $3$-torus, Proposition~3.6
in~\cite{Hui:ClassifyT3Rods} tells us that $\mathcal{M}$ is irreducible and $\partial$-irreducible.
We also know that $\mathcal{M}$ has non-empty JSJ decomposition because it is toroidal by assumption, and because it is not Seifert fibred by Theorem~4.1 in~\cite{Hui:ClassifyT3Rods}.
Thus, Theorem~\ref{thm:JSJDecomp} applies to any hyperbolic component that results from cutting along the JSJ tori.

Let $\{T_i\}$ be the set of essential tori that gives the JSJ decomposition of $\mathcal{M}$;
recall that this set of tori is unique up to isotopy, which means that the JSJ pieces
(the $3$-manifold components obtained by cutting along these tori) are unique up to homeomorphism.
By Lemma~\ref{lem:3-torus}, each essential torus $T_i$ bounds a solid torus $V_i$ in $\mathbb{T}^3$ whose interior contains two or more linearly isotopic rods.
Thus, one of the JSJ pieces is a $3$-manifold $\mathcal{H} := \mathcal{M} - \bigcup_i\Int(V_i)$ whose interior is homeomorphic to a rod complement.
We will show that $\mathcal{H}$ is the unique hyperbolic JSJ piece for $\mathcal{M}$.

To do this, we repeatedly appeal to the classification given by Theorem~4.1 in~\cite{Hui:ClassifyT3Rods}.
Since $\mathcal{M}$ has at least three linearly independent rods, observe that the same must be true for $\mathcal{H}$.
Thus, by the classification, $\mathcal{H}$ cannot be Seifert fibred.
This means that $\mathcal{H}$ must be atoroidal, since it is one of the JSJ pieces.
Using the classification again, we therefore see that $\mathcal{H}$ cannot have a pair of linearly isotopic rods.
This has two implications:
\begin{itemize}
\item
First, we must have exactly one essential torus $T_i$ for each linear isotopy class containing at least two rods from $\mathcal{M}$,
and the corresponding solid torus $V_i$ in $\mathbb{T}^3$ must contain \emph{all} of the rods in this class.
\item
Second, by the classification, the interior of $\mathcal{H}$ admits a complete hyperbolic structure.
\end{itemize}
To see that $\mathcal{H}$ is the only hyperbolic JSJ piece, observe that each of the other JSJ pieces is
obtained from one of the solid tori $V_i$ by deleting a small neighbourhood of all the (linearly isotopic) rods inside $V_i$;
in other words, all the other JSJ pieces are Seifert fibred, since they are homeomorphic to solid tori with at least two core curves removed.

Finally, since $\mathcal{H}$ is a hyperbolic piece obtained after cutting along the JSJ tori for $\mathcal{M}$,
it follows from Theorem~\ref{thm:JSJDecomp} that $\Delta(\mathcal{H}) < \Delta(\mathcal{M})$.
\end{proof}

\bibliography{crush}

\end{document}